%% file: riemannhilbert.tex
\documentclass[twoside, a4paper]{amsart}				%---- „amsart“ for aritcles and reports

\usepackage{amsmath,amssymb,amsfonts,amsthm}
\usepackage[utf8]{inputenc}				%---- Package for umlaute
\usepackage{tikz-cd}					%---- Package to draw commutative diagrams 
\usetikzlibrary{arrows,babel}		%---- Option for above package
\tikzcdset{arrow style=tikz, diagrams={>=stealth}}	%---- Change tikzcd arrow style
\usepackage{enumerate}					%---- Package to enable enumerate
\usepackage{lipsum}
\usepackage{imakeidx}
\usepackage{hyperref}					%---- Package to make index
\usepackage{titletoc}
\usepackage{xcolor}
\hypersetup{
    colorlinks,
    linkcolor={red!50!black},
    citecolor={blue!50!black},
    urlcolor={blue!80!black}}

\makeatletter
\newcommand\cdotfill{%
    \leavevmode\cleaders\hb@xt@.44em{\hss$\cdot$\hss}\hfill\kern\z@
}
\makeatother

\makeindex[options= -s example.ist]

\theoremstyle{plain}
\newtheorem{thm}{Theorem}[section]
\newtheorem{prop}[thm]{Proposition}
\newtheorem{cor}[thm]{Corollary}
\newtheorem{lem}[thm]{Lemma}
\newtheorem*{claim}{Claim}
\newtheorem*{thm*}{Theorem}
\newtheorem*{openThm*}{Conjecture}
\newtheorem*{cor*}{Corollary}

\theoremstyle{definition}
\newtheorem{defn}[thm]{Definition}
\newtheorem{exmp}[thm]{Example}

\theoremstyle{remark}
\newtheorem{rem}[thm]{Remark}

\newtheorem{con}[thm]{Construction}

\input{sh.tex}

\mathchardef\ordinarycolon\mathcode`\:
\mathcode`\:=\string"8000
\begingroup \catcode`\:=\active
  \gdef:{\mathrel{\mathop\ordinarycolon}}
\endgroup

\title[Analytic Relative Riemann-Hilbert theorem and the $\bar{\partial}$-equation]{Relative Riemann-Hilbert \& Newlander-Nirenberg theorems for torsion-free analytic sheaves on maximal and homogeneous spaces}
\author{Thomas Kurbach}

\begin{document}

\maketitle

\section*{Abstract}
In this paper it is shown that for locally trivial complex analytic morphisms between some reduced spaces the Relative Riemann-Hilbert Theorem still holds up to torsion, i.e. tame flat relative connections on torsion-free sheaves are in 1-to-1 correspondence with torsion-free relative local systems. Subsequently, it is shown that generalized $\bar{\partial}$-operators on real analytic sheaves over complex analytic spaces can be viewed as relative complex analytic connections on the complexification of the underlying real analytic space with respect to a canonical morphism. By means of complexification, the Relative Riemann-Hilbert Theorem then yields a Newlander-Nirenberg type theorem for $\bar{\partial}$-operators on torsion-free real analytic sheaves over some complex analytic varieties. In the non-relative case, this result shows that on all maximal and homogeneous analytic spaces tame flat analytic connections are in 1-to-1 correspondence with local systems, which in turn are in 1-to-1 correspondence with linear representations of the fundamental group (assuming connectedness).

\tableofcontents

\section{Introduction}

Flat connections on manifolds are interesting differential operators that surprisingly only encode topological information, i.e. they are entirely determined by their induced representation of the fundamental group and the underlying topological space. Relative connections encode nicely varying families of connections on the fibers of a morphism. Flat complex analytic relative connections on locally trivial morphisms with smooth fibers were shown to be entirely determined by their kernel by P. Deligne in~\cite{deligne}. In the following, this result is extended to the case of flat complex analytic relative connections on locally trivial morphisms with potentially singular fibers. However, some concessions need to be made: The spaces involved are assumed to be maximal or homogeneous and the connections need to be \emph{tame} (Definition~\ref{tamedefinition}).\\
The methods developed here culminate in two Relative Riemann-Hilbert Theorems about complex analytic flat relative connections:\\
\begin{thm*}[\ref{MaxRH} and \ref{HomCor}]
Let $f\colon X\to N$ be a reduced locally trivial morphism of complex analytic spaces with maximal fibers or let $f\colon X=N\times M\to N$ be the projection with $M\subseteq \bb{C}^n$ homogeneous. Then there is a 1-to-1 correspondence between
\begin{enumerate}[(i)]
\item pairs $\left( \mc{F},\nabla \right)$ of torsion-free $\coan{X}$-coherent modules with tame flat $f$-relative connections and
\item torsion-free relative local systems $V$.
\end{enumerate}
The correspondence takes pairs $\left(\mc{F},\nabla\right)$ and sends them to $\ker{\nabla}$ and conversely takes a relative local system $V$ and sends it to $\left(V\otimes_{f^{-1}\coan{N}} \coan{X},\id\otimes d_f\right)$.
\end{thm*}

For an outline of the proof strategy employed in this paper see subsection~\ref{Strat}, where the philosophy and guiding ideas behind the proof are presented.\\
The assumption of the connections being \emph{tame} is very natural and indeed necessary, as the connections associated to a relative local system are always tame over submersions (see Proposition~\ref{PullConTor}). An example of a non-tame flat connection $\nabla$ is given in section~\ref{NonTame}, where the connection is such that $\nabla\left( s \right)$ has non-empty support and the support is entirely contained in the singular set of the underlying space. Therefore a flat connection on a singular complex analytic space is not necessarily tame.\\
The presented form of a Relative Riemann-Hilbert Theorem can then be applied to another situation: the case of integrable pseudo-holomorphic structures (Section~\ref{trans}). In this case it is highly important that the \emph{complexification} (Sections~\ref{two} and~\ref{compelxificationsofcomplexanalyticspaces}) and the notion of \emph{differential forms} (Sections~\ref{four} and~\ref{pq}) are firmly controlled. Therefore, both topics are discussed at great length in the beginning of this paper.\\
To connect Riemann-Hilbert type Theorems to pseudo-holomorphic structures one needs to recognize that the sheaf $\Omega^{0,1}M$ of $\left( 0,1 \right)$-forms on a complex analytic space $M$ can be seen as a sheaf of relative differential forms on the complexification of $M$. To this end, one constructs the relevant morphism $\Phi\colon M^{\bb{C}}\to M$ in Theorem~\ref{canonicalfibration}. The existence of such a locally trivial morphism is not surprising because locally the complexification of $M$ is isomorphic to $M\times \bar{M}$ and the constructed global morphism is locally equivalent to the projection to the first factor. It then turns out that the relative forms with respect to $\Phi$ yield the $\left( 0,1 \right)$-forms, i.e.
\[\rst{\Omega^1_{\Phi}\left( \coan{M^{\bb{C}}} \right)}{M}\cong \Omega^{0,1}M.\]
Again, this is not surprising in the local picture as the $\left( 0,1 \right)$-forms are essentially the holomorphic forms on $\bar{M}$.\\
With this idea in hand, the final proof of the following real analytic Newlander-Nirenberg type Theorem follows elegantly from Theorem~\ref{MaxRH}:
\begin{thm*}[\ref{holomorpseudo}]
Let $M$ be a maximal complex analytic space and $\left( \mc{F},\bar{\partial}_{\mc{F}} \right)$ a torsion-free coherent $\corean{M}$-module with a tame integrable pseudo-holomorphic structure. Then $\mc{G}:=\ker{\bar{\partial}_{\mc{F}}}$ is a torsion-free coherent $\coan{M}$-module and $\mc{F}=\mc{G}\otimes_{\coan{M}}\corean{M}$.\\
In particular, there is a 1-to-1 correspondence between torsion-free coherent $\coan{M}$-modules and torsion-free coherent $\corean{M}$-modules equipped with a tame integrable pseudo-holomorphic structure.
\end{thm*}
As a final application, in the non-relative case, the theorems show that tame complex analytic flat connections on coherent sheaves over maximal or homogeneous complex analytic spaces are equivalent to local systems. Local systems in turn are equivalent to representations of the fundamental group (see e.g.~\cite{deligne}). Therefore, tame complex analytic flat connections on maximal or homogeneous complex analytic spaces are entirely determined by their parallel transport around homotopy classes of loops. In particular:
\begin{cor*}
Tame complex analytic flat connections on simply connected maximal or homogeneous complex analytic spaces are trivial.
\end{cor*}
The material in this paper is organized as follows: Section~\ref{two} introduces the sheaf of complex-valued and real-valued real analytic functions on an analytic space and shows that the associated functors are nicely behaved. The sheaf of complex-valued real analytic functions is then used to define the complexification of a real analytic space. Section~\ref{compelxificationsofcomplexanalyticspaces} discusses the complexification of a complex analytic space in detail and constructs the locally trivial fibration alluded to earlier. The analytic differential forms on a singular analytic space are constructed as universal objects in Section~\ref{four} and some basics on derivation on analytic spaces are proven. On a complex analytic space, one can split the complex-valued real analytic forms into the $(1,0)$- and $(0,1)$-forms, just as one expects from the manifold case. The construction of this splitting and recognizing the $(0,1)$-forms as relative differential forms on the complexification is presented in Section~\ref{pq}. The classical Relative Riemann-Hilbert Theorem on submersions is recalled in Section~\ref{RelSub}. Section~\ref{interlude} discusses aspects of torsion-free sheaves and introduces the notion of tame connections. For torsion-free sheaves on reduced locally trivial morphisms with maximal and homogeneous fibers the Relative Riemann-Hilbert Theorem is proven in Section~\ref{SingFib}. This section is split into multiple subsections. Subsection~\ref{WeakHoloMaxPropMod} introduces and discusses the basic concept of weakly holomorphic functions and the definition of maximal spaces is given (Definition~\ref{Normal}). The notions suffice to give an outline of the strategy and philosophy guiding the proof presented in this paper. This outline is contained in subsection~\ref{Strat}. Subsection~\ref{weakSol} shows that in the case of torsion-free sheaves weak solutions always exist. Then one can show that these weak solutions are holomorphic if and only if they are holomorphic along the fibers. The argument for this reduction to the ``absolute'' case is given in subsection~\ref{RedAbs} (Theorem~\ref{RelToAbs}) and the maximal fiber case follows immediately from these considerations (Theorem~\ref{MaxRH}). One can even further reduce the absolute case to the case of complex analytic curves (Theorem~\ref{RedToCurveGen}), which can be observed in subsection~\ref{RedToCurve}. From the methods developed there one also obtains the homogeneous fiber case (Theorem~\ref{HomCor}). Subsequently, a real analytic Newlander-Nirenberg type Theorem~\ref{holomorpseudo} is obtained by complexification in Section~\ref{trans}. Further research questions are then gathered in Section~\ref{Conc}. 

\subsection{Acknowledgements}

This manuscript represents the authors Ph.D.-Thesis which was completed at the University of Wuppertal in 2025 under the supervision of Prof. Dr. Jean Ruppenthal, Prof. Dr. Roger Bielawski, Prof. Dr. Stefan Nemirovski, Prof. Dr. Kay Rülling and Dr. Maxim Kukol.\\
The author whishes to thank the committee overseeing his Pd.D.-defense and also the Komplexe Analysis group at the University of Wuppertal for their support.\\
The author was partially supported by the ANR-DFG project ‘QuasiDy – Quantization, Singularities, and Holomorphic Dynamics’ (Project-ID 490843120).

\section{Complexification of real analytic spaces}
\label{two}

Before beginning the discussion a few notations are introduced. Throughout the entire paper the symbols $\coan{\cdot}$, $\corean{\cdot}$ and $\rean{\cdot}$ denote respectively the sheaf of holomorphic, complex-valued real analytic and real-valued real analytic functions. Moreover, $\bb{K}$ denotes either the field  $\bb{R}$ or $\bb{C}$. A $\bb{K}$-ringed space is a locally ringed space, such that the sheaf of rings is also a $\bb{K}$-algebra and the residue field at each point is isomorphic to $\bb{K}$. Whenever $\phi\colon \left( M,\mc{A} \right)\to \left( N,\mc{B} \right)$ is a morphism of ringed spaces, the topological component of $\phi$ is denoted by $\phi$ as well and the sheaf component is denoted by $\tilde{\phi}$. The canonical morphism on germs of ringed spaces at a point $p$ induced by a morphism of ringed spaces, is denoted by $\tilde{\phi}_p$. Monomorphisms and epimorphisms of $\bb{K}$-ringed spaces are defined by their usual cancellation properties.\\
It will be important that one has a firm grasp on complex analytic, real analytic and complex-valued real analytic functions on a complex analytic space. This has the advantage that the real analytic case essentially follows as a corollary of the complex analytic methods. Also, real analytic generalised $\bar{\partial}$-operators naturally act on sheaves of modules over the complex-valued real analytic functions.

\begin{con}
Let $\left( M,\coan{M} \right)$ be a complex analytic space. First, assume that $M$ is a closed subspace of an open subset $U\subseteq \bb{C}^n$ such that the defining ideal $J\subseteq \coan{U}$ is generated by $f_{1},\dots,f_k$ on $U$. One obtains the finitely generated ideal sheaves
\[J^{\bb{C}}:=J+\bar{J} \subseteq \corean{U}\; \text{and} \;J^{\bb{R}}:=\left( \mathrm{Re}\left( f_1 \right),\mathrm{Im}\left( f_1 \right),\dots,\mathrm{Re}\left( f_k \right), \mathrm{Im}\left( f_k \right) \right)\subseteq \rean{U}.\] Moreover, set
\[\corean{M}:=\iota^{-1}\left( \quo{\corean{U}}{J^{\bb{C}}} \right) \; \text{and} \; \rean{M}:=\iota^{-1}\left( \quo{\rean{U}}{J^{\bb{R}}} \right),\]
where $\iota\colon M\to U$ is the inclusion. It is clear that one obtains morphisms $\rean{M}\to \corean{M}$ and $\coan{M}\to \corean{M}$ by taking the quotient of the inclusion morphism. The first morphism is injective.

\begin{claim}
The morphism $\psi\colon \coan{M}\to \corean{M}$ is injective.
\end{claim}

\begin{proof}
One can check the claim germ-wise and for that, one may assume that $p=0\in M \subseteq U$. Let $\left[ g_p \right]\in \coan{M,p}$ be such that $\psi_p\left( \left[ g_p \right] \right)=0$. This is the same as saying

\[g= \sum_{i=1}^k f_i a_i + \sum_{i=1}^k \bar{f}_i b_i\]

with $a_i,b_i \in \corean{U}\left( U' \right)$, for some $U'\subseteq U$. Denote by $a_i^{h}$ the holomorphic part of $a_i$, for this note that the $a_i$ are analytic functions in $z_i$ and $\bar{z}_i$ and
\[a_i^h\left( z_1,\dots,z_n \right)=a_i\left( z_1,\dots,z_n,0,\dots,0 \right).\]
Set $a_i^r:=a_i - a_i^h$. Then, the left-hand and the right-hand side of the following equation need to vanish independently:

\[g - \sum_{i=1}^k f_i a_i^h = \sum_{i=1} f_i a_i^r + \sum_{i=1}^k \bar{f}_i b_i.\]

This is because the left-hand side only has terms containing powers of $z_i$ and products thereof and the right-hand side only contains terms proportional to some non-zero power of $\bar{z}_j$. Hence, 

\[g = \sum_{i=1}^k f_i a_i^h\]

and thus $g_p \in J_p$ and $\left[ g_p \right]=0\in \coan{M,p}$.
\end{proof}

Utilizing these morphisms, one obtains a morphism of $\bb{R}$-ringed spaces $\left( M,\corean{M} \right)\to \left( M,\rean{M} \right)$ and a morphism of $\bb{C}$-ringed spaces $\left( M,\corean{M} \right)\to \left( M,\coan{M} \right)$. These spaces are the \emph{associated analytic spaces} to a complex analytic space $M$. Note that the conjugation of complex-valued real analytic functions descends to $\corean{M}$.\\
Suppose $N\subseteq V\subseteq \bb{C}^l$ is another complex analytic space in a local model with $V\subseteq \bb{C}^l$ open and $I\subseteq \coan{V}$ the defining ideal. Suppose $\phi\colon M\to N$ is a holomorphic map.

\begin{claim}
%There exist an analytic morphism $\phi^{\bb{R}}\colon \left( M,\rean{M} \right)\to \left( N,\rean{N} \right)$
There exists a unique morphism of $\bb{C}$-ringed spaces
\[\phi^{\bb{C}}\colon \left( M,\corean{M} \right)\to \left( N,\corean{N} \right)\]
such that the following diagram commutes:
\[
\begin{tikzcd}
\left( M,\corean{M} \right) \arrow[r,"\phi^{\bb{C}}"] \arrow[d] & \left( N,\corean{N} \right) \arrow[d]\\
\left( M,\coan{M} \right) \arrow[r,"\phi"] & \left( N,\coan{N} \right)
%\left( M,\rean{M} \right) \arrow[r,"\phi^{\bb{R}}"] \arrow[d] & \left( N,\rean{N} \right) \arrow[d]\\
%\left( M,\corean{M} \right) \arrow[r,"\phi^{\bb{C}}"] & \left( N,\corean{N} \right)
\end{tikzcd}.
\]
Moreover, the sheaf component of $\phi^{\bb{C}}$ respects the complex conjugation. If $\phi$ is an isomorphism, so is $\phi^{\bb{C}}$.
\end{claim}

\begin{proof}
Let $p\in M$. There exist open neighbourhoods $M_p\subseteq M$ and $U_p\subseteq U$ around $p$ and a holomorphic morphism $\psi_p\colon U_p \to V$ such that the following diagram commutes:
\[
\begin{tikzcd}
U_p \arrow[r,"\psi_p"] & V \\
M_p \arrow{u} \arrow[r,"\phi"] & N \arrow[u]
\end{tikzcd}.
\]
Since $\psi_p$ is a holomorphic morphism of manifolds, there exists a natural morphism $\psi^{\bb{C}}_p\colon \left( U_p, \corean{U_p} \right) \to \left( V,\corean{V} \right)$. It needs to be shown that
\[\left( \psi^{\bb{C}}_p \right)^{\sim}\left( I^{\bb{C}} \right)\subseteq J^{\bb{C}}.\]
This is true because $\left( \psi^{\bb{C}}_p \right)^{\sim}$ maps $I$ into $J$ and commutes with conjugation. Thereby, one obtains a morphism
\[\phi^{\bb{C}}_p\colon \left( M_p,\corean{M_p} \right) \to \left( N,\corean{N} \right),\]
for every $p\in M$ and these morphisms satisfy the necessary commutative diagram by definition because each $\psi_p$ does. To see this more directly, consider the following diagram:
\[
\begin{tikzcd}
\left( U_p,\corean{U_p} \right) \arrow[d,"a"] \arrow[rrr,"\psi^{\bb{C}}_p",bend left =15]& \arrow[l,hook',"i"] \left( M_p,\corean{M_p} \right) \arrow[d,"b"] \arrow[r,"\phi_p^{\bb{C}}",swap]  & \left( N,\corean{N} \right) \arrow[r,hook,"j",swap] \arrow[d,"c"] & \left( V,\corean{V} \right) \arrow[d,"d"]\\
\left( U_p,\coan{U_p} \right) \arrow[rrr,"\psi_p",bend right =15,swap]& \arrow[l,hook',"k",swap] \left( M_p,\coan{M_p} \right) \arrow[r,"\phi"] & \left( N,\coan{N} \right) \arrow[r,hook,"m"] & \left( V,\coan{V} \right)
\end{tikzcd}.
\]
The squares on the left and right are clearly commuting and moreover,
\[\psi_p\circ a = d\circ \psi^{\bb{C}}_p,\; j\circ \phi_p^{\bb{C}} = \psi^{\bb{C}}_p\circ i \text{ and } m\circ \phi= \psi_p\circ k.\]
Observe that
\[m\circ c \circ \phi^{\bb{C}}_p= d\circ j\circ \phi^{\bb{C}}_p=d\circ \psi^{\bb{C}}_p\circ i = \psi_p\circ a \circ i= \psi_p\circ k\circ b = m\circ \phi \circ b.\]
The morphism $m$ is a monomorphism of $\bb{C}$-ringed spaces and thus $c\circ \phi^{\bb{C}}_p= \phi\circ b$.\\
Now the question is: Do these morphisms $\phi_p^{\bb{C}}$ glue to a morphism on $M$? The topological components of the $\phi_p^{\bb{C}}$ are all the same, i.e. simply the restriction of $\phi$ and the sheaf components respect the complex conjugation. However, on $\coan{N,p}$ the canonical morphisms of germs are all the same, hence, they are the same on $\coan{N,p}$ and $\bar{\mc{O}}_{N,p}$ and thus on $\corean{N.p}$. Therefore, the morphisms glue to $\phi^{\bb{C}}\colon \left( M,\corean{M} \right)\to \left( N,\corean{N} \right)$. This argument also shows that $\phi^{\bb{C}} $ is unique.\\
If $\phi$ is an isomorphism, then the same construction applied to $\phi^{-1}$ yields an inverse map to $\phi^{\bb{C}}$.
\end{proof}

The preceding claim shows, in particular, that the sheaf $\corean{M}$ is independent of the local model as one obtains a natural isomorphism for every choice of local model.

\begin{claim}
Let $\phi\colon M\to N$, $\psi\colon N\to P$ be holomorphic morphisms between local models of complex analytic spaces. Then, $\left( \psi\circ \phi \right)^{\bb{C}}=\psi^{\bb{C}} \circ \phi^{\bb{C}}$.
\end{claim}

\begin{proof}
This follows from the uniqueness of the morphism.
\end{proof}

Suppose that $\left( M,\coan{M} \right)$ is a complex analytic space. One may assume that it is given by some gluing system of local models of complex analytic spaces. The preceding claim shows that one naturally obtains a gluing system of the locally defined sheaves of $\bb{C}$-algebras $\corean{U}$. Thus, one also gets a $\bb{C}$-ringed space $\left( M,\corean{M} \right)$ that is locally defined by the ideal $J+\bar{J}$.

\begin{claim}
Let $\left( M,\coan{M} \right)$ be a complex analytic space. Then the locally defined inclusion maps $\coan{U}\to \corean{U}$ are independent of the local model and thus glue to a global morphism $\coan{M}\to \corean{M}$.
\end{claim}

\begin{proof}
Let $p\in M$ and let $\left( W,\coan{W} \right)\to\left( V_1,\coan{V_1} \right)$ and $\left( W,\coan{W} \right)\to \left( V_2,\coan{V_2} \right)$ be two local models around $p$. Then there exists a commutative diagram
\[
\begin{tikzcd}
\left( V_1,\coan{V_1} \right) \arrow[rr,"\phi"]  & & \left( V_2,\coan{V_2} \right)\\
& \left( W,\coan{W} \right) \arrow[lu,"\iota_1"] \arrow[ru,"\iota_2"] & 
\end{tikzcd},
\]
after potentially shrinking the local model, and then there exists the associated diagram of associated spaces
\[
\begin{tikzcd}
\left( V_1,\corean{V_1} \right) \arrow[rr,"\phi^{\bb{C}}"]  & & \left( V_2,\corean{V_2} \right)\\
& \left( W,\corean{W} \right) \arrow[lu,"\iota_1^{\bb{C}}"] \arrow[ru,"\iota_2^{\bb{C}}",swap] & 
\end{tikzcd}.
\]
One obtains a big diagram
\[
\begin{tikzcd}
\left( V_1,\corean{V_1} \right) \arrow[rrrr,"\phi^{\bb{C}}"] \arrow[rd,"a"] & & & & \left( V_2,\corean{V_2} \right) \arrow[ld,"b",swap]\\
 & \left( V_1,\coan{V_1} \right) \arrow[rr,"\phi"] & & \left( V_2,\coan{V_2} \right) & \\
 & & \left( W , \coan{W} \right) \arrow[ul,"\iota_1"] \arrow[ur,"\iota_2",swap] & & \\
 & & \left( W,\corean{W} \right) \arrow[u,"B",swap,bend right=30] \arrow[u,"A",bend left=30] \arrow[uuurr,"\iota_2^{\bb{C}}",bend right=10,swap] \arrow[uuull,"\iota_1^{\bb{C}}",bend left =10] & &
\end{tikzcd},
\]
and the question is whether $A=B$. Observe the following calculation
\[\iota_2\circ B = b\circ \iota_2^{\bb{C}}=b\circ \phi^{\bb{C}} \circ \iota_1^{\bb{C}} = \phi\circ a \circ \iota_1^{\bb{C}} = \phi\circ \iota_1\circ A = \iota_2\circ A.\]
Since $\iota_2$ is a monomorphism of $\bb{C}$-ringed spaces it follows that $B=A$. 
\end{proof}

\begin{claim}
Let $\left( M,\coan{M} \right)$ be a $\bb{C}$-analytic space. Then the locally defined conjugation maps on $\corean{M}$ are independent of the local model and thus glue to a global morphism of sheaves of rings.
\label{inclusionaninhol}
\end{claim}

\begin{proof}
Let $p\in M$ and let $\left( W,\coan{W} \right)\to\left( V_1,\coan{V_1} \right)$ and $\left( W,\coan{W} \right)\to \left( V_2,\coan{V_2} \right)$ be two local models around $p$. Then there exists a commutative diagram
\[
\begin{tikzcd}
\left( V_1,\coan{V_1} \right) \arrow[rr,"\Phi"]  & & \left( V_2,\coan{V_2} \right)\\
& \left( W,\coan{W} \right) \arrow[lu] \arrow[ru] & 
\end{tikzcd},
\]
after potentially shrinking the local model, and then there exists the associated diagram of associated spaces
\[
\begin{tikzcd}
\left( V_1,\corean{V_1} \right) \arrow[rr,"\Phi^{\bb{C}}"]  & & \left( V_2,\corean{V_2} \right)\\
& \left( W,\corean{W} \right) \arrow[lu,"\iota_1"] \arrow[ru,"\iota_2",swap] & 
\end{tikzcd}.
\]
Let $s\in \corean{W}$ be such that there exist $s_1\in \corean{V_1}$ and $s_2 \in \corean{V_2}$ with $\tilde{\iota}_1\left( s_1 \right)=s$ and $\tilde{\iota}_2\left( s_2 \right)=s$. Thus, $\tilde{\Phi}^{\bb{C}}\left( s_2 \right)=s_1 + J_1^{\bb{C}}$ and then $\tilde{\iota_2}\left( \bar{s}_2 \right) = \Phi^{\bb{C}}_*\tilde{\iota}_1\left( \tilde{\Phi}^{\bb{C}}\left( \bar{s}_2 \right) \right)$, but as $\Phi^{\bb{C}}$ is a morphism of manifolds it follows that $\tilde{\Phi}^{\bb{C}}\left( \bar{s}_2 \right)=\overline{\tilde{\Phi}^{\bb{C}}\left( s_2 \right)}$. Therefore it holds that
\[\Phi^{\bb{C}}_*\tilde{\iota}_1\left( \tilde{\Phi}^{\bb{C}}\left( \bar{s}_2 \right) \right)=\Phi^{\bb{C}}_*\tilde{\iota}_1\left( \bar{s}_1 + J_1^{\bb{C}} \right) = \tilde{\iota_1}\left( \bar{s}_1 \right).\]
This shows that the two definitions of conjugation lead to the same section and thus the definition is independent of the local model.
\end{proof}

\end{con}

The construction carried out above proves the following proposition.

\begin{prop}
Let $\left( M,\coan{M} \right)$ be a complex analytic space. Then, there exists a sheaf of $\bb{C}$-algebras $\corean{M}$, that in a local model in $V$ is given by $\quo{\corean{V}}{\left( J+\bar{J} \right)}$, together with a canonical inclusion $\coan{M}\hookrightarrow \corean{M}$, which gives a canonical morphism of $\bb{C}$-ringed spaces $\left( M,\corean{M} \right)\to \left( M,\coan{M} \right)$.\\
For every holomorphic morphism $\phi \colon M \to N$ of complex analytic spaces one obtains a unique morphism of $\bb{C}$-ringed spaces $\phi^{\bb{C}}\colon \left( M,\corean{M} \right)\to \left( N,\corean{N} \right)$ such that
\[
\begin{tikzcd}
\left( M,\corean{M} \right) \arrow[r,"\phi^{\bb{C}}"] \arrow[d] & \left( N,\corean{N} \right) \arrow[d]\\
\left( M,\coan{M} \right) \arrow[r,"\phi"] & \left( N,\coan{N} \right)
\end{tikzcd}
\]
commutes and the sheaf component of $\phi^{\bb{C}}$ respects the complex conjugation.
\label{associatedmorphcomplex}
\end{prop}

\begin{con}
If $\left( M,\coan{M} \right)$ is a complex analytic space, then one can consider the subsheaf $A$ of $\bb{R}$-algebras
\[U\mapsto A\left( U \right):=\left\{ s\in \corean{M}\left( U \right)\mid \bar{s}=s \right\}.\]
In a local model it is clear that $A$ is isomorphic to $C^{\omega}_M$ as defined earlier. Therefore, it is denoted by $C^{\omega}_M$, even if $M$ is not a local model. The $\bb{R}$-ringed space $\left( M,\rean{M} \right)$ is thus a real analytic space and referred to as the \emph{associated real analytic space}.\\
Moreover, observe that since $\phi^{\bb{C}}$ of a holomorphic morphism $\phi\colon M \to N$ commutes with complex conjugation one obtains a morphism
\[\phi^{\bb{R}}\colon \left( M,\rean{M} \right)\to \left( N,\rean{N} \right),\]
such that
\[
\begin{tikzcd}
\left( M,\corean{M} \right) \arrow[r,"\phi^{\bb{C}}"] \arrow[d] & \left( N,\corean{N} \right) \arrow[d]\\
\left( M,\rean{M} \right) \arrow[r,"\phi^{\bb{R}}"] & \left( N,\rean{N} \right)
\end{tikzcd}
\]
commutes. Further, $\corean{M}=\rean{M}\otimes_{\bb{R}}\bb{C}$ is immediate, because $s=\frac{1}{2}\left( s+\bar{s} \right)+ \frac{1}{2}\left( s-\bar{s} \right)$ and the second term is proportional to $i$ in each local model.\\
If $\left( M,\rean{M} \right)$ is a real analytic space, one defines $\corean{M}:= \rean{M}\otimes_{\bb{R}} \bb{C}$. Any real analytic morphism $\phi\colon \left( M,\rean{M} \right)\to \left( N,\rean{N} \right)$ naturally defines $\phi^{\bb{C}}\colon \left( M,\corean{M} \right)\to \left( N,\corean{N} \right)$ by setting
\[\tilde{\phi}^{\bb{C}}\left( s+it \right):=\tilde{\phi}\left( s \right)+i \tilde{\phi}\left( t \right).\]
Note that in both cases $\corean{M}=\rean{M}^{\oplus 2}$ as $\rean{M}$-modules and thus $\corean{M}$ is flat over $\rean{M}$.
\label{associatedmorphreal}
\end{con}

With this understanding of complex-valued real analytic functions on analytic spaces one is now in a position to define the complexification of a real analytic space and gather some useful, well-known facts about it. For more details on these constructions and complexifications in general see e.g.~\cite{realan}.

\begin{defn}
Let $\left( M,\coan{M} \right)$ be a complex analytic space and $\iota\colon \left( N,\rean{N} \right)\to \left( M,\rean{M} \right)$ a real analytic subspace. Then one naturally obtains a morphism
\[\coan{M}\to \corean{M}\to \iota_*\corean{N}\]
and this in turn yields
\[A\colon\iota^{-1}\coan{M}\to \corean{N},\]
as direct and inverse image are adjoint. The space $M$ is called \emph{complexification of $N$} if $A$ is an isomorphism.
\end{defn}

\begin{thm}
Let $M$ be a real analytic space. Then $M$ admits a complexification, that is denoted by $M^{\bb{C}}$ in the following.
\end{thm}

\begin{proof}
See e.g.~\cite[III.3.33]{realan}.
\end{proof}

\begin{prop}
Suppose that $\phi\colon M\to N$ is a morphism of real analytic spaces. Then, there exists $\phi_{\bb{C}}\colon M^{\bb{C}} \to N^{\bb{C}}$ such that the following diagram commutes
\[
\begin{tikzcd}
\left(M^{\bb{C}}, \corean{M^{\bb{C}}}\right) \arrow[r,"\left( \phi_{\bb{C}} \right)^{\bb{C}}"] & \left(N^{\bb{C}},\corean{N^{\bb{C}}} \right)\\
\left( M,\corean{M} \right) \arrow[u,"\iota_1^{\bb{C}}"] \arrow[r,"\phi^{\bb{C}}"] & \left( N,\corean{N} \right) \arrow[u,"\iota_2^{\bb{C}}"]
\end{tikzcd},
\]
where $\iota_1\colon M\to M^{\bb{C}}$ and $\iota_2\colon N \to N^{\bb{C}}$ denote the embeddings, after potentially shrinking $M^{\bb{C}}$ around $M$.
\end{prop}

\begin{proof}
See e.g.~\cite[III.1.8]{realan}.
\end{proof}

\section{Complexification of complex analytic spaces}
\label{compelxificationsofcomplexanalyticspaces}

Recall that if one complexifies a complex manifold $M$ by complexifying the underlying real analytic manifold, then this complexification is locally very simple as it is $M\times \bar{M}$. In the following, it is shown that this is still true for singular spaces and that this locally trivial holomorphic fibration of the complexification over $M$ can be glued to a global fibration. In this work, this global morphism is not necessarily needed, as our question is local. However, it neatly describes the sheaf of $\left( 0,1 \right)$-forms 
on a complex analytic space in a global manner (Section~\ref{pq}). As such, this construction seems valuable enough to present.

\begin{con}
\label{conjugatecomplex}
Let $\left( M,\coan{M} \right)$ be a complex analytic space and let $\iota\colon \left( M,\rean{M} \right)\to \left( M^{\bb{C}},\rean{M^{\bb{C}}} \right)$ be a complexfication of $M$. Then, by definition $\coan{M^{\bb{C}}, p}\cong \corean{M,p}$. Now, in a local model one has
\[\corean{M,p}\cong \quo{\corean{W,p}}{J_p+\bar{J}_p}.\]
One may identify $\corean{\bb{C}^{n},p}$ with $\coan{\bb{C}^{2n},p}$ by mapping the $\bar{z_i}$ coordinates to $z'_{n+i}$. Under this identification the ideal $\bar{J}_p$ is generated by the functions
\[f'_i=\sum_{|I|=1}^{\infty}\bar{a}_{I} {z'}_{n+1}^{I_1}\cdot \dots \cdot {z'}_{2n}^{I_n},\]
where the functions
\[f_i=\sum_{|I|=1}^{\infty}a_{I} z_{1}^{I_1}\cdot \dots \cdot z_{n}^{I_n}\]
generate the ideal $J_p$. Denote by $\bar{M}$ the analytic space defined by $\bar{J}$ under the identification above in an open set $W\subseteq \bb{C}^n$. It follows that $\coan{M\times \bar{M},\left( p,\bar{p} \right)} \cong \corean{M,p}\cong \coan{M^{\bb{C}},p}$. Therefore\footnotemark\footnotetext{Recall, that the category of analytic algebras and germs of analytic spaces are contravariantly equivalent (see e.g.~\cite[0.21]{fischer}).}, there exists an open neighbourhood $U\subseteq M$ around $p$ and $V\subseteq M^{\bb{C}}$ around $\iota\left( p \right)$ such that
\[U\times \bar{U}\cong V\]
as complex analytic spaces. In particular, the sheaf component of the projection $\pi$ to the first factor is simply the inclusion of $\coan{U}$ into $\pi_* \coan{V}$. Thus $\tilde{\pi}_p\colon \coan{U,p}\to \coan{V,p}$ is equal to the concatenation of morphisms
\[\coan{M,p}\to \corean{M,p}\overset{A_p^{-1}}{\to} \coan{M^{\bb{C}},p}.\]
Here, $A_p$ is the germ of the isomorphism that exists for a complexification.
\end{con}

\begin{thm}
Let $ M$ be a complex analytic space and suppose $\iota\colon M \to M^{\bb{C}}$ is a complexification of $M$. This means the canonical morphism $A\colon \iota^{-1}\coan{M^{\bb{C}}}\to \corean{M}$ is an isomorphism and consider the morphism $B\colon \coan{M}\to \corean{M} \overset{A^{-1}}{\to} \iota^{-1}\coan{M^{\bb{C}}}$.\\
Then, after shrinking $M^{\bb{C}}$ around $M$, there exists a holomorphic morphism $\psi\colon M^{\bb{C}}\to M$ such that $\tilde{\psi}_p= B_p$ for every $p\in M$.\\
After shrinking $M^{\bb{C}}$ once again, one may assume that $\psi$ is locally equivalent to a projection.
\label{canonicalfibration}
\end{thm}

\begin{proof}
For every $p\in M$ there exists an open neighbourhood $U_p\subseteq M^{\bb{C}}$ around $p$ and a holomorphic morphism $\psi^p\colon U_p\to M$ such that $\tilde{\psi}^p_p = B_p$. After shrinking $U_p$ around $p$ one has $\psi^p\circ \rst{\iota}{U_p\cap M}=\id_{U_p\cap M}$ for every $p\in M$. Consider the morphism
\[C^p:=\left(\rst{\iota}{U_p}\right)^{-1}(\psi^p)^{\#}\colon \coan{U_p\cap M} \to \left(\rst{\iota}{U_p}\right)^{-1}\corean{U_p}\]
and notice that it agrees with the morphism $B$ at $p$, here
\[\left( \psi^p \right)^{\#}\colon \left( \psi^p \right)^{-1} \coan{M}\to \corean{U_p}\]
denotes the adjoint morphism of $\left( \psi^p \right)^{\sim}$. Since the image of the morphism
\[\rst{B}{M\cap U_p} - C^p\]
is finitely generated and its germ is zero at $p$, it follows that the image sheaf is zero in an open neighbourhood $V_p\subseteq M$ of $p$ and the morphisms agree on that open set. After shrinking $U_p$ around $p$ once again, one may assume $M\cap U_p=V_p$.\\
This implies that along $M$ all the canonical morphisms on germs of the morphisms $\psi^p$ are the same. Since $M^{\bb{C}}$ is paracompact one may assume that the $U_p$ give a locally finite covering $\left\{ W_i \right\}_{i\in I}$ of $M^{\bb{C}}$ after shrinking $M^{\bb{C}}$ around $M$. Denote by $\psi^{i}$ the restriction of $\psi^{p}$ to $W_i\subseteq U_p$ for every $i\in I$. Consider\footnotemark{}\footnotetext{This method is e.g. used in~\cite[p.66, II.9.5]{bredon}.} the set
\[X:=\left\{ q\in M^{\bb{C}}\mid q\in W_{i}\cap W_{i'} \implies \psi^{i}\left( q \right)=\psi^{i'}\left( q \right) \text{ and } \tilde{\psi}^i_q = \tilde{\psi}^{i'}_q \right\}.\]
First of all, the set $X$ is not empty as $M\subseteq X$. Moreover, it is open as the covering is locally finite and all canonical morphisms of germs of the $\psi^i$ are equal to $B$ along $M$. Over $X$ all the topological components and sheaf components of the $\psi^i$ agree and thus  one obtains a morphism of complex analytic spaces $\psi\colon X \to M$.\\
By Construction~\ref{conjugatecomplex}, it follows that around $p\in M$ the analytic space $M^{\bb{C}}$ is isomorphic to $U\times \bar{U}$ and the canonical morphisms on germs of both $\psi$ and the projection to the first factor are equal to $B_p$. Therefore, they agree in an open neighbourhood $V\times \bar{V}$. The second claim holds after sufficiently shrinking $M^{\bb{C}}$.
\end{proof}

\section{Analytic differential forms as universal objects}
\label{four}

In this paper it is of utmost importance that for all sheaves $\rean{\cdot}$, $\corean{\cdot}$ and $\coan{\cdot}$ the differential forms are chosen and constructed in the right way, as one needs to relate all three with each other. On a smooth manifold the switch from holomorphic to analytic (or smooth) forms or functions is seamless and that is one of the strengths of differential geometry. On singular spaces the transitions are not always immediately obvious and certain constructions or approaches can fail. As the differential forms play such a vital role in this paper, they are introduced from scratch and all relevant aspects are proven.\\
The method for introducing the differential forms is similar to the introduction of \emph{K\"ahler differentials} in algebraic geometry (see e.g.~\cite[§1.1.18]{kahler}). It should be noted that the Kähler differentials of analytic spaces are not finitely generated and are not the right notion of differential forms for this paper. Here, an entirely analogous construction of differential forms is carried out, except not in the category of all modules but rather only in the category of finitely generated modules. On an analytic manifold this finitely generated construction returns the usual analytic differential forms.\\
The basic idea is that differential forms should represent the derivation functor and that this characterisation uniquely determines the differential forms. 

\begin{defn}
Let $\phi\colon \left( M,\mc{A} \right)\to \left( N,\mc{B} \right)$ be a morphism of $\bb{K}$-ringed spaces and $\mc{F}$ an $\mc{A}$-module. Denote by $\mathrm{Der}_{\phi}\left( \mc{A}, \mc{F} \right)$ the sheaf of $\phi^{-1}\mc{B}$-linear sheaf morphisms $\delta\colon \rst{\mc{A}}{U}\to \rst{\mc{F}}{U}$ such that
\[\delta\left( f\cdot g \right) = \delta\left( f \right)\cdot g + f\cdot \delta\left( g \right)\]
for all $f,g\in \rst{\mc{A}}{U}$. Elements of this sheaf are referred to as \emph{$\phi$-derivations}.\\
If $\left( N,\mc{B} \right)=\left( \left\{ \text{pt.} \right\}, \bb{K} \right)$, then the $\phi$ may be dropped from the notation.
\end{defn}

\begin{defn}
Let $\phi\colon \left( M,\mc{A} \right)\to \left( N,\mc{B} \right)$ be a morphism of $\bb{K}$-ringed spaces. Then a pair $\left( \Omega,d \right)$, consisting of an $\mc{A}_M$-module $\Omega$ and a $\phi$-derivation $d\colon \mc{A}\to \Omega$, is called a \emph{differential module of $\mc{A}$ relative to $\phi$} if the following hold:
\begin{enumerate}[(i)]
\item the module $\Omega$ is finitely generated,
\item for every open $U\subseteq M$ and finitely generated $\rst{\mc{A}}{U}$-module $\mc{F}$ the morphism

\[\homo{\rst{\Omega}{U}, \mc{F}} \to \mathrm{Der}_{\rst{\phi}{U}}\left( \rst{\mc{A}}{U}, \mc{F} \right),\; h\mapsto h\circ \rst{d}{U}\]

is an isomorphism of $\rst{\mc{A}}{U}$-modules.
\end{enumerate}
\end{defn}

One should now verify that the notion of a differential module is unique up to an appropriate isomorphism. An appropriate isomorphism would be a morphism that relates the two derivations.

\begin{prop}
Let $\phi\colon \left( M,\mc{A} \right)\to \left( N,\mc{B} \right)$ be a morphism of $\bb{K}$-ringed spaces. Suppose that $\left( \Omega,d \right),\left( \Omega',d' \right)$ are both differential modules of $\mc{A}$ relative to $\phi$. Then there exists a unique isomorphism $A\colon \Omega \to \Omega'$ such that $A \circ d = d'$.
\label{DiffIso}
\end{prop}

\begin{proof}
By the universal property there exists a unique morphism $A\colon \Omega \to \Omega'$ such that for $d'\in \mathrm{Der}_{\phi}\left( \mc{A}, \Omega' \right)$ one has $A \circ d = d'$. Conversely, by the same argument one obtains a unique morphism $B \colon \Omega' \to \Omega$ such that $B\circ d' = d$. Thus one has $B \circ A \circ d = d$ and $A \circ B\circ d' =d'$, which implies $B \circ A = \id_{\Omega}$ and $A \circ B = \id_{\Omega'}$, as $d$ and $d'$ are derivations and they are induced by a unique morphism (as $\Omega$ and $\Omega'$ are differential modules), i.e. $\id_{\Omega}$ and $\id_{\Omega'}$. Hence, $A$ and $B$ are unique isomorphisms and the claim follows.
\end{proof}

\begin{defn}
Let $\phi\colon \left( M,\mc{A}_M \right)\to \left( N,\mc{A}_N \right)$ be a morphism of $\bb{K}$-ringed spaces. Because of Proposition~\ref{DiffIso} the differential module of $\mc{A}_M$ with respect to $\phi$ is denoted by
\[\left( \Omega^1_{\phi}\left( \mc{A}_M \right),d_{\phi} \right),\]
whenever it exists.\\
Moreover, the derivation $d_{\phi}$ may be referred to as the \emph{canonical derivation} or the \emph{exterior derivative relative to $\phi$}, whenever they exist.\\
When the morphism $\phi$ is simply $\left( M,\mc{A}_M \right)\to \left( \left\{ p \right\}, \bb{K} \right)$, then the $\phi$ may be dropped from the definition.
\end{defn}

With these elementary properties in hand, one can show that open subsets $U\subseteq \bb{K}^n$ admit a differential module relative to $\phi\colon U \to \left( \left\{\mathrm{pt.}\right\},\bb{K} \right)$.

\begin{prop}
Let $\left( U, \mc{A} \right)$ be a $\bb{K}$-ringed space, with $U$ an open subset of $\bb{K}^n$ and $\mc{A}\in \left\{ \coan{U},\rean{U},\corean{U} \right\}$. Suppose moreover that $\left\{ x_1,\dots,x_n \right\}$ are the standard coordinates on $\bb{K}^n$. Define the following free module $\Omega:= \mc{A}^{\oplus n}$ and the derivation $d_{\mc{A}}\colon \mc{A}\to \Omega$ by

\[f\mapsto \left( \pard{f}{x_1},\dots,\pard{f}{x_n} \right).\]

Then $\left( \Omega,d_{\mc{A}} \right)$ is a differential module of $\mc{A}$ relative to $\phi\colon \left( M,\mc{A} \right)\to \left( \left\{ p \right\}, \bb{K}\right)$.
\end{prop}

\begin{proof}
Let $V\subseteq U$ be open and $\mc{F}$ a finitely generated $\rst{\mc{A}}{V}$-module with $\delta\colon \rst{\mc{A}}{V}\to \mc{F}$ a derivation. Now define $h\left( e_i \right):= \delta\left( x_i \right)$, where $e_i$ is the $i$-th basis vector of $\rst{\Omega}{V}$.  Since $\Omega$ is free this extends by linearity to a morphism $h\colon \rst{\Omega }{V}\to \mc{F}$. Now, one has for all $f\in \rst{\mc{A}}{V}$ that the following holds:
\[\delta\left( f \right)= \sum_{i=1}^n \pard{f}{x_i} \delta\left( x_i \right) =\sum_{i=1}^n \pard{f}{x_i} h\left( e_i \right) =h\left( d_{\mc{A}}f \right),\]
as derivations on $U$ are simply linear combinations of partial derivatives\footnotemark\footnotetext{To see this, note that the derivation $\delta - \sum_{i=1}^n \delta\left( x_i \right)\pard{}{x_i}$ vanishes on polynomials and by Taylor expansion and Krull intersection it follows  that $\delta = \sum_{i=1}^n\delta\left( x_i \right)\pard{}{x_i}$.}. Therefore, the morphism $\mathrm{Hom}\left( \rst{\Omega}{V}, \mc{F} \right) \to \mathrm{Der}\left( \rst{\mc{A}}{V},\mc{F} \right)$ is surjective.\\
Suppose that $h,h'\colon \rst{\Omega}{V}\to \mc{F}$ are such that $h\circ \rst{d_{\mc{A}}}{V}=h'\circ \rst{d_{\mc{A}}}{V}$. However, $d_{\mc{A}}\left( x_i \right)=e_i$ and therefore $h\left( e_i \right)=h'\left( e_i \right)$. Thus $h=h'$.\\
Hence the morphism is injective and $\left( \Omega,d_{\mc{A}} \right)$ is a differential module.
\end{proof}

From the preceding proof one can make the interesting observation that: If the image of a derivation $\delta$ generates the module, the morphism $h\mapsto h\circ \delta$ is injective.\\
Observe the following behaviour of finitely generated ideals with the differential module on an open subset $U\subseteq \bb{K}^n$.

\begin{lem}
Suppose that $M\subseteq \bb{K}^n$ is open and $\mc{A}\in \left\{ \coan{M},\rean{M},\corean{M} \right\}$. Let $J\subseteq \mc{A}$ be an ideal generated by $\left\{ f_1,\dots,f_k \right\}$ and denote by $\left\{ x_1,\dots,x_n \right\}$ the coordinates on $\bb{K}^n$. Then the submodule $dJ\cdot \mc{A}\subseteq \Omega^1\left( \mc{A} \right)$ is finitely generated by the elements
\[df_1,\dots,df_k,f_1dx_1,\dots,f_1dx_n,f_2dx_1,\dots,f_kdx_n.\]
\end{lem}

\begin{proof}
Consider the product $x_if_j$ and note that
\[d\left( x_i f_j \right) = f_jdx_i + x_idf_j\in dJ\]
which implies that $f_jdx_i\in dJ\cdot \mc{A}$ as $x_idf_j\in dJ\cdot \mc{A}$. Moreover,
\[d\left( g\cdot f_j \right) = f_j dg + g df_j = f_j \sum_{i=1}^n \pard{g}{x_i}dx_i + g df_j.\]
This shows that the listed elements do generate the submodule $dJ\cdot \mc{A}$.
\end{proof}

This lemma is the only key needed to show that the differential module of an open subset $U\subseteq \bb{K}^n$ induces a differential module on a local model of an analytic space such that the differential module is not only finitely generated but also finitely presented. The idea is that one simply quotients out the image of the defining ideal under the derivation $d$.

\begin{prop}
Let $\iota\colon \left( N,\mc{A}_N \right) \to \left( M,\mc{A}_M \right)$ be a closed subspace of $\bb{K}$-ringed spaces with defining ideal $J$. Suppose that $\left( \Omega,d \right)$ is a differential module for $\mc{A}_M$ relative to $\phi\colon \left( M,\mc{A}_M \right)\to \left( \left\{ p \right\}, \bb{K} \right)$. Then $\Omega':= \iota^{-1}\left(\quo{\Omega}{\mc{A}_M dJ}\right)$ together with a unique derivation $d'\colon \mc{A}_N \to \Omega'$ such that $\iota_*d'\circ \tilde{\iota} = \mu\circ d$ is the differential module of $\mc{A}_N$, where $\mu\colon \Omega \to \iota_*\Omega'$ is the quotient morphism. In other words, the derivation $d'$ fits into the following commutative diagram
\[
\begin{tikzcd}
\mc{A}_M \arrow[r,"d"] \arrow[d,"\tilde{\iota}"]& \Omega \arrow[d,"\mu"] \\
\iota_*\mc{A}_N \arrow[r,"\iota_*d'"]& \iota_*\Omega'
\end{tikzcd}.
\]
If $\Omega$ is finitely presented and $\mc{A}_M dJ$ is finitely generated, then $\Omega'$ is finitely presented.
\label{localmodeldifferential}
\end{prop}

\begin{proof}
One obtains a $\bb{K}$-linear morphism $\tilde{d}\colon \iota_*\mc{A}_N \to \iota_* \Omega'$ and then a $\bb{K}$-linear morphism $d':=\iota^{-1}\tilde{d}\colon \mc{A}_N \to \Omega'$ by applying the inverse image. The equation $\iota_*d'\circ \tilde{\iota}=\mu \circ d$ is satisfied by definition. This implies
\[d'\circ \iota^{\#}= \iota^{-1}\mu \circ \iota^{-1}d,\]
by applying $\iota^{-1}$, here $\iota^{\#}=\iota^{-1}\tilde{\iota}\colon \iota^{-1}\mc{A}_M \to \mc{A}_N$ denotes the adjoint of $\tilde{\iota}$. The morphism satisfies the Leibniz-rule, as the following calculation on germs shows
\begin{align*}
d'_p\left( \iota^{\#}_p\left( f \right)\cdot \iota^{\#}_p\left( g \right) \right) &= \mu_p\left( d_p\left( f\cdot g \right) \right) = \mu_p\left( f\cdot d_pg + d_pf \cdot g \right)\\
&= \iota^{\#}_p\left( f \right)\cdot d'_p\left( \iota^{\#}_p\left( g \right) \right) + d'_p\left( \iota^{\#}_p\left( f \right) \right)\cdot \iota^{\#}_p\left( g \right).
\end{align*}
It is immediate that $\Omega'$ is generated by $d'\mc{A}_N$ and thus it suffices to show, that for a given finitely generated $\mc{A}_{U}$-module $\mc{F}$ and a derivation $\delta \colon \mc{A}_{U} \to \mc{F}$, there exists a morphism $h'\colon \rst{\Omega'}{U} \to \mc{F}$ such that $\delta = h'\circ \rst{d'}{U}$.\\
Let $\iota'\colon U \to V$ be the embedding of $U\subseteq N$ into a suitable open subset $V\subseteq M$. Now note that $\iota'_*\delta\circ \tilde{\iota}'$ is a derivation from $\mc{A}_{V}$ to $\iota_*'\mc{F}$ and thus there exists $h\colon \rst{\Omega}{V} \to \iota'_*\mc{F}$ such that $h\circ d = \iota'_*\delta \circ \tilde{\iota}'$. For the preceding, note that $\iota'_*\mc{F}$ is finitely generated as an $\mc{A}_V$-module. However because of 
\[h\left( dJ \right)=\iota'_*\delta\left( \tilde{\iota}'\left( J \right) \right)=0,\]
it follows that $\mc{A}_{M}\cdot dJ$ is in the kernel of $h$ and thus induces a morphism $h'\colon \rst{\Omega'}{U} \to \mc{F}$ with $h = \iota'_*h' \circ \mu$. The following holds:
\[\iota'_*\delta \circ \tilde{\iota}' = h\circ d = \iota'_*h' \circ \mu \circ d = \iota'_*h' \circ \iota'_*d' \circ \tilde{\iota}'\]
and this implies $\iota'_*\delta = \iota'_*h' \circ \iota'_*d'$, as $\tilde{\iota}'$ is surjective. Therefore, $\delta=h'\circ d'$. Hence, $\left( \Omega',d' \right)$ is a differential module on $N$.
\\
It is a standard fact that the quotient of a finitely presented module by a finitely generated submodule is still finitely presented.\end{proof}

\begin{rem}
\label{localrtocdifferential}
The proposition above shows that whenever $\left( M,\mc{A}_M \right)$ is a local model of a $\bb{K}$-analytic space, then the differential module
\[\left(\Omega^1\left( \mc{A}_M \right), d_{\mc{A}_M}\right)\]
exists and is finitely presented, i.e. coherent.\\
Let $\left( M,\rean{M} \right)$ be a $\bb{R}$-analytic space in a local model and $\left( M,\corean{M} \right)$ the associated $\bb{C}$-ringed space. Then the differential modules $\Omega^1\left( \rean{M} \right)$ and $\Omega^1\left( \corean{M} \right)$ exist by the preceding proposition. Moreover, note that the pair
\[\left(\Omega^1\left( \rean{M} \right)\otimes_{\bb{R}} \bb{C},d:=d_{\rean{M}}\otimes_{\bb{R}} \id_{\bb{C}}\right)\]
is such that the image of $d$ generates the $\corean{M}$-module. Therefore, the morphism
\[\homo{\Omega^1\left( \rean{M} \right)\otimes_{\bb{R}} \bb{C}, \mc{F}} \to \der{\corean{M},\mc{F}}\]
is injective for any finitely generated $\corean{M}$-module $\mc{F}$. Let $\delta\colon \corean{M} \to \mc{F}$ be a derivation. By $\bb{C}$-linearity it follows that $\delta$ is completely determined by its action on $\rean{M}$ as a derivation $\delta'\colon \rean{M} \to \mc{F}$ of $\rean{M}$-modules. Thus, there exists a unique morphism $\alpha'\colon \Omega^1\left( \rean{M} \right) \to \mc{F}$ of $\rean{M}$-modules such that $\delta'=\alpha' \circ d_{\rean{M}}$. Now, this morphism $\alpha'$ defines a morphism $\alpha\colon \Omega^1\left( \rean{M} \right)\otimes_{\bb{R}} \bb{C} \to \mc{F}$ of $\corean{M}$-modules by complex linear extension. Then one has
\begin{align*}
\alpha\left( d\left( f_1+if_2 \right) \right)&=\alpha\left( d_{\rean{M}}f_1+id_{\rean{M}}f_2 \right)=\alpha'\left( d_{\rean{M}}f_1 \right)+i \alpha'\left( d_{\rean{M}}f_2 \right)\\
&=\delta'\left( f_1 \right)+i\delta'\left( f_2 \right)=\delta\left( f_1+if_2 \right).
\end{align*}
Thus the pair $\left( \Omega^1\left( \rean{M} \right)\otimes_{\bb{R}} \bb{C},d \right)$ is a differential module for $\corean{M}$ and therefore $\Omega^1\left( \corean{M} \right)\cong \Omega^1\left( \rean{M} \right)\otimes_{\bb{R}} \bb{C}$ via a unique morphism respecting the exterior derivatives.
\end{rem}

Having shown that differential modules are unique up to isomorphisms respecting the exterior derivative pays off now, as it immediately follows that the differential modules defined in local models need to glue to a global differential module.

\begin{lem}
Let $\left( M,\mc{A}_M \right)$ be a $\bb{K}$-ringed space such that for every $p\in M$ there exists an open neighbourhood $U_p$ such that the differential module $\left( \Omega^1\left( \mc{A}_{U_p}\right), d_{\mc{A}_{U_p}}  \right)$ exists. Then the differential module $\left( \Omega^1\left( \mc{A}_M \right), d_{\mc{A}_M} \right)$ exists.
\end{lem}

\begin{proof}
Since between two differential modules there exists a unique isomorphism that respects the exterior derivatives, it follows that the collection
\[\left\{ \left(\Omega^1\left( \mc{A}_{U_p} \right), d_{\mc{A}_{U_p}}\right) \right\}_{p\in M}\]
leads to a gluing system and one obtains a finitely presented $\mc{A}_M$-module $\Omega$ and a $\bb{K}$-linear sheaf morphism $d\colon \mc{A}_M \to \Omega$. Thus, $\left( \Omega,d \right)$ is the differential module for $\mc{A}_M$.
\end{proof}

\begin{rem}
The lemma above shows that whenever $\left( M,\mc{A}_M \right)$ is a $\bb{K}$-analytic space the differential module $\left( \Omega^1\left( \mc{A}_M \right), d_{\mc{A}_M} \right)$ relative to $\phi \colon \left( M,\mc{A}_M \right) \to \left( \left\{ p \right\}, \bb{K} \right)$ exists.\\
Let $\left( M,\rean{M} \right)$ be a real analytic space and $\left( M,\corean{M} \right)$ the associated $\bb{C}$-ringed space then both $\Omega^1\left( \rean{M} \right)$ and $\Omega^1\left( \corean{M} \right)$ exist and
\[\Omega^1\left( \rean{M} \right)\otimes_{\bb{R}} \bb{C} \cong \Omega^1\left( \corean{M} \right),\]
just as outlined in Remark~\ref{localrtocdifferential}.
\end{rem}

This remark settles the question of existence for differential forms on analytic spaces. The question remains how differential forms of two analytic spaces may be related via a morphism of analytic spaces. This is the role of the \emph{pull-back morphism} constructed below.

\begin{defn}
Let $\psi\colon M \to N$ be a morphism of $\bb{K}$-ringed spaces, that both admit a differential module. Then a morphism
\[D\psi\colon \psi^*\Omega^1\left( \mc{A}_N \right)\to \Omega^1\left( \mc{A}_M \right)\]
is called the \emph{pull-back morphism of $\psi$} if
\[D\psi\left( \psi^*\left( d_{\mc{A}_N} f \right) \right) = d_{\mc{A}_{M}}\left( \psi^*\left( f \right) \right)\]
for every $f\in \mc{A}_N$.
\end{defn}

\begin{rem}
If the pull-back morphism of a morphism exists, it is of course uniquely determined by the given relation.\\
Notice that if $\psi\colon U\subseteq \bb{K}^n\to V\subseteq\bb{K}^m$ is an analytic morphism of open subsets, then it admits a pull-back morphism as the differential modules involved are free sheaves of modules and the given relation defines a morphism by linear extension.
\end{rem}

\begin{prop}
Let $\psi\colon \left( M,\mc{A}_M \right) \to \left( N,\mc{A}_N \right)$ be a morphism of $\bb{K}$-ringed spaces that both admit a differential module. Suppose moreover, that for every $p\in M$ there exist an open neighbourhood $U\subseteq M$ of $M$ and an open neighbourhood $V\subseteq N$ of $\psi\left( p \right)$ together with a diagram
\[
\begin{tikzcd}
W_1 \arrow[r,"\Psi"] & W_2\\
U \arrow[u,"\iota_1"] \arrow[r,"\rst{\psi}{U}"] & V \arrow[u,"\iota_2"]
\end{tikzcd}
\]
where $\iota_1$ and $\iota_2$ are embeddings, $W_1$ and $W_2$ admit a differential module and $\Psi$ admits a pull-back morphism. Assume that the defining ideals of $U$ resp. $V$ in $W_1$ resp. $W_2$ are finitely generated and the modules generated by the image of the ideal under $d_{W_1}$ and $d_{W_2}$ are also finitely generated.\\
Then $\psi$ admits a pull-back morphism.
\label{pullbackanalytic}
\end{prop}

\begin{proof}
First, assume that one is in the situation of local models. Then, by assumption, there exists a commutative diagram of local models
\[
\begin{tikzcd}
W_1\arrow[r,"\Psi"] & W_2\\
U \arrow[u,"\iota_1"] \arrow[r,"\psi"] & V \arrow[u,"\iota_2"]
\end{tikzcd}.
\]
Now, for the embeddings $\iota_1$ and $\iota_2$ there is a natural choice for the maps $D\iota_1$ and $D\iota_2$. One simply takes $D\iota_1:= \iota_1^*\mu_1$ where $\mu_1\colon \Omega^1\left( \mc{A}_{W_1} \right) \to \iota_{1*}\Omega^1\left( \mc{A}_{U} \right)$ is the quotient map and $D\iota_2$ is defined similarly (see Proposition~\ref{localmodeldifferential}). The relation on the exterior derivatives then holds by definition and $D\iota_1$ and $D\iota_2$ are surjective.\\
At this point one is in the following situation:
\[
\begin{tikzcd}
\iota_1^*\Psi^*\Omega^1\left( \mc{A}_{W_2} \right) \arrow[d,"\psi^*D\iota_2"]\arrow[r,"\iota_1^*D\Psi"] \arrow[dr,"\alpha"]&\iota_1^*\Omega^1\left( \mc{A}_{W_1} \right) \arrow[d,"D\iota_1"]\\
\psi^*\Omega^1\left( \mc{A}_V \right) & \Omega^1\left( \mc{A}_U \right)
\end{tikzcd}
\]
and wants to show that $\alpha$ descends to $\psi^*\Omega^1\left( \mc{A}_{V} \right)$. Observe that
\[\alpha\left( \iota_1^*\Psi^*\left(d_{\mc{A}_{W_2}}s \right)\right)=d_{\mc{A}_U}\left( \iota_1^*\Psi^*s \right)=d_{\mc{A}_U}\left( \psi^*\iota_2^*s \right)\]
and thus $\alpha\left( \iota_1^*\Psi^*\left( d_{A_{W_2}}J_V \right) \right)=0$, where $J_V$ denotes the defining ideal of $V$ in $W_2$. Thus $\alpha$ annihilates the kernel of $\psi^*D\iota_2$ and thus defines a morphism
\[D\psi \colon \psi^*\Omega^1\left( \mc{A}_V \right) \to \Omega^1\left( \mc{A}_U \right).\]
The following calculation on germs shows that the desired relation holds:
\begin{align*}
D\psi\left( \psi^*d_{\mc{A}_V}s \right) &= D\psi\left( \psi^*\left( d_{\mc{A}_V} \iota_2^*s' \right) \right) = D\psi\left( \psi^*\left( D\iota_2\left( \iota_2^*d_{\mc{A}_{W_2}}s' \right) \right) \right)\\
&= \alpha\left( \psi^*\iota_2^*d_{\mc{A}_{W_2}}s' \right) = \alpha\left( \iota_1^*\Psi^* d_{\mc{A}_{W_2}}s' \right)\\
&= d_{\mc{A}_U}\left( \psi^*\iota_2^*s' \right)= d_{\mc{A}_U}\left( \psi^*s \right).
\end{align*}
The morphism obtained this way is unique as $\Omega^1\left( \mc{A}_N \right)$ is generated by $d_{\mc{A}_N}\mc{A}_N$.\\
Returning now to the general case. It was shown above that for every $p\in M$ there exists an open neighbourhood $U_p\subseteq M$ and a unique morphism $\rst{\psi^*\Omega^1\left( \mc{A}_N \right)}{U_p} \to \rst{\Omega^1\left( \mc{A}_M \right)}{U_p}$ satisfying the desired relation. Since these morphisms are unique it follows that any pair of them coincides on the intersection of their domains of definition. Thus these morphisms glue to the desired global morphism $D\psi$. 
\end{proof}

The preceding proposition can now be applied to all the cases, where morphisms have suitable local models. The relevant situations are gathered in the corollary below.

\begin{cor}
\begin{enumerate}[(i)]
\item Let $\phi\colon M\to N$ be a morphism of $\bb{K}$-analytic spaces. Then $\phi$ admits a pull-back morphism.
\item Let $\phi\colon \left( M,\rean{M} \right)\to \left( N,\rean{N} \right)$ be a morphism of real analytic spaces and
\[\phi^{\bb{C}}\colon \left( M,\corean{M} \right)\to \left( N,\corean{N} \right)\]
the associated morphism of $\bb{C}$-ringed space. Then $\phi^{\bb{C}}$ admits a pull-back morphism $D\phi^{\bb{C}}$. One has $D\phi^{\bb{C}}= D\phi\otimes_{\bb{R}}\id_{\bb{C}}$.
\item Let $M$ be a complex analytic space. Consider the canonical morphism
\[\phi\colon\left( M,\corean{M} \right) \to \left( M,\coan{M} \right).\]
Then $\phi$ admits an injective pull-back morphism
\[D\phi\colon \phi^*\Omega^1\left( \coan{M} \right) \to \Omega^1\left( \corean{M} \right).\]
\end{enumerate}
\label{pullbackcmorph}
\end{cor}

The pull-back operation allows to identify exactly what sheaf the relative differential forms can be represented by.

\begin{prop}
Let $\phi\colon \left( M,\mc{A}_M \right) \to \left( N,\mc{A}_{N} \right)$ be a morphism of $\bb{K}$-ringed spaces, that both admit a differential module and further suppose that the map $\phi$ admits a pull-back morphism $D\phi$. Then the quotient module
\[p\colon \Omega^1\left( \mc{A}_{M} \right)\to\Omega:= \quo{\Omega^1\left( \mc{A}_M \right)}{D\phi\left( \phi^*\Omega^1\left( \mc{A}_N \right) \right)}\]
together with the projected exterior derivative $d_{\phi}:=p\circ d_{\mc{A}_M}$ defines the differential module of $\mc{A}_M$ relative to $\phi$.\\
That is, for any morphism $\phi\colon \left( M,\mc{A}_M \right)\to \left( N,\mc{A}_N \right)$, the differential module relative $\phi$
\[ \left( \Omega^1_{\phi}\left( \mc{A}_M \right), d_{\mc{A}_M,\phi} \right)\]
exists if $\phi$ admits a pull-back morphism.\\
If $\Omega^1\left( \mc{A}_M \right)$ is finitely presented, then so is $\Omega^1_{\phi}\left( \mc{A}_M \right)$.
\end{prop}

\begin{proof}
First, $\Omega$ is a finitely generated $\mc{A}_M$-module and $\mathrm{Der}_{\phi}\left( \mc{A}_M,\mc{F} \right)\subseteq \der{\mc{A}_{M}, \mc{F}}$ for any $\mc{A}_M$-module $\mc{F}$. As $\Omega^1\left( \mc{A}_M \right)$ is generated by $d_{\mc{A}_M} \mc{A}_{M}$, it follows that $\Omega$ is generated by $d_{\phi} \mc{A}_M$. Once again it suffices to show that for every finitely generated $\mc{A}_{M}$-module $\mc{F}$ and every derivation $\delta\in \mathrm{Der}_{\phi}\left( \mc{A}_M, \mc{F} \right)$ there exists a morphism $\alpha\colon \Omega \to \mc{F}$ such that $\delta = \alpha \circ d_{\phi}$.\\
However, $\delta$ is also an element of $\der{\mc{A}_M, \mc{F}}$ and thus there exists a morphism
\[\alpha'\colon \Omega^1\left( \mc{A}_M \right) \to \mc{F} \text{ such that } \delta=\alpha'\circ d_{\mc{A}_M}.\]
Note that

\[0=\delta\left( \phi^*s \right)= \alpha'\left( d_{\mc{A}_M}\left( \phi^*s \right) \right)=\alpha'\left( D\phi\left( \phi^*d_{\mc{A}_N}s \right) \right).\]

This implies that $\alpha'$ descends to a morphism $\alpha\colon \Omega \to \mc{F}$ such that $\alpha' = \alpha\circ p$ and thus $\delta=\alpha\circ d_{\psi}$.\\
If $\Omega^1\left( \mc{A}_M \right)$ is finitely presented, then $\Omega$ is the quotient by a finitely generated module and thus finitely presented.
\end{proof}

Due to the uniqueness the pull-back morphism of a concatenation is the concatenation of pull-back morphisms. This is shown in the next lemma.

\begin{lem}
Let $\psi_1\colon \left( M,\mc{A}_M \right) \to \left( N,\mc{A}_N \right)$ and $\psi_2\colon \left( N,\mc{A}_N \right) \to \left( X, \mc{A}_X \right)$ be morphisms of $\bb{K}$-ringed spaces that admit differential modules and pullbacks $D\psi_1$ and $D\psi_2$. Then the following holds:
\[D\psi_1\circ \psi_1^*D\psi_2 = D\left( \psi_2\circ \psi_1 \right).\]
\end{lem}

\begin{proof}
By uniqueness it suffices to show that
\[\left( D\psi_1\circ \psi_1^*D\psi_2 \right)\left( \psi_1^*\psi_2^*d_{\mc{A}_X}s \right)=d_{\mc{A}_M}\left( \psi_1^*\psi_2^*s \right).\]
This holds because
\begin{align*}
 \psi_1^*D\psi_2 \left( \psi_1^*\psi_2^*d_{\mc{A}_X}s \right)&= \psi_1^*\left(D\psi_2 \left( \psi_2^*d_{\mc{A}_X}s \right)  \right)= \psi_1^*\left( d_{\mc{A}_N}\psi_2^*s \right)
\end{align*}
and thus
\begin{align*}
\left( D\psi_1\circ \psi_1^*D\psi_2 \right)\left( \psi_1^*\psi_2^*d_{\mc{A}_X}s \right)= D\psi_1 \left( \psi_1^*\left( d_{\mc{A}_N}\psi_2^*s \right)
 \right)= d_{\mc{A}_M}\left( \psi_1^*\psi_2^*s \right).
\end{align*}
\end{proof}

The pull-back morphism of differential forms descends to a pull-back morphism of relative differential forms as demonstrated below.

\begin{prop}
Suppose that
\[
\begin{tikzcd}
\left( M,\mc{A}_M \right) \arrow[r,"\psi_1"] \arrow[d,"\phi_1"] & \left( N,\mc{A}_N \right) \arrow[d,"\phi_2"] \\
\left( S,\mc{A}_S \right) \arrow[r,"\psi_2"] & \left( S',\mc{A}_{S'} \right)
\end{tikzcd}
\]
is a commutative square of morphisms of $\bb{K}$-ringed spaces that admit differential modules and pull-backs. Then there exists a unique morphism
\[D^{\phi_1,\phi_2}_{\psi_2}\psi_1 \colon \psi_1^* \Omega^1_{\phi_2}\left( \mc{A}_N \right) \to \Omega^1_{\phi_1}\left( \mc{A}_M \right)\]
such that
\[D^{\phi_1,\phi_2}_{\psi_2}\psi_1\left( \psi_1^*d_{\mc{A}_N,\phi_2}s \right)  = d_{\mc{A}_M,\phi_1}\psi_1^*s.\]
\end{prop}

\begin{proof}
Once again it is clear, that if such a morphism exists it is uniquely determined by the given relation. Notice that
\[D\psi_1\left( \psi_1^*D\phi_2\left( \phi_2^*\Omega^1\left( \mc{A}_{S'} \right) \right) \right)=D\phi_1\left( \phi_1^*D\psi_2\left( \psi_2^*\Omega^1\left( \mc{A}_{S'} \right) \right) \right)\subseteq D\phi_1\left( \phi_1^*\Omega^1\left( \mc{A}_S \right) \right).\]
This shows that $D\psi_1$ maps the image of $\psi_1^*D\phi_2$ into the image of $D\phi_1$ and thus there exists a morphism
\[D^{\phi_1,\phi_2}_{\psi_2}\psi_1\colon \psi_1^*\Omega^1_{\phi_2}\left( \mc{A}_N \right) \to \Omega^1_{\phi_1}\left( \mc{A}_M \right).\]
The relation on exterior derivatives is therefore also immediately satisfied.
\end{proof}

Another aspect of differential forms is of technical importance later on. Namely, how does moving to stalks affect differential modules? In particular, the question is: Is the following morphism an isomorphism
\[\mathrm{Hom}\left( \Omega_p,\mc{F}_p \right) \to \mathrm{Der}\left( \mc{A}_{p},\mc{F}_p \right),\; h\mapsto h\circ d_p,\]
where $\left( \Omega,d \right)$ is a differential module on some ringed space with structure sheaf $\mc{A}$? It is clear that this is injective as the image of $d_p$ generates $\Omega_p$. However, surjectivity is less clear. Note that the preceding morphism always fits into the following commutative diagram
\[
\begin{tikzcd}
\mathrm{Der}\left( \mc{A},\mc{F} \right)_p \arrow[r,hook,] & \mathrm{Der}\left( \mc{A}_{p},\mc{F}_p \right)\\
\homo{\Omega,\mc{F}}_p \arrow[u,hook,two heads] \arrow[r,hook] & \homo{\Omega_p,\mc{F}_p} \arrow[u,hook]
\end{tikzcd}.
\]
The bottom arrow is an injective morphism because $\Omega$ is finitely generated and it is an isomorphism if $\Omega$ is finitely presented.

\begin{prop}
Let $\iota\colon M\to U$ be a closed subspace of a $\bb{K}$-ringed space $U$. Suppose the morphism
\[\mathrm{Der}\left( \mc{A}_U,\mc{G} \right)_p \to \mathrm{Der}\left( \mc{A}_{U,p},\mc{G}_p \right)\]
is an isomorphism for every $p\in M$ and any finitely generated $\mc{A}_U$-module $\mc{G}$. Suppose that the defining ideal $J$ of $M$ in $U$ is finitely generated. Then the morphism
\[\mathrm{Der}\left( \mc{A}_M,\mc{F} \right)_p \to \mathrm{Der}\left( \mc{A}_{M,p},\mc{F}_p \right)\]
is an isomorphism for every $p\in M$ and every finitely generated $\mc{A}_M$-module $\mc{F}$.
\label{dervgermsgermsderv}
\end{prop}

\begin{proof}
One obtains a commutative diagram
\[
\begin{tikzcd}
\mathrm{Der}\left( \mc{A}_U,\iota_*\mc{F} \right)_p \arrow[r,hook,two heads,"d"]& \mathrm{Der}\left( \mc{A}_{U,p},\left(\iota_*\mc{F}\right)_p \right)\\
\mathrm{Der}\left( \mc{A}_M,\mc{F} \right)_p \arrow[r,"c"] \arrow[u,hook,"a"] & \mathrm{Der}\left( \mc{A}_{M,p},\mc{F}_p \right) \arrow[u,hook,"b"]
\end{tikzcd}
\]
where $a\left( \delta_p \right):=\left(\iota_*\delta\circ\tilde{\iota}\right)_p$ and $b\left( \delta \right):=\delta\circ \tilde{\iota}_p$. It immediately follows that $c$ is injective. Moreover, for every $\delta\in \mathrm{Der}\left( \mc{A}_{M,p},\mc{F}_p \right)$, one gets $\delta'_p:=d^{-1}\left( b\left( \delta \right) \right)$ and $\delta'_p\left( J_p \right)=0$. As $J$ is finitely generated by $f_i$ around $p$, one obtains a small enough open neighbourhood of $p$ such that $\delta'\left( f_i \right)=0$ and thus
\[\delta'\left( \sum_{i} g_if_i \right)=\sum_i f_i \delta'\left( g_i \right)\in \iota_*\mc{F}.\]
However this has to vanish as $J\cdot \iota_*\mc{F}=0$. Therefore, $\delta'$ descends to
\[\delta''\colon \iota_*\mc{A}_M\to \iota_*\mc{F}.\]
The derivation $\iota^{-1}\delta''$ is then such that
\[b\left( c\left( \iota^{-1}\delta'' \right) \right)=d\left( a\left( \iota^{-1}\delta'' \right) \right)=b\left( \delta \right),\]
which implies $\delta =c\left( \iota^{-1}\delta'' \right)$. Hence, $c$ is surjective.
\end{proof}

The proposition shows that in the cases of analytic spaces and associated $\bb{C}$-ringed spaces of analytic spaces the morphism of the germs of derivations is an isomorphism.\\
In order to talk about flatness and curvature, one needs higher order differential forms and one needs to extend the exterior derivative via the following definition.

\begin{defn}
Suppose $\phi\colon M\to N$ is a morphism of $\bb{K}$-ringed space such that $M$ and $N$ admit a differential module and $\phi$ admits a pull-back morphism. Then one defines
\[\Omega^p_{\phi}\left( \mc{A}_M \right):=\bigwedge^{p}_{i=1} \Omega^1_{\phi}\left( \mc{A}_M \right)\]
and
\[d_{\phi}\colon \Omega^p_{\phi}\left( \mc{A}_M \right)  \to \Omega^{p+1}_{\phi}\left( \mc{A}_M \right),\]
by requiring
\begin{enumerate}[(i)]
\item $d_{\phi}\circ d_{\phi}=0$ and
\item $d_{\phi}\left( \alpha\wedge \beta \right)=d_{\phi}\alpha\wedge \beta + \left( -1 \right)^p\alpha \wedge d_{\phi} \beta$, where $\alpha\in \Omega^p_{\phi}\left( \mc{A}_M \right)$.
\end{enumerate}
\end{defn}

\section{Real analytic differential forms}
\label{pq}

On a complex manifold the sheaf of complex-valued real analytic forms splits into the $\left( 1,0 \right)$-part, generated by the holomorphic forms, and the $\left( 0,1 \right)$-part, generated by the anti-holomorphic forms. The following section is devoted to showing the same for singular spaces, to realising the $\left( 0,1 \right)$-forms as special relative differential forms and to rephrasing them in terms of yet a different type of relative differential forms on the complexification.

\begin{con}
Let $\left( M,\mc{O}_M \right)\to \left( U,\mc{O}_U \right)$ be a local model of a $\bb{C}$-analytic space. Then one obtains, by Proposition~\ref{associatedmorphcomplex}, a local model of the $\bb{C}$-ringed space
\[\iota\colon\left( M,\corean{M} \right)\to \left( U,\corean{U} \right).\]
Moreover, by Proposition~\ref{pullbackcmorph}, there exists a surjective pull-back morphism
\[D\iota\colon \iota^*\Omega^1\left( \corean{U} \right) \to \Omega^1\left( \corean{M} \right).\]
Now, the module $\Omega^1\left( \corean{U} \right)$ splits into $\Omega^{1,0}U$ and $\Omega^{0,1}U$, where the former is generated by $dz_i$ and the latter by $d \bar{z}_i$. The images $\Omega^{1,0}M$ and $\Omega^{0,1}M$ of these two submodules via the pull-back morphism induce a splitting
\[\Omega^1\left( \corean{M} \right)=\Omega^{1,0}M\oplus \Omega^{0,1}M.\]
In order to see that this splitting holds, independently of the local model, note that the defining ideal $\left( M,\corean{M} \right)$ is given by $J':=\left( J,\bar{J} \right)$, where $J$ is the defining ideal of $\left( M,\mc{O}_M \right)$. The short exact sequence
\[
\begin{tikzcd}
0\arrow[r] & \corean{U} dJ' \arrow[r] & \Omega^1\left( \corean{U} \right) \arrow[r] & \iota_* \Omega^1\left( \corean{M} \right) \arrow[r] & 0
\end{tikzcd}
\]
naturally decomposes to
\[
\begin{tikzcd}[cramped]
0\arrow[r] & J^{1,0} \oplus J^{0,1} \arrow[r] & \Omega^{1,0}U\oplus \Omega^{0,1}U \arrow[r] & \iota_* \Omega^{1,0}M \oplus \iota_* \Omega^{0,1}M \arrow[r] & 0
\end{tikzcd},
\]
where $J^{1,0}:=\corean{U} \partial J+ J'\Omega^{1,0}U$ and $J^{0,1}:=\corean{U} \bar{\partial} \bar{J}+ J'\Omega^{0,1}U$. Suppose now that $\left( M,\coan{M} \right)\to \left( V, \coan{V} \right)$ is a different local model for $M$. Then there exists a commutative square
\[
\begin{tikzcd}
\left( U,\coan{U} \right) \arrow[r] & \left( V,\coan{V} \right)\\
\left( M,\coan{M} \right) \arrow[u] \arrow[r,"\id"] & \left( M,\coan{M} \right) \arrow[u]
\end{tikzcd},
\]
after potentially shrinking $U$, $V$ and $M$. Moreover, by Proposition~\ref{associatedmorphcomplex}, there exists a commutative square
\[
\begin{tikzcd}
\left( U,\corean{U} \right) \arrow[r,"\phi"] & \left( V,\corean{V} \right)\\
\left( M,\corean{M} \right) \arrow[u,"\iota"] \arrow[r,"\id"] & \left( M,\corean{M} \right) \arrow[u,"\iota'"]
\end{tikzcd}.
\]
Clearly, one may also consider the identity map in the other direction, so that one obtains a diagram
\[
\begin{tikzcd}
\left( U,\corean{U} \right) \arrow[rr,"\phi", bend left =5]  & & \arrow[ll,"\psi", bend left =5] \left( V,\corean{V} \right)\\
& \left( M,\corean{M} \right) \arrow[ur,"\iota'",swap] \arrow[ul,"\iota"] &
\end{tikzcd}
\]
with $\phi\circ\iota = \iota'$ and $\psi\circ \iota'=\iota$, after potentially shrinking $U$, $V$ and $M$. The aim is to show that
\[D\iota\left( \iota^*\Omega^{0,1}U \right) = D\iota'\left( \iota'^*\Omega^{0,1}V \right).\]
As $\phi$ is induced from a holomorphic morphism of manifolds, it follows that
\[D\phi\left( \phi^*\Omega^{0,1}V \right)\subseteq \Omega^{0,1}U\]
and thus one obtains
\[D\iota'\left( \iota'^* \Omega^{0,1}V \right)=D\iota\left( \iota^{*}D\phi\left( \phi^*\Omega^{0,1}V \right) \right)\subseteq D\iota \left( \iota^* \Omega^{0,1}U \right).\]
Doing the same for the map $\psi$ leads to
\[D\iota\left( \iota^* \Omega^{0,1}U \right)=D\iota'\left( \iota'^{*}D\psi\left( \psi^*\Omega^{0,1}V \right) \right)\subseteq D\iota' \left( \iota'^* \Omega^{0,1}V \right).\]
Thus one has shown that the definition of $\Omega^{0,1}M$ is independent of the local model. The same argument works for $\Omega^{1,0}M$.\\
Note that via this construction the sheaf $\Omega^1\left( \corean{M} \right)$ naturally comes with an endomorphism
\[\mathcal{I}_M\colon \Omega^{1,0}M\oplus \Omega^{0,1}M \to  \Omega^{1,0}M\oplus \Omega^{0,1}M,\; a\oplus b \mapsto ia\oplus -ib.\]
As this construction is independent of the local model it follows that the subsheaves $\Omega^{0,1}M$ and $\Omega^{1,0}M$ as well as the endomorphism $\mc{I}_M$ are defined globally on $\left( M,\corean{M} \right)$.
\label{10construction}
\end{con}

\begin{defn}
Let $\left( M,\coan{M} \right)$ be a $\bb{C}$-analytic space. The subsheaf $\Omega^{0,1}M\subseteq\Omega^1\left( \corean{M} \right)$ constructed in~\ref{10construction} is called the \emph{sheaf of $\left( 0,1 \right)$-differential forms on $M$} and the subsheaf $\Omega^{1,0}M$ is called the \emph{sheaf of $\left( 1,0 \right)$-differential forms on $M$}.\\
The endomorphism $\mc{I}_M\colon \Omega^1\left( \corean{M} \right)\to\Omega^1\left( \corean{M}\right)$ is called the \emph{associated almost complex structure of $M$} and note that $\mc{I}_M^2=-\id$.
\end{defn}

\begin{rem}
The $\corean{M}$-modules $\Omega^{1,0}M$ and $\Omega^{0,1}M$ are $\corean{M}$-coherent.
\end{rem}

\begin{lem}
Let $\left( M,\coan{M} \right)$ be a $\bb{C}$-analytic space. Then there exists a split exact sequence
\[
\begin{tikzcd}
0\arrow[r] & \Omega^{1,0}M \arrow[r] &\Omega^1\left( \corean{M} \right) \arrow[r] & \Omega^{0,1}M \arrow[r] & 0
\end{tikzcd}.
\]
That is
\[\Omega^1\left( \corean{M} \right)=\Omega^1\left( \rean{M} \right)\otimes_{\bb{R}} \bb{C} \cong \Omega^{1,0}M\oplus \Omega^{0,1}M\]
and denote the thus induced projections by $p_{1,0}$ and $p_{0,1}$.
\end{lem}

\begin{proof}
The projection $p_{1,0}$ is given by
\[\alpha\mapsto \frac{1}{2}\left(\alpha -i \mc{I}_M\left( \alpha \right)\right)\]
and the projection $p_{0,1}$ is given by
\[\alpha\mapsto \frac{1}{2}\left(\alpha +i \mc{I}_M\left( \alpha \right)\right).\]
\end{proof}

\begin{defn}
Let $\left( M,\coan{M} \right)$ be a $\bb{C}$-analytic space. Then one defines the operators
\[\partial_M:=p_{1,0}\circ d_{\corean{M}}\colon \corean{M} \to \Omega^{1,0}M\]
and
\[\bar{\partial}_M:=p_{0,1}\circ d_{\corean{M}} \colon \corean{M} \to \Omega^{0,1}M.\]
From the local description of $\Omega^{0,1}M$ it is clear that $\coan{M}\subseteq \ker{\bar{\partial}_M}$.
\end{defn}

Now it is shown that the $\left( 1,0 \right)$-forms are generated by the holomorphic forms in a canonical way and that the $\left( 0,1 \right)$-forms are relative differential forms.

\begin{prop}
Let $\left( M,\coan{M} \right)$ be a complex analytic space. Then there exists a canonical isomorphism
\[\Omega^{1,0}M \cong \Omega^1\left( \coan{M} \right)\otimes_{\coan{M}} \corean{M}.\]
Moreover, the morphism $\phi\colon \left( M, \corean{M} \right) \to \left( M,\coan{M} \right)$, given by inclusion in the sheaf component, admits a pull-back morphism and this yields the isomorphism above.\\
Another way of saying this is
\[\Omega^{0,1}M \cong \Omega^1_{\phi}\left( \corean{M} \right),\]
that is, the $\left( 0,1 \right)$-forms are the relative differential forms of $\corean{M}$ relative to $\phi$.
\end{prop}

\begin{proof}
The morphism $\phi$ admits a pull-back by Proposition~\ref{pullbackcmorph} and the pull-back is injective. Recall that locally the pull-back morphism is induced from a local model, i.e.
\[
\begin{tikzcd}
\left( V,\corean{V} \right) \arrow[r,"\Psi"] & \left( V,\coan{V} \right) \\
\left( U,\corean{U} \right) \arrow[u,"\iota^{\bb{C}}"] \arrow[r,"\rst{\phi}{U}"] & \left( U,\coan{U} \right) \arrow[u,"\iota"]
\end{tikzcd},
\]
meaning that
\[D\phi \circ \phi^*D\iota = D\iota^{\bb{C}}\circ  {\iota^{\bb{C}}}^*D\Psi.\]
As $\Psi$ is a morphism of manifolds, it follows that $\im{D\Psi}=\Omega^{1,0}V$ and thus,
\[\im{  D\iota^{\bb{C}}\circ  {\iota^{\bb{C}}}^*D\Psi }=\Omega^{1,0}U.\]
The morphism $\phi^*D\iota$ is surjective, therefore it follows that $\im{D\phi}= \Omega^{1,0}U$.\\
Because $\Omega^1\left( \corean{M} \right)=\Omega^{1,0}M\oplus \Omega^{0,1}M$ and the morphism $D\phi$ is injective with image $\Omega^{1,0}M$, it follows that $\Omega^{0,1}M$ is the relative differential module with respect to $\phi$.
\end{proof}

In order to rephrase the $\left( 0,1 \right)$-forms in terms of the complexification, one needs a technical lemma to extend derivations to the complexification.

\begin{lem}
Let $M^{\bb{C}}$ be the complexification of the complex analytic space $M$ and denote by $\iota\colon M\to M^{\bb{C}}$ the inclusion. Suppose that $\mc{F}$ is a finitely generated $\coan{M^{\bb{C}}}$-module and $\delta\colon \corean{M}\to \iota^{-1}\mc{F}$ is a derivation. Then there exists an open neighbourhood $U\subseteq M^{\bb{C}}$ of $M$ and a derivation $D\colon \coan{U}\to \rst{\mc{F}}{U}$ such that $\iota^{-1}D=\delta$.
\label{extenddervcomplex}
\end{lem}

\begin{proof}
By Proposition~\ref{dervgermsgermsderv} it follows that around $p\in M$ the derivation $\delta$ is determined by $\delta_p\colon \corean{M,p}\to \mc{F}_p$. This implies that there exists an open neighbourhood $U_p\subseteq M^{\bb{C}}$ around $p$ and a derivation $D^p\colon \coan{U_p}\to \rst{\mc{F}}{U_p}$ such that $D^p_p=\delta_p$. Again, by Proposition~\ref{dervgermsgermsderv} the derivation $D^p$ is determined by $D^p_p$ locally around $p$. Therefore $\iota^{-1}D^p =\rst{\delta}{V_p}$ after shrinking $U_p$, where $V_p=M\cap U_p$.\\
These maps $D^{p}$ glue to a derivation on an open neighbourhood $U_p'\subseteq M^{\bb{C}}$ of $M$ by the same gluing argument as in the proof of Theorem~\ref{canonicalfibration}.
\end{proof}

All these methods allow one to see the $\left( 0,1 \right)$-forms as relative differential forms of the canonical fibration of the complexification.

\begin{prop}
Let $M$ be a complex analytic space and $\Phi\colon M^{\bb{C}}\to M$ be the canonical fibration (see Theorem \ref{canonicalfibration}) that extends the canonical morphism
\[\phi\colon \left( M,\corean{M} \right)\to \left( M,\coan{M} \right).\]
Then $\Omega^1_{\Phi}\left( \coan{M^{\bb{C}}} \right)$ is such that
\[\iota^{-1}\Omega^1_{\Phi}\left( \coan{M^{\bb{C}}} \right) \cong \Omega^{0,1}M,\]
where $\iota$ is the inclusion of $M$ in $M^{\bb{C}}$ and the isomorphism is such that $\iota^{-1}d_{\Phi}$ gets mapped to $\bar{\partial}_M$.\\
More precisely, the pair $\left( \iota^{-1}\Omega^1_{\Phi}\left( \coan{M^{\bb{C}}} \right),\iota^{-1}d_{\Phi} \right)$ is the differential module of $\corean{M}$ with respect to $\phi$.
\label{Rel10Dif}
\end{prop}

\begin{proof}
As discussed earlier, $\Omega^{0,1}M$ is the relative differential module of $\corean{M}$ with respect to $\phi$. Note that
\[\iota^{-1}d_{\Phi}\colon \iota^{-1}\coan{M^{\bb{C}}}\to \iota^{-1}\Omega^1_{\Phi}\left( \coan{M^{\bb{C}}} \right)\]
is a differential on $\corean{M}$ with values in $\iota^{-1}\Omega^1_{\Phi}\left( \coan{M^{\bb{C}}} \right)$ and the module is generated by the image of that differential.\\
Notice that, since $\Omega^{0,1}M$ and $\iota^{-1}\Omega^1_{\Phi}\left( \coan{M^{\bb{C}}} \right)$ are both coherent, it suffices to check that derivations into coherent modules are induced by morphisms on $\iota^{-1}\Omega^1_{\Phi}\left( \coan{M^{\bb{C}}} \right)$.\\
For any coherent $\corean{U}$-module $\mc{F}$ there exists a coherent $\coan{U^{\bb{C}}}$-module $\mc{F}'$ that extends $\mc{F}$ to an open neighbourhood of the complexfication (see~\cite[I.2.8]{realan}) and every $\phi$-derivation $\delta\colon\corean{U}\to \mc{F}$ also extends to a derivation $\delta'$ on $\coan{U^{\bb{C}}}$ and $\mc{F}'$ (see Lemma~\ref{extenddervcomplex}). For such modules and derivations there exists a morphism $\alpha\colon \rst{\Omega^{1}\left( \coan{M^{\bb{C}}} \right)}{U^{\bb{C}}}\to \mc{F}'$ such that $\alpha\circ \rst{d}{U^{\bb{C}}}= \delta'$, and thus
\[\iota^{-1}\alpha \circ \iota^{-1}\rst{d}{U^{\bb{C}}}=\delta.\]
However, $\delta_p=\alpha_p\circ d_{p}$ annihilates $\left(\Phi^{-1}\coan{M}\right)_p$, showing that $\alpha_p$ annihilates the image of $\left( \Phi^*\Omega^1\left( \coan{M} \right) \right)_p$ for every $p\in M$. Thus, after shrinking $U^{\bb{C}}$ one may assume that $\alpha$ annihilates the image of $\rst{\Phi^{*}\Omega^1\left( \coan{M} \right)}{U^{\bb{C}}}$. Therefore, this morphism descends to $\rst{\Omega^1_{\Phi}\left( \coan{M^{\bb{C}}} \right)}{U^{\bb{C}}}$ such that $\alpha \circ \rst{d_{\Phi}}{U^{\bb{C}}}=\delta$ and $\iota^{-1}\alpha \circ \iota^{-1}d_{\Phi}=\iota^{-1}\delta'=\delta$.\\
This shows that the pair $\left( \iota^{-1}\Omega^{1}_{\Phi}\left( \coan{M^{\bb{C}}} \right),\iota^{-1}d_{\Phi} \right)$ is the differential module of $\corean{M}$ with respect to $\phi$.
\end{proof}

\section{Relative Riemann-Hilbert Theorem for submersions}
\label{RelSub}

Before going into the details of the relative Riemann-Hilbert Theorem with singular fibers, it is instructive and necessary to recall the smooth version as proven by P.~Deligne~\cite{deligne}. This section starts with the basic definitions of connections, local triviality and relative local systems. Afterwards, implications and consequences of the smooth relative Riemann-Hilbert Theorem are gathered.

\begin{defn}
Let $f\colon M \to N$ be a morphism between $\bb{K}$-analytic spaces. The map $f$ is called \emph{locally trivial at $p\in M$} if there exist an open neighbourhood $U\subseteq M$ of $p$, an open neighbourhood $V\subseteq N$ of $f\left( p \right)$, with $f\left( U \right)\subseteq V$, a $\bb{K}$-analytic space $Z$ and an analytic isomorphism $\psi\colon U \to V \times Z$ such that the diagram
\[
\begin{tikzcd}
 U \arrow[rr,"\psi"] \arrow[dr,"f",swap] & & V\times Z \arrow[dl,"\pi_1"] \\
 & V &
\end{tikzcd}
\]
commutes. The map $f$ is called \emph{locally trivial} if it is  locally trivial at every element of $M$. In the following, we call the map a \emph{reduced} locally trivial morphism if $N$ and all fibers of $f$ are reduced. The map is called a \emph{submersion (at $p$)} if in the definition $Z$ can be chosen to be an open subset of $\bb{K}^k$.
\label{LocTriv}
\end{defn}

\begin{defn}
Let $f\colon M\to N$ be an analytic map between $\bb{K}$-analytic spaces and $\mc{F}$ a sheaf of $\mc{A}_M$-modules. Then a morphism of $f^{-1}\mc{A}_N$-modules $\nabla_f\colon \mc{F}\to \Omega^1_f\left( \mc{A}_M \right)\otimes_{\mc{A}_M} \mc{F}$ satisfying
\[\nabla_f(g\cdot s) =(d_f g)\otimes s + g\cdot(\nabla_f s),\]
for $g\in \mc{A}_M$ and $s\in \mc{F}$, is called \emph{relative connection on $\mc{F}$}. A relative connection defines morphisms of $f^{-1}\mc{A}_N$-modules
\[\nabla_f^{(1)}\colon \Omega^1_f\left( \mc{A}_M \right)\otimes_{\mc{A}_M} \mc{F} \to \Omega^{2}_f\left( \mc{A}_M \right)\otimes_{\mc{A}_M} \mc{F}\]
such that
\[\nabla_f^{\left( 1 \right)}\left( \alpha \otimes s \right) = \left( d_f\alpha \right)\otimes s -\alpha \wedge \left( \nabla_f s \right).\]
The relative connection is called \emph{flat} if $F^{\nabla}:=\nabla_f^{(1)}\circ \nabla_f =0$.\\
If the map $f$ is $f\colon M \to \left( \left\{ p \right\}, \bb{K} \right)$, then $\nabla_f$ is simply called \emph{connection}.
\end{defn}

\begin{defn}
Let $f\colon M \to N$ be an analytic map of $\bb{K}$-analytic spaces and let $\mc{F}$ and $\mc{G}$ be sheaves of $\mc{A}_{M}$-modules with relative connection $\nabla$ and $\nabla'$, respectively. Suppose $A\colon \mc{F} \to \mc{G}$ is a morphism of $\mc{A}_{M}$-modules such that
\[\id\otimes A \circ \nabla = \nabla' \circ A\colon \mc{F} \to \Omega^1_{f}\left( \mc{A}_{M} \right)\otimes_{\mc{A}_{M}} \mc{G}.\]
Then, $ A$ is called a \emph{morphism of connections} or \emph{preserving the connections $\nabla$ and $\nabla'$}. One may denote $A$ being a morphism of connections by writing
\[A\colon \left( \mc{F},\nabla \right) \to \left( \mc{G},\nabla' \right).\]
\end{defn}

\begin{lem}
Let $f\colon M \to N$ be an analytic map of $\bb{K}$-analytic spaces and let $A\colon \left( \mc{F},\nabla \right)\to \left( \mc{G},\nabla' \right)$ be a morphism of connections. Suppose that $\mc{F}$ is finitely generated. Then the following statements are true:
\begin{enumerate}[(i)]
\item $F^{\nabla'}\circ A = \id\otimes A \circ F^{\nabla}\colon \mc{F} \to \Omega^2_f\left( \coan{M} \right)\otimes_{\coan{M}} \mc{G}$.
\item The cokernel sheaf of $A$ comes with a connection $\nabla''$ such that the quotient morphism preserves the connections.
\end{enumerate}
In particular, if $\nabla'$ is flat, then so is $\nabla''$.
\label{connecpreserve}
\end{lem}

\begin{proof}
\begin{enumerate}[(i)]
\item Let $s\in \mc{F}$ and write $\nabla s = \sum_i \alpha_i \otimes s_i$, where the $s_i$ are generators of $\mc{F}$, and write $\nabla s_k = \sum_i \alpha_{ki}\otimes s_i$. One can assume this as the question is local. Note that $\nabla'\left( A\left( s \right) \right)=  \sum_i \alpha_i \otimes A\left( s_i \right)$ and similarly, $\nabla'\left( A\left( s_k \right) \right)=  \sum_i \alpha_{ki} \otimes A\left( s_i \right)$. Then one has
\begin{align*}
\id\otimes A\left( \nabla^{\left( 1 \right)}\left( \nabla\left( s \right) \right) \right) &=\id\otimes A \left( \sum_i d_f\alpha_i \otimes s_i - \sum_i \alpha_i\wedge \nabla s_i \right)\\
&= \sum_i d_f \alpha_i \otimes A\left( s_i \right) - \sum_{i,j} \alpha_i \wedge \alpha_{ij} \otimes A\left(s_j\right)\\
&= \sum_i d_f \alpha_i \otimes A\left( s_i \right) - \sum_{i} \alpha_i \wedge \nabla'\left(A\left(s_i\right)\right)\\
&= \nabla'^{\left( 1 \right)} \left(  \sum_i \alpha_i\otimes A\left( s_i \right) \right)\\
&= \nabla'^{\left( 1 \right)} \left(  \nabla' \left( A\left( s \right) \right) \right)\\
&= F^{\nabla'} \left( A\left( s \right) \right).
\end{align*}
\item Denote the image of $A$ by $K$ and the cokernel sheaf by $Q$. Moreover, denote by $q\colon \mc{G} \to Q$ the quotient morphism. Note that the definition
\[\nabla''\left( q\left( t \right) \right):= \id\otimes q \left( \nabla' \left(t\right) \right)\in \Omega^1_f\left( \mc{A}_{M} \right)\otimes_{\mc{A}_{M}} Q,\]
for $t\in \mc{G}$ is well-defined because if $q\left( t' \right)=q\left( t \right)$, then $t-t'\in K$ and hence
\[\nabla'\left( t-t' \right)\in \im{\id\otimes A}=\ker{\id\otimes q}\subseteq \Omega^1_f\left( \mc{A}_{M} \right)\otimes_{\mc{A}_{M}} \mc{G},\]
which means
\[\id\otimes q \left( \nabla' \left( t-t' \right) \right)=0.\]
This defines a relative connection on the cokernel presheaf of $A$ and by sheafification a relative connection on $Q$.
\end{enumerate}
\end{proof}

\begin{defn}
Let $f\colon M \to N$ be a locally trivial morphism of $\bb{K}$-analytic spaces. A \emph{relative local system on $M$} is a $f^{-1}\mc{A}_N$-module $V$ such that for every element $p\in M$ there exists an open neighbourhood $U\subset M$ of $p$ and a coherent sheaf of $\mc{A}_{f\left( U \right)}$-modules $\mc{F}$ such that $\rst{V}{U}\cong \left( \rst{f}{U} \right)^{-1} \mc{F}$. A relative local system is called \emph{torsion-free} if the modules $\mc{F}$ in the definition can be assumed to be torsion-free.\\
If $N=\left( \left\{ \mathrm{pt.} \right\},\bb{K} \right)$, then $V$ is simply called \emph{local system}.
\end{defn}

\begin{defn}
Let $f\colon M\to N$ be a locally trivial morphism of $\bb{K}$-analytic spaces and let $V$ be a relative local system. Then $\nabla^V:=\id_V\otimes d_f$ defines a canonical flat connection on the coherent module $V\otimes_{f^{-1}\mc{A}_N} \mc{A}_M$. Note that $V\to V\otimes_{f^{-1}\mc{A}_N} \mc{A}_M$ is injective as the canonical morphisms $\tilde{f}_p$ of stalks of the structure sheaves are faithfully flat because they are flat between local rings.
\label{canconnec}
\end{defn}

Now, recall the well-known relative Riemann-Hilbert correspondence in the case of a holomorphic submersion between complex analytic spaces.

\begin{thm}
Let $f\colon M\to N$ be a holomorphic submersion of complex spaces, i.e. the fibers are assumed to be complex manifolds.
\begin{enumerate}[(i)]
\item Let $\mc{F}$ be a coherent sheaf on $M$ with a flat relative connection $\nabla_f$. Then the sheaf $\ker{\nabla_f}$ is a relative local system. Moreover,
\[\mc{F}\cong \ker{\nabla_f}\otimes_{f^{-1}\mc{O}_N} \mc{O}_M.\]
\item Let $V$ be a relative local system. Then $\mc{F}:= V\otimes_{f^{-1}\mc{O}_N} \mc{O}_M$ is a coherent sheaf and $\nabla^V:=\id_V\otimes d_f$ defines a flat relative connection on $\mc{F}$. Moreover,
\[\ker{\nabla^V} = V.\]
\end{enumerate}
\label{relriemanhilbert}
\end{thm}

\begin{proof}
See~\cite{deligne}.
\end{proof}

An important tool is going to be the pull-back of a connection since in the singular version one will pull-back the relative connection to a submersion and then work there.

\begin{defn}
Let 
\[
\begin{tikzcd}
M' \arrow[r,"g"] \arrow[d,"f'",swap] & M \arrow[d,"f"]\\
N' \arrow[r,"g'",swap] & N
\end{tikzcd}
\]
be a commutative square of holomorphic maps. Let $\left( \mc{F},\nabla_f \right)$ be a coherent sheaf on $M$ with $f$-relative connection. Locally the coherent sheaf $g^*\mc{F}$ is generated by pull-back sections $g^*s_i$ of $\mc{F}$. Define
\[\left(g^*\nabla_f\right)\left( \sum_i a_i g^*s_i \right) := \sum_i d_{f'}\left( a_i \right) g^*s_i + \sum_i a_i \left(D_{g'}^{f',f}g\otimes{\id_{g^*\mc{F}}}\right)\left(g^*\left( \nabla_f s_i \right)\right),\]
where $D_{g'}^{f',f}g\colon g^*\Omega^1_f\left( \coan{M} \right) \to \Omega^1_{f'}\left( \coan{M'} \right)$ is the canonical relative differential of $g$. This is well-defined because $\nabla_f$ is a connection and satisfies the Leibniz-rule. The map $g^*\nabla_f \colon g^*\mc{F} \to \Omega^1_{f'}\left( \coan{M'} \right)\otimes_{\coan{M'}} g^*\mc{F}$ is a $f'$-relative connection. It is called the \emph{pull-back connection}.\\
Note that the pull-back of a flat connection is also flat. Moreover, if $N'$ is a simple point and $M'$ is the fiber over this point, then the pull-back relative connection is simply a connection.
\label{pullbackconnec}
\end{defn}

\begin{prop}
\label{relativepullback}
Let 
\[
\begin{tikzcd}
M' \arrow[r,"g"] \arrow[d,"f'",swap] & M \arrow[d,"f"]\\
N' \arrow[r,"g'",swap] & N
\end{tikzcd}
\]
be a commutative square of holomorphic maps where $f$ and $f'$ are submersions. Let $\left( \mc{F},\nabla_f \right)$ be a coherent sheaf on $M$ with flat $f$-relative connection. Then, 
\[\ker{g^{*}\nabla_f} = g^{-1} \ker{\nabla_f} \otimes_{f'^{-1}g'^{-1}\mc{O}_N} f'^{-1}\mc{O}_{N'} .\]
The tensor product above makes sense since $g^{-1}\ker{\nabla_f}$ is a $g^{-1}f^{-1}\mc{O}_N=f'^{-1}g'^{-1}\mc{O}_N$-module.
\label{relpullbackconnec}
\end{prop}

\begin{proof}
There are the following identities:
\begin{align*}
g^*\mc{F}&= g^{-1}\mc{F} \otimes_{g^{-1}\mc{O}_M} \mc{O}_{M'}\\
&= g^{-1}\left( \ker{\nabla_f} \otimes_{f^{-1}\mc{O}_N} \mc{O}_M \right) \otimes_{g^{-1}\mc{O}_M} \mc{O}_{M'}\\
&= g^{-1}\ker{\nabla_f} \otimes_{g^{-1}f^{-1}\mc{O}_N} \mc{O}_{M'}\\
&= \left(g^{-1}\ker{\nabla_f}\otimes_{f'^{-1}g'^{-1}\mc{O}_N} f'^{-1}\mc{O}_{N'}\right) \otimes_{g^{-1}f^{-1}\mc{O}_N} \mc{O}_{M'}
\end{align*}
Under this identification the connection $g^*\nabla_f$ is simply the canonical connection associated to the relative local system
\[g^{-1}\ker{\nabla_f}\otimes_{f'^{-1}g'^{-1}\mc{O}_N} f'^{-1}\mc{O}_{N'}.\]
By Theorem~\ref{relriemanhilbert} it follows that the kernel is equal to this relative local system.
\end{proof}

\begin{cor}
Let
\[
\begin{tikzcd}
M' \arrow[r,"\iota"] \arrow[d,"f'",swap] & M \arrow[d,"f"]\\
\left( \left\{ p \right\},\bb{C} \right) \arrow[r,"\iota'",swap] & N
\end{tikzcd}
\]

be a commutative square where $f$ is a submersion and $M'$ is the fiber of $f$ over $p\in N$. Suppose $\left( \mc{F},\nabla_f \right)$ is an $\coan{M}$-coherent sheaf with a flat $f$-relative connection. Then
\[\ker{\iota^*\nabla_f}_q=\ker{\nabla_f}_q \otimes_{\coan{N,p}} \bb{C}, \]
where $q\in M'$.
That is to say, the stalk of the kernel of the connection along the fibers is equal to the fiber of the kernel of the relative connection.
\end{cor}

This observation about the pull-back of relative connections to the fibers allows one to recognize when a given flat relative connection is trivial. Namely, exactly when the connections along the fibers all have global parallel frames.

\begin{thm}
Let $f\colon M\times N \to M$ be a holomorphic submersion of reduced complex spaces with connected fibers and let $\left( \mc{F},\nabla_f \right)$ be a locally free sheaf with flat $f$-relative connection. Then there exists a coherent sheaf $\mc{G}$ on $M$ such that $f^*\mc{G}\cong \mc{F}$ and $\nabla_f=\nabla^{f^{-1}\mc{G}}$ if and only if $\rst{\left( \mc{F},\nabla_f \right)}{\preim{f}{x}}\cong \left( \mc{O}_{\preim{f}{x}}^{\oplus r}, d^{\oplus r} \right)$ for every $x\in M$.
\label{nonrelativetorelative}
\end{thm}

\begin{proof}
$\Rightarrow$: It is clear, that the induced holomorphic connection on the fibers for $\nabla^{f^{-1}\mc{G}}$ is trivial.\\
$\Leftarrow$: Throughout this proof, denote by $M_{f,x}$ the fiber of $f$ in $M\times N$ over $x$ and by $V:= \ker{\nabla_f}$ the relative local system of parallel sections.\\
Let $s\colon U \to W$ be a section of $f$ such that $\rst{V}{W}\cong \left( \rst{f}{W}\right)^{-1}\mc{G}$ for some coherent sheaf $\mc{G}$, where $W\subseteq M\times N$ is an open subset and $U:=f\left( W \right)$. Then the restriction maps of $V$ induce a sheaf morphism
\[\phi\colon f_*\left( \rst{V}{U\times N} \right) \to s^{-1}V\cong \mc{G}.\]
The morphism sends a section over $\preim{f}{U'}$ to the germs in the image of $s$. Suppose this morphism was not injective, this would mean that two such sections $t,t'$ have the same germs along the image of $s$. In particular, the section $t-t'$ vanishes on an open set. This implies that for every $x\in U'$ the restriction $\rst{\left( t-t' \right)}{M_{f,x}}\in \ker{\rst{\nabla_f}{M_{f,x}}}$ vanishes on an open subset. However $N$ is connected and sections of  $\ker{\rst{\nabla_f}{M_{f,x}}}$ are locally constant and thus $\rst{\left( t-t' \right)}{M_{f,x}}$ is zero on all of $M_{f,x}$. As the value of a section in the fibers is invariant under pullback, one has
\[\forall \left( x,y \right)\in U'\times N\colon \left[ t-t' \right]_{\left( x,y \right)}=0 \in \mc{F}\otimes_{\coan{U'\times N}} \bb{C}.\]
Thus, it follows that $t=t'$ as $\mc{F}$ is locally free and thus $\phi$ is injective.\\
To prove surjectivity, one needs to show that, given $t\in V_p$, there exists $U\subseteq M$ and a section of $t'$ of $V$ over $U\times N$ such that $t'_p=t$. To do this, one needs to construct a well controlled open cover and then successively extend the locally defined sections.\\
Let $t\in V_p$. First, let $U\subseteq M$ and $W_0$ be such that there exists a representative $s$ on $U\times W_0$. Since $M$ is locally compact, $U$ may be chosen such that $\bar{U}$ is compact. Now, for every $p=(x,y)\in \bar{U}\times N$, choose open sets $U_p\times W_p$ such that
\[\rst{V}{U_p\times W_p} \cong  \rst{\left(f^{-1}\mc{G}\right)}{U_p\times W_p}.\]
In order to cover $\bar{U}\times\left\{y\right\}$ one only needs finitely many of the open sets $U_{\left( x,y \right)}$. Denote these by $\left\{U'_{y,i}\right\}_{i=1}^{a_y}$ and then set $U_{y,i}:= U'_{y,i}\cap  U$. Denote the corresponding open subset $W_p$ by $W_{y,i}$ and set $W_y:=\bigcap_{i=1}^{a_y}W_{y,i}$. It may be assumed that every $U_{y,i}\times W_y$ is connected. Then $\left\{ U_{y,i}\times W_y \right\}_{y\in N,i=1}^{a_y}$ is an open cover of $U\times N$. In particular, $W_y$ is an open cover of $N$. As $N$ is a manifold it follows that there exist countable many $y_j\in N$ such that $W_{y_j}$ is a countable open cover. Furthermore, it may be assumed that $W_{y_j}\cap \left(\bigcup_{i=1}^{j-1} W_{y_i}\right)$ is never empty and $W_{y_1}\cap W_0\neq \emptyset$, by the connectedness of $N$ and reordering the $y_j$. Now, $\left\{ U_{y_j,i}\times W_{y_j} \right\}_{j,i=1}^{\infty,a_{y_j}}$ is a countable cover of $U\times N$ by connected open sets. This cover is nice enough for our purposes.\\
Each open set $U_{y_1,i}\times W_{y_1}$ has a non-empty intersection with $U\times W_0$ by assumption and one has $\rst{V}{U_{y_1,i}\times W_{y_1}}\cong \rst{\left( f^{-1}\mc{G} \right)}{U_{y_1,i}\times W_{y_1}}$. Pick a connected component $A$ of $W_{y_1}\cap W_0$ and extend $s$ from $U_{y_1,j}\times A$ to $U_{y_1,j}\times W_{y_1}$ as $s_{1,j}$.\\
As the intersection $W_{y_1}\cap W_0$ may have multiple connected components, one first needs to verify that $s_{1,j}$ agrees with $s$ on the entire intersection. This follows, just as before, from restricting to fibers, where everything is locally constant with respect to the parallel frames that exist globally along the fibers by assumption. Note that $W_0$ and $W_{y_1,j}$ are connected and thus parallel sections on them are constant.\\
By the exact same argument it follows that
\[\rst{s_{1,j}}{\left( U_{y_1,j}\cap U_{y_1,k}\right) \times W_{y_1}} = \rst{s_{1,k}}{\left( U_{y_1,j}\cap U_{y_1,k} \right)\times W_{y_1}}\]
as these sections agree on $\left( U_{y_1,j}\cap U_{y_1,k} \right)\times \left(W_{y_1}\cap W_0\right)$.\\
By induction one ends up with a section $s_i$ over $U\times \bigcup_{j=1}^i W_{y_j}$ for every $i\in \left\{ 1,2,\dots \right\}$. As $V$ forms a sheaf it follows that one ends up with a section $t'$ on $U\times N$ such that $t'_p=t$.\\
Therefore, the morphism $\phi$ is an isomorphism and $f_*V$ is actually $\coan{M}$-coherent.\\
Recall that the morphism
\[f^{-1}f_*V \to V\]
is given by the restriction of sections. However the preceding argument showed that this is an isomorphism and hence $V\cong f^{-1}f_*V$. The sheaf $f_*V$ is coherent on $M$ and $f^*f_*V \cong \mc{F}$. That $\nabla_f \cong \nabla^{f^{-1}f_*V}$ is clear because their relative local systems of parallel sections are isomorphic.
\end{proof}

\section{Reducedness, torsion and tame connections}
\label{interlude}

For the proof of the Relative Riemann-Hilbert Theorem in this paper, the assumptions of reducedness and torsion-freeness are important. In particular, the notion of \emph{torsion-free} is needed in the case where the stalks of the structure sheaf are not integral domains. This means, the concept of torsion is needed over possibly reducible spaces. Often references restrict themselves to the theory of torsion over locally irreducible spaces as the theory is very well-behaved there.\\
This section contains two important technical results, namely, Proposition~\ref{torsionfreeembed2} and Proposition~\ref{PullConTor}. The first statement shows that even over locally reducible spaces, one can locally embed torsion-free sheaves into a free sheaf. This statement utilizes the normalisation of a complex analytic space and the fact that this statement is true on the normalisation. The second statement explains that pulling back a torsion-free sheaf on a locally trivial morphism with potentially singular fibers to a submersion remains torsion-free if it is equipped with a flat relative connection.\\
The following section starts with a recap of reduced spaces, then moves on to a discussion of torsion-free sheaves and concludes with the torsion-freeness of pull-backs of torsion-free sheaves with flat connections. Moreover, the useful notion of \emph{tame} connections is established and discussed in some settings.

\begin{defn}
Let $\left( M,\mc{A} \right)$ be a $\bb{K}$-ringed space. The space $\left( M,\mc{A} \right)$ is called \emph{reduced} if the point-wise evaluation morphism
\[\mc{A} \to \mathrm{Map}\left( M,\bb{K} \right),\; s \mapsto \left(p\mapsto s\left( p \right)\right)\]
to the residue field is injective.
\end{defn}

\begin{rem}
Let $\left( M,\coan{M} \right)$ be a reduced analytic space. Then it is a standard fact that the set of singular points $\mathrm{sing}\left( M \right)$ is nowhere dense and hence, the set of smooth (or regular) points $M_{\mathrm{reg}}$ is dense. Moreover, it is simple to see that the stalks of the structure sheaf of a reduced space are reduced rings, i.e. contain no nilpotent elements.
\end{rem}

Now the discussion moves on to the definition of torsion in the case of possibly reducible spaces. In a lot of cases, the study of torsion is restricted to the case where the rings are assumed to be integral domains. As the discussion in this paper does not depend on this restriction, some elementary aspects of torsion in this more general situation will be discussed.

\begin{defn}
Let $\left( M,\mc{A} \right)$ be a reduced real or complex analytic space and $\mc{F}$ an $\mc{A}$-module. A section $s\in \mc{F}\left( U \right)$ is called \emph{torsion} if for every $p\in U$ there exists a regular element $r_p\in \coan{M,p}$ such that $r_p\cdot s_p=0$. Recall that a regular element of a ring is an element that is not a zero-divisor. The subsheaf of torsion sections is denoted by $\mc{T}_{\mc{F}}$.\\
If $\mc{T}_{\mc{F}}=0$, the module $\mc{F}$ is called \emph{torsion-free}.
\label{TorFree}
\end{defn}

Observe that dividing out the torsion submodule results in a torsion-free module, as expected.

\begin{lem}
Let $\left( M,\mc{A} \right)$ be a reduced real/complex analytic space and $\mc{F}$ an $\mc{A}$-module. As $\mc{A}$ is a sheaf of commutative rings, it follows that $\mc{T}_{\mc{F}}$ is a submodule of $\mc{F}$.\\
Moreover, the quotient $\mc{F}_{\mathrm{tf}}:=\quo{\mc{F}}{\mc{T}_{\mc{F}}}$ is torsion-free.
\end{lem}

\begin{proof}
Let $s,t\in \mc{F}\left( U \right)$ be two torsion sections and let $p\in U$ with $r_p,r'_p\in\mc{A}_p$ be regular elements such that $r_ps_p=0=r'_pt_p$. Then also $r_p\cdot r'_p$ is a regular element and
\[r_p\cdot r'_p\left( s+t \right)_p=r'_pr_ps_p + r_pr'_pt_p=0,\]
by commutativity. Also by commutativity, $r_p\cdot\left( f\cdot s \right)_p=f_p\cdot r_p\cdot s=0$ for any $f\in \mc{A}\left( U \right)$. Hence, $\mc{T}_{\mc{F}}$ is a submodule.\\
Suppose now that $t\in \quo{\mc{F}}{\mc{T}_{\mc{F}}}\left( U \right)$ is a torsion section. As this question is local, one may assume that $t=\left[ s \right]$ for some section $s\in \mc{F}\left( U \right)$. Then, $t$ being a torsion section implies that for every $p\in U$ there exists a regular element $r_p\in \mc{A}_p$ such that
\[r_pt_p=\left[ r_ps_p \right] =0.\]
This implies that $r_p s_p\in \mc{T}_{\mc{F},p}$. Therefore, for every $p\in U$ there exists a regular element $u_p$ such that $u_pr_ps_p=0$. However $u_pr_p$ is a regular element and hence, $s$ is torsion and thus $t=0$.
\end{proof}

The torsion subsheaf of a coherent sheaf on a reduced complex analytic space is actually itself coherent.

\begin{lem}
Let $\mc{F}$ be a coherent sheaf over a reduced complex analytic space. Then the torsion subsheaf $\mc{T}_{\mc{F}}$ is coherent.
\end{lem}

\begin{proof}
Denote by $0_{\mc{F}}\left[ X \right]$ the gap sheaf of $0$ in $\mc{F}$ with respect to an analytic subset $X$. Note that $\mc{T}_{\mc{F}}= 0_{\mc{F}}\left[ \mathrm{sing}\left( \mc{F} \right) \right]$ because $\mc{F}$ is locally free outside of the nowhere dense subset $\mathrm{sing}\left( \mc{F} \right)$. Gap sheaves on complex analytic spaces of coherent sheaves are coherent by~\cite[Proposition 3.4.]{siu}.
\end{proof}

For the remainder of the section, assume that $\left( M,\coan{M} \right)$ is a reduced complex analytic space, unless stated otherwise. Another technical necessary lemma is that torsion-free sheaves can be locally embedded into free sheaves. First, one considers the case of locally irreducible varieties and then the more general case. The locally irreducible case is a well-known fundamental result:

\begin{lem}
Let $\mc{F}$ be a coherent torsion-free $\coan{M}$-module on a reduced locally irreducible complex analytic space $M$, i.e. for every $p\in M$ the stalk $\coan{M,p}$ is an integral domain. Then for every $p\in M$ there exists an open neighbourhood $U$ of $p$ and a monomorphism $\rst{\mc{F}}{U}\hookrightarrow \coan{M}^{\oplus k}$.
\label{torsionfreeembed}
\end{lem}

\begin{proof}
Denote by $\mc{M}_{M,p}$ the quotient field of $\coan{M,p}$ and one gets a canonical mapping
\[ \phi\colon\mc{F}_p\to \mc{F}_p\otimes_{\coan{M,p}}\mc{M}_{M,p}\cong \mc{M}_{M,p}^{\oplus k}.\]
It follows that this mapping is a monomorphism by torsion-freeness. Let $v_1,\dots v_k$ be a basis of $\mc{M}_{M,p}^{\oplus k}$ and denote by $s_1,\dots,s_l$ a system of generators of $\mc{F}_p$. Then $s_j = \sum_{i=1}^k \frac{a_{ji}}{b_{ji}}v_i$. Set $b:=\prod_{i,j} b_{ji}$. One obtains
\[b\cdot \phi\colon \mc{F}_p \hookrightarrow \coan{M,p}^{\oplus k}.\]
This monomorphism can be extended to an open neighbourhood of $p$ as $\mc{F}$ is coherent.
\end{proof}

Observe that torsion-free sheaves always have restriction morphisms to dense subsets that are injective. This fact will also help with extending Lemma~\ref{torsionfreeembed} to reducible spaces.

\begin{lem}
Let $\mc{F}$ be a torsion-free coherent $\coan{M}$-module. Moreover, suppose that $U\subseteq M$ is an open subset and $V\subseteq U$ is an open dense subset. Then the restriction morphism
\[\mc{F}\left( U \right)\to \mc{F}\left( V \right)\]
is injective.
\label{monorestriction}
\end{lem}

\begin{proof}
Let $0\neq s\in \mc{F}\left( U \right)$ be a section with $\rst{s}{V}=0$ and consider the associated morphism $A\colon \coan{U}\to \rst{\mc{F}}{U}$. Denote by $J$ the kernel of $A$ - note that it is an ideal sheaf. By assumption there exists $p\in U$ such that for all regular elements $r_p\in \coan{U,p}$ one has $r_p\cdot s_p \neq 0$. Moreover, one has
\[\mathrm{supp}\left( \quo{\coan{U}}{J} \right)\subseteq \mathrm{sing}\left( U \right).\]
Hence, by~\cite[Proposition 0.15.]{fischer} it follows that $I^k\subseteq J$, where $I$ denotes the defining ideal of $\mathrm{sing}\left( U \right)$.\footnotemark{}\footnotetext{This fact is due to the complex analytic Nullstellensatz.} By Lemma~\cite[page 98]{grauert} it follows that $J_p$ contains at least one regular element because the analytic subspace defined by $I$ is nowhere dense and thus each stalk of $I$ contains a regular element. This is a contradiction and therefore $s=0$.
\end{proof}

This attribute of torsion-free sheaves has the interesting effect, that one can test whether a morphism of sheaves with torsion-free domain is a monomophism by testing it on a dense subset.

\begin{cor}
Let $\mc{F}$ and $\mc{G}$ be $\coan{M}$-modules. Suppose that $\mc{F}$ is torsion-free and $\psi\colon \mc{F}\to \mc{G}$ is a morphism of $\coan{M}$-modules that is injective on a dense open subset $W$ of $M$. Then, $\psi$ is injective everywhere.
\label{monoondense}
\end{cor}

\begin{proof}
The claim follows from the preceding proposition by considering the following commutative diagram of morphisms and their restrictions
\[
\begin{tikzcd}
\mc{F}\left( U \right) \arrow[r] \arrow[d,hook]& \mc{G}\left( U \right) \arrow[d]\\
\mc{F}\left( V \right) \arrow[r,hook] & \mc{G}\left( V \right)
\end{tikzcd},
\]
where $V:=W\cap U$ and $U$ is an abritrary open subset of $M$.
\end{proof}

These considerations suffice to obtain the desired result. The idea is that every complex analytic space admits a finite morphism from a locally irreducible space, i.e. the normalisation. The pull-back of a torsion-free coherent sheaf to the normalisation embeds into free sheaves locally and, as the morphism is finite, one expects to embed the torsion-free sheaf into the direct image of the pull-back. Before giving the precise argument a useful notion in this context is introduced, namely the \emph{torsion-free pullback}.

\begin{defn}
Also, suppose that $\psi\colon N\to M $ is a complex analytic morphism. Then define
\[\psi^T\mc{F}:=\quo{\psi^*\mc{F}}{\mathrm{torsion}}\]
as the \emph{torsion-free pull-back}.
\label{tensorprodtor}
\end{defn}

\begin{rem}
Properties of the torsion-free pull-back and its interplay with direct images along proper modifications were studied extensively in~\cite{ruppenthal}.
\end{rem}

\begin{prop}
Let $\mc{F}$ be a torsion-free coherent sheaf. Then for every $p\in M$ there exists an open neighbourhood $U$ and a monomorphism
\[\rst{\mc{F}}{U}\hookrightarrow \mc{O}_U^{\oplus k}.\]
\label{torsionfreeembed2}
\end{prop}

\begin{proof}
Denote by $\nu \colon M^{\nu} \to M$ the normalisation of $M$ and consider the torsion-free pull-back $\nu^{T}\mc{F}$ on $M^{\nu}$. Recall that $\nu$ is a finite morphism and denote by $p_i\in M^{\nu}$ the elements in the fiber of $\nu$ over $p$. Moreover, recall that $M^{\nu}$ is locally irreducible because it is normal. One may assume that one has disjoint open subsets $U_i$ around each $p_i$ and $U$ around $p$ such that $\nu^{-1}\left( U \right)=\bigcup_i U_i$ (see e.g.~\cite[p.48]{grauert}). Furthermore, it may be assumed that for each $i$ one has a monomorphism
\[\rst{\nu^{T}\mc{F}}{U_i}\hookrightarrow \coan{U_i}^{\oplus k_i}\]
and since there are only finitely many $p_i$, one can set $k_i=k_j$ for all $i,j$ by simply taking the maximum.
Hence, one has
\[\rst{\nu^{T}\mc{F}}{\nu^{-1}\left( U \right)} \hookrightarrow \coan{\nu^{-1}\left( U \right)}^{\oplus k}\]
and can take the direct image of this morphism, leading to
\[\nu_*\rst{\nu^{T}\mc{F}}{\nu^{-1}\left( U \right)} \hookrightarrow \nu_*\coan{\nu^{-1}\left( U \right)}^{\oplus k}\cong \tilde{\mc{O}}_{U}^{\oplus k}.\]
Here $\tilde{\mc{O}}_U$ denotes the sheaf of weakly holomorphic functions on $U$, i.e. those holomorphic functions on the regular part of $U$ which are locally bounded at the singularities of $U$. Observe that $\rst{\mc{F}}{U}$ naturally embeds into the left-hand side because of Corollary~\ref{monorestriction} and that the right-hand side embeds into $\mc{O}_{U}^{\oplus k}$ by multiplying with a universal denominator that does not vanish on any irreducible component of $U$ (see e.g.~\cite[p.110]{nara}). Hence one obtains
\[\rst{\mc{F}}{U}\hookrightarrow \coan{U}^{\oplus k}.\]
\end{proof}

In the following discussion, it is going to be of the utmost importance that pulling back a torsion-free sheaf with a flat relative connection to a submersion does not produce torsion. This behaviour is elegantly proven utilizing the Relative Riemann-Hilbert Theorem for submersions (Proposition~\ref{PullConTor}). Moreover, with the style of argument used to proving this result one also touches on the definition of \emph{tameness}. It seems that in the singular situation this notion is indispensable for the considerations at hand. The definition of tameness is given first and then the argument showing that flat relative connections on submersion are tame is given.

\begin{defn}
Suppose $f\colon M\to N$ is a locally trivial morphism of complex spaces. Then a $f$-relative connection $\nabla_f$ on $\mc{F}$ is called \emph{tame} if for every section $s\in \mc{F}\left( U \right)$ and for every nowhere dense analytic subset $X\subseteq U$ such that for every $q\in N$ the intersection
\[X\cap f^{-1}\left( \left\{ q \right\} \right)\]
is nowhere dense in $f^{-1}\left( \left\{ q \right\} \right)$ one has that the following implication holds:
\[\rst{\nabla_fs}{U\setminus X}=0 \implies \nabla_f s=0.\]
\label{tamedefinition}
\end{defn}

This definition might seem artificial at first, however the following Proposition shows how flat relative connections on submersions naturally behave like this.

\begin{prop}
Let $f\colon M\to N$ be a submersion of complex analytic spaces. Suppose that $\mc{F}$ is a coherent $\coan{M}$-module with a flat $f$-relative connection $\nabla$. Suppose that $s\in \mc{F}\left( U \right)$ is a section such that $\mathrm{supp}\left( s \right)\cap f^{-1}\left( \left\{ q \right\} \right)$ is nowhere dense in $U\cap f^{-1}\left( \left\{ q \right\} \right)$ for all $q\in f\left( U \right)\subseteq N$. Then $s$ is zero.\\
In particular, the connection $\nabla$ is tame and moreover, if $\mc{F}$ is torsion-free on a dense open subset $W\subseteq M$ such that $W\cap f^{-1}\left( \left\{ q \right\} \right)$ is dense in $f^{-1}\left( \left\{ q \right\} \right)$ for every $q\in N$, then $\mc{F}$ is torsion-free.\\
\label{PullConTor}
\end{prop}

\begin{proof}
Consider the analytic subset $X:=\mathrm{supp}\left( s \right)\subseteq M$ and the gap sheaf $0_{\mc{F}}\left[ X \right]$ of $0$ in $\mc{F}$ with respect to $X$. Then $0_{\mc{F}}\left[ X \right]$ is coherent by~\cite[Proposition 3.4.]{siu}. Moreover, note that the relative $1$-forms $\Omega^1_{f}\left( \coan{M} \right)$ are locally free and hence for $t\in 0_{\mc{F}}\left[ X \right]$ it is clear that
\[\nabla t \in 0_{\mc{F}\otimes \Omega^1_{f}\left( \coan{M} \right)}\left[ X \right]=0_{\mc{F}}\left[ X \right]\otimes \Omega^1_{f}\left( \coan{M} \right).\]
Hence, the coherent subsheaf $0_{\mc{F}}\left[ X \right]\subseteq \mc{F}$ is preserved by the flat connection and inherits an induced flat connection simply by restriction. Then by Theorem~\ref{relriemanhilbert} it follows that $0_{\mc{F}}\left[ X \right]$ is locally a pull-back sheaf from $N$, but it is zero on a dense subset of the fibers and hence zero everywhere. Hence, $s\in 0_{\mc{F}}\left[ X \right]=0$ is zero as well.\\
The statement about tameness follows immediately from the first claim by noting that $\Omega^1_f\left( \coan{M} \right)$ is locally free and the statement about the torsion-freeness follows by observing that any torsion section would have support that intersected with the fibers is nowhere dense.
\end{proof}

Later on and for the tameness of the canonical connection associated to a torsion-free relative local system one will need the following technical statement about the interplay of sections of the pull-back of a torsion-free sheaves and sections before pull-back.

\begin{prop}
Let $f\colon M\to N$ be a flat and surjective morphism of reduced complex analytic spaces and suppose $\iota\colon \mc{F}\to \coan{N}^{\oplus r}$ is an injective morphism of coherent sheaves. If $s\in \coan{N}^{\oplus r}$ and $u\in f^*\mc{F}$ are such that $f^*s= f^*\iota\left( u \right)$ then there exists $t\in \mc{F}$ such that $f^*t= u$.
\label{pullbackSec}
\end{prop}

\begin{proof}
Consider the module $\mc{L}$ generated by the section $s$ in $\coan{M}^{\oplus r}$ and denote by $Q:= \mc{F}\cap \mc{L}$ the intersection of $\mc{F}$ and $\mc{L}$ in $\coan{M}^{\oplus r}$. Recall that the intersection of sheaves of modules may be defined by the following exact sequence
\[
\begin{tikzcd}
0 \arrow[r] & Q \arrow[r,"a"] &\mc{L} \arrow[r] & \quo{\coan{M}^{\oplus r}}{\mc{F}}
\end{tikzcd}
\]
and the pull-back of this sequence reads
\[
\begin{tikzcd}
 0\arrow[r] & f^*Q \arrow[r,"f^*a"] &f^*\mc{L} \arrow[r] & f^*\left(\quo{\coan{M}^{\oplus r}}{\mc{F}}\right)=\quo{f^{*}\coan{M}^{\oplus r}}{f^{*}\mc{F}} 
\end{tikzcd}.
\]
Note that the pull-back sequence remains exact because $f$ is assumed to be flat. However, because $f^*s = \left( f^*\iota \right)\left( u \right)\in f^*\mc{F}$ it follows that $f^*a$ is an epimorphism and since
\[0=\coker{f^*a}\cong f^*\coker{a}\]
holds, one obtains that $a$ is an epimorphism as well. Note that in the last deduction it was used that $f$ is surjective, that fiber rank of a sheaf is invariant under pull-back and that $N$ is reduced, so that a coherent sheaf with vanishing fiber rank everywhere is zero. Hence, $s\in \im{\iota}$ and there exists $t\in \mc{F}$ such that $\iota\left( t \right)=s$. However,
\[f^*\iota \left( f^*t \right)=f^*s=f^*\iota\left( u \right)\]
hence, $u=f^*t$ because $f^*\iota$ remains a monomorphism by flatnes by flatness by flatness by flatness by flatness by flatness by flatness by flatness.
\end{proof}

Tameness of flat relative connections can now even be shown in the more general situation of canonical connections on torsion-free sheaves over reduced locally trivial morphisms, as the following lemma shows.

\begin{thm}
Let $f\colon M \to N$ be a reduced locally trivial complex analytic morphism and $V$ a torsion-free relative local system. Then $\ker{\nabla^V}=V$. Moreover, $\nabla^{V}$ is tame.
\label{kernelofcanonical}
\end{thm}

\begin{proof}
Note that $V\subseteq \ker{\nabla^V}$ and away from the singularities of the fibers of $f$ this inclusion is an equality. So suppose $p\in M$ and $U\subseteq M$ around $p$ open and small enough such that $U\cong W\times X$, with $W$ and $X$ connected. Additionaly, suppose that $\rst{V}{U}\cong f^{-1}\left( \mc{G} \right)$, for some torsion-free $\coan{N}$-coherent sheaf $\iota\colon \mc{G}\hookrightarrow \coan{N}^{\oplus r}$. Note that because $f$ is flat one still has $f^*\iota\colon f^{*}\mc{G}\hookrightarrow f^*\coan{N}^{\oplus r}$ after pull-back along $f$. Moreover, note that $\nabla^{f^{-1}\coan{N}^{\oplus r}}=d_f$ is of course tame, as the kernel of $d_f$ is $f^{-1}\coan{N}$.\\
Suppose that $u\in \left(f^{*}\mc{G}\right)\left( U \right)$ is a section and $Y\subseteq U$ is a nowhere dense analytic subset such that $Y$ intersected with the fibers of $f$ is nowhere dense in the fibers of $f$. Moreover, assume that $\rst{\nabla^{V}u}{U\setminus Y}=0$. Then it follows that $f^{*}\iota\left( u \right)$ is in the kernel of $d_f$, i.e. in $\left(f^{-1}\coan{N}^{\oplus r}\right)\left( W\times X \right)=\coan{N}^{\oplus r}\left( W \right)$, because $X$ is connected. Hence, there exists a section $s\in \left(\coan{N}^{\oplus r}\right)\left( W \right)$ such that $f^*\iota\left( u \right)=f^*s$. Then by Proposition~\ref{pullbackSec} it follows that there exists $t\in \mc{G}\left( U \right)$ such that $f^*t =s$. Thus, $s\in \left( f^{-1}\mc{G} \right)\left( U \right)=V\left( U \right)\subseteq \ker{\nabla^V}$ and thus, $\nabla^V$ is tame. This argument also shows that $\ker{\nabla}^V = V$, as one could have started the argument with just any parallel $u$ and concluded that $u\in V$.
\end{proof}

Under suitable conditions cokernels and kernel of morphisms preserving connections admit themselves a compatible connection in the following way.

\begin{prop}
Let $f\colon M\to N$ be a reduced locally trivial morphism of complex analytic spaces.
\begin{enumerate}[(i)]
\item The kernel of $d_f$ on $\coan{M}^{\oplus r}$ is the relative local system $f^{-1}\coan{N}^{\oplus r}$.
\item Let $A\colon \left( \coan{M}^{\oplus r_1 },d_f \right)\to \left( \coan{M}^{\oplus r_2},d_f \right)$ be a morphism of connections. Suppose the cokernel of $A$ is torsion-free. Then the cokernel connection $\nabla$ on the cokernel of $A$ is such that $\coker{A}\cong \ker{\nabla}\otimes_{f^{-1}\coan{N}} \coan{M}$ and $\ker{\nabla}$ is a $f$-relative local system.
\item Suppose $M$ is a Stein complex analytic space and let $B\colon \left( \mc{F},\nabla' \right)\to \left( \mc{G},\nabla'' \right)$ be an surjective morphism of connections on coherent sheaves. Then the kernel $\iota\colon \ker{B}\to \mc{F}$ can be equipped with a relative connection such that $\iota$ becomes a morphism of connections.
\end{enumerate}
\label{MorphOfConnec}
\end{prop}

\begin{proof}
\begin{enumerate}[(i)]
\item This is clear from the reducedness and the fact that functions in the kernel of $d_f$ are constant along the fibers of $f$.
\item One obtains a morphism of sheaves
\[A'\colon f^{-1}\coan{N}^{\oplus r_1}\to f^{-1}\coan{N}^{\oplus r_2}\]
such that
\[A=A'\otimes \id \colon f^{-1}\coan{N}^{\oplus r_1}\otimes_{f^{-1}\coan{N}} \coan{M}\to f^{-1}\coan{N}^{\oplus r_2}\otimes_{f^{-1}\coan{N}}\coan{M}.\]
However, clearly $\coker{A'}$ is a relative local system such that
\[\coker{A'}\otimes_{f^{-1}\coan{N}} \coan{M}\cong \coker{A}.\]
As such the $\coker{A}$ can be equipped with the canonical connection associated to a relative local system. Moreover, the epimorphism to the cokernel will be a morphism of connections by construction. By Theorem~\ref{kernelofcanonical} it follows that $\ker{\nabla}= \coker{A'}$.
\item Consider the following diagram
\[
\begin{tikzcd}
 & \coan{M}^{\oplus r_2} \arrow[d,"C"] & & & \\
 & \coan{M}^{\oplus r_1} \arrow[d,"D",two heads] \arrow[dr,"F"] & & & \\
0 \arrow[r] & \ker{B} \arrow[r,"\iota"] & \left( \mc{F},\nabla' \right) \arrow[r,"B"] & \left( \mc{G},\nabla'' \right) \arrow[r] & 0
\end{tikzcd},
\]
where $F=\iota\circ D$. One can obtain a connection of $\ker{B}$ by equipping it with a cokernel connection from the morphism $C$. To that end, let $e_i$ be a frame of $\coan{M}^{\oplus r_1}$ and note that
\[B\otimes\id\left(\nabla'\left( F\left( e_i \right) \right)\right)= \nabla'' \left( B\left( F\left( e_i \right) \right) \right)=0.\]
Hence,
\[\nabla'\left( F\left( e_i \right) \right)\in \ker{B\otimes \id}= \im{\iota\otimes \id}\subseteq \mc{F}\otimes_{\coan{M}}\Omega^1_f\left( \coan{M} \right)\]
and observe that $F\otimes \id$ surjects onto $\im{\iota\otimes \id}$ and since $M$ is Stein there exists $\alpha_i\in \coan{M}^{\oplus r_1}\otimes \Omega^1_f\left( \coan{M} \right)$ such that
\[F\otimes \id \left( \alpha_i \right)= \nabla'\left( F\left( e_i \right) \right).\]
As $\coan{M}^{\oplus r_1}$ is free, one can define a connection on it by $\nabla_1\left( e_i \right):=\alpha_i$ and then imposing the $d_f$-Leibniz rule. Moreover, by definition this connection turns $F$ into a morphism of connections.\\
The same construction allows one now to define a connection $\nabla_2$ on $\coan{M}^{\oplus r_2}$ such that $C$ turns into a morphism of connections. By Lemma~\ref{connecpreserve} there then exists the cokernel connection $\nabla_B$ on $\ker{B}=\coker{C}$.\\
It remains to argue that $\iota$ becomes a morphism of connections with respect to $\nabla_B$ and $\nabla'$. To this end consider the following
\begin{align*}
\nabla'\left( \iota\left( D\left( e_i \right) \right) \right)&= F\otimes\id \left( \nabla_1\left( e_i \right) \right)\\
&=\iota\otimes \id \left( D\otimes \id \left( \nabla_1\left( e_i \right) \right) \right)\\
&=\iota\otimes\id\left( \nabla_B\left( D\left( e_i \right) \right) \right).
\end{align*}
But $D\left( e_i \right)$ are generators of $\ker{B}$ and hence $\iota$ is a morphism of connections.
\end{enumerate}
\end{proof}

\begin{rem}
Notice that while a cokernel connection is always uniquely defined, this is a priori not the case for the kernel connection as tensoring with the sheaf of relative $1$-forms is not necessarily exact and hence does not preserve injectivity.
\end{rem}

With these facts one can now make a useful observation. Namely, if a connection on a submersion admits parallel sections on a closed \emph{singular} complex analytic subset then it admits parallel extension of such sections to an entire open neighbourhood.

\begin{cor}
Let $f\colon N\times M \to N$ be a submersion of complex analytic spaces, i.e. suppose that $M$ is smooth. Moreover, let $\iota'\colon Y\hookrightarrow M$ be a reduced complex analytic subspace and let $\left( \mc{F},\nabla_f \right)$ be a torsion-free $\coan{N\times M}$-coherent sheaf with flat $f$-relative connection. Denote by $\iota\colon N\times Y\to N\times M$ the inclusion. Then one has that
\[\ker{\iota^*\nabla_f}=\iota^{-1}\ker{\nabla_f}\]
and as complex analytic spaces are paracompact one has
\[\Gamma\left( N\times Y, \iota^{-1}\ker{\nabla_f} \right) = \varinjlim_{U\supseteq N\times Y} \Gamma\left( U,\ker{\nabla_f} \right).\]
In particular, parallel section of $\iota^*\nabla_f$ can be extended to an open neighbourhood of $N\times Y$ in $N\times M$ and the connection $\iota^*\nabla_f$ is tame.
\label{ExtPara}
\end{cor}

\begin{proof}
The statement regarding the kernels is a simple application of Theorem~\ref{kernelofcanonical}. As $\left( \mc{F},\nabla_f \right)\cong \left( \ker{\nabla_f}\otimes_{f^{-1}\coan{N}}\coan{N\times M}, \id\otimes d_f \right)$ and hence
\[\left( \iota^*\mc{F},\iota^*\nabla_f \right)\cong \left( \iota^{-1}\ker{\nabla_f}\otimes_{\iota^{-1}f^{-1}\coan{N}}\coan{N\times Y}, \id\otimes d_f \right).\]
However, $\iota^{-1}\ker{\nabla_f}$ is a torsion-free relative local system, hence
\[\ker{\iota^*\nabla_f}= \iota^{-1}\ker{\nabla_f}\]
and the tameness of $\iota^*\nabla_f$ follow from Theorem~\ref{kernelofcanonical}.\\
For the statement about the sections of the inverse image sheaf in paracompact spaces see e.g.~\cite[page 66, Theorem 9.5]{bredon}.
\end{proof}

This section is now concluded with a simple observation. Namely, that coherent sheaves on reduced spaces with an absolute connection are necessarily locally free.

\begin{lem}
Let $M$ be a reduced complex analytic space and let $\mc{F}$ be a coherent sheaf with connection $\nabla$. Then $\mc{F}$ is locally free.
\label{ConnecLocFree}
\end{lem}

\begin{proof}
First assume that $\nabla$ is flat. Let $\psi \colon \tilde{M}\to M$ be a desingularization and note that $\psi^*\mc{F}$ is locally free as it is locally determined by a local system of vector spaces. However, fiber rank of a sheaf of modules is invariant under pull-back and $\psi$ is surjective. Moreover, $M$ is reduced and thus $\mc{F}$ is locally free.\\
If $\nabla$ is not flat then locally on $\tilde{M}$ any to points $p,q$ in a small open set can be connected by a smooth complex curve $\gamma_{p,q}$. The pull-back $\gamma_{p,q}^*\psi^*\mc{F}$ will be locally free as any connection on a smooth curve is flat. Hence, the fiber rank of $\psi^*\mc{F}$ is locally constant and therefore so is the fiber rank of $\mc{F}$. Hence, $\mc{F}$ is locally free.
\end{proof}

\section{Example of a flat non-tame connection}
\label{NonTame}

In this section, it is demonstrated -- by example -- that a flat connection on a singular space is not necessarily tame. Showing that the tameness condition can in general not be dropped on an arbitrary singular space when considering the Riemann-Hilbert correspondence. One can obtain such an example by constructing a closed non-zero torsion $1$-form $\alpha$ on a complex variety and the connection $\nabla = d + \alpha$ will be flat but not tame. The precise argument and example are laid out in the following.\\
Throughout this section denote by $f\in \coan{\bb{C}^2}$ the polynomial $f=x^4+y^4x+y^5$ on $\bb{C}^2$ and by $C\subseteq \bb{C}^2$ the curve defined by $f=0$. Moreover, denote by $J$ the ideal generated by $f$ in $\coan{\bb{C}^2}$. Then the following fact can be proven:

\begin{thm}
Let $\alpha= (x^4y+\frac{y^5x}{5} + \frac{y^6}{6})dx$ be a $1$-form on $C$. Then there exists a holomorphic function $G$ on $C$ such that 
\[\alpha -dG\]
is a non-zero closed torsion $1$-form on $C$, i.e.
\[\rst{\alpha}{C_{\mathrm{reg}}}=d \rst{G}{C_{\mathrm{reg}}} \text{ but } dG\neq \alpha.\]
\label{NonExact}
\end{thm}

We establish this example of a non-zero closed torsion $1$-form in several steps. The argument relies on the fact that in~\cite{reiffen} it was proven that $C$ has a non-exact de Rham sequence in $0\in C$ in degree $1$. The argument for non-exactness given in~\cite{reiffen} boils down to the following claim:

\begin{claim}
The equation

\[f= \pard{\left(f\cdot A\right)}{x} + \pard{\left(f\cdot B\right)}{y}\]

does not admit any solutions $A,B\in \coan{\bb{C}^2}$ in any neighbourhood of $0\in \bb{C}^2$.
\end{claim} 

\begin{proof}
Note that any such solutions would yield a vector field $X:=A \pard{}{x} + B \pard{}{y}$ that preserves the ideal $J$. Therefore $X$ restricts to a vector field on an open neighbourhood of $0$ in $C$. As such $X$ needs to vanish at $0\in C$, because $0$ is a singular point of $C$ (see e.g.~\cite[Corollary 3.3]{rossi}). Hence, $A\left( 0 \right)=B\left( 0 \right)=0$. Let now
\[A=\sum_{k\geq 1} A_k \text{ and } B=\sum_{k\geq 1} B_k\]
be the homogeneous decomposition of $A$ and $B$ in a sufficiently small open neighbourhood of $0$. Further introduce
\[A_1=A_{11}\cdot x + A_{12} \cdot y \text{ and } B_1=B_{11}\cdot x + B_{12} \cdot y.\]
Comparing different order terms in the equation then yields
\begin{align*}
x^4 =& \pard{(x^4 \cdot A_1)}{x} + \pard{\left( x^4\cdot B_1 \right)}{y}\\
 =& 4\cdot x^3 (A_{11}\cdot x + A_{12}\cdot y) + x^4\cdot A_{11} + x^4\cdot B_{12}\\
 =& 5 \cdot x^4 \cdot A_{11} + 4 \cdot x^3 \cdot y \cdot A_{12} + x^4\cdot B_{12}\\
y^4\cdot x + y^5 =& \pard{\left(A_1\left( y^5+y^4\cdot x \right)\right)}{x} + \pard{\left(B_1\left( y^5+y^4\cdot x \right)\right)}{y} + \pard{\left(x^4\cdot A_2\right)}{x} + \pard{\left(x^4\cdot B_2\right)}{y}\\
=& A_{11}\cdot y^5 + 2 \cdot A_{11}\cdot y^4\cdot x + A_{12} \cdot y^5 + B_{12} \left( y^5 + y^4\cdot x \right)\\
& + \left( B_{11} \cdot x + B_{12}\cdot y \right) \left( 5\cdot y^4 + 4\cdot y^3\cdot x \right)\\
& + 4\cdot x^3\cdot  A_2 + x^4 \cdot \pard{A_2}{x} + x^4 \cdot \pard{B_2}{y}
\end{align*}

Note that the first equation implies that $A_{12}=0$ and the second implies $B_{11}=0$. Collecting similar terms then yields the following system of linear equations:

\begin{align*}
5\cdot A_{11}  +  B_{12} &= 1\\
A_{11}  +  6\cdot B_{12} &= 1\\
2 \cdot A_{11} +  5\cdot B_{12} &= 1
\end{align*}

The associated determinant

\[
\mathrm{det} \begin{pmatrix}
5 & 1 & 1\\
1 & 6 & 1\\
2 & 5 & 1
\end{pmatrix} = 5\cdot \left( 6-5 \right) - 1\cdot \left( 1-2 \right) + 1 \cdot \left( 5-12 \right) = 5 + 1 -7 = -1\neq 0 
\]

is not equal to zero and hence there does not exist a solution to this system of linear equations and therefore also no solutions $A$ and $B$ of the partial differential equation in any neighbourhood of $0$.
\end{proof}

Immediately from the non-existence of such solutions, one can infer that the $1$-form $\alpha$ is closed but not locally exact around $0$ on $C$.

\begin{claim}
The $1$-form $\alpha$ is closed on $C$ and not exact in any open neighbourhood of $0\in C$. In particular, $\alpha\neq 0$ around $0\in C$.
\end{claim}

\begin{proof}
The closedness of $\alpha$ on $C$ is clear because,
\[d\alpha = \left( x^4+y^4x+y^5 \right)dy\wedge dx= f dy\wedge dx.\]
Now the partial differential equation:
\begin{align}
f&= \pard{\left(fA\right)}{x} + \pard{\left(fB\right)}{y}
\end{align}
for $f$ does not have solutions $A$ and $B$. This shows that $\alpha$ is not exact around $0\in C$. To see this, suppose there exists $H\in \coan{\bb{C}^2}$ such that
\[\alpha - dH = h\cdot df+ f\beta.\]
Then one would have that $\alpha - d(H + h\cdot f)= f(\beta - dh)$. That is there would exist a holomorphic function $H'$ and a $1$-form $\beta'$ such that
\[\alpha -dH' = f \beta'.\]
This equality in turn implies
\[-fdx\wedge dy = d\alpha = d(f\beta')\]
or in coordinates
\[-f = \pard{\left(f\beta'_2\right)}{x} - \pard{\left(f\beta'_1\right)}{y}.\]
However, then $-\beta'_2$ and $\beta'_1$ would be a solution of (1), but no such solution exists as previously established. Hence, $\alpha$ is not exact in any open neighbourhood of $0\in C$.
\end{proof}

Now we establish the concrete form of a desingularization of $C$ around $0$. The aim in doing this is to be able to identify a universal denominator around $0$ on $C$ and therefore better understand the strongly holomorphic functions on $C$ as a subalgebra of the desingularization.

\begin{claim}
Let $U\subset \bb{C}$ be a sufficiently small open neighbourhood of $0\in \bb{C}$. Then the morphism
\[\psi\colon U \to C, \; t\mapsto \left( - \frac{t^5}{1+t}, - \frac{t^4}{1+t} \right)\]
is a desingularization of $C$ around $0$.
\end{claim}

\begin{proof}
Notice that $\psi$ actually factors through $C$, i.e.
\begin{align*}
f\circ \psi &= \frac{t^{20}}{(1+t)^4} - \frac{t^{16}}{(1+t)^4} \cdot \frac{t^5}{(1+t)}-\frac{t^{20}}{(1+t)^5}\\
 &= \frac{t^{20}(1+t) - t^{21}-t^{20}}{(1+t)^5}\\
 &=0.
\end{align*}
Moreover, $\psi$ is injective as
\begin{align*}
& & \left( - \frac{t^5}{1+t}, -\frac{t^4}{1+t} \right) &= \left( -\frac{t'^5}{1+t'},-\frac{t'^4}{1+t'} \right)\\
& \iff & t\cdot\frac{t^4}{1+t} = t'\cdot \frac{t'^4}{1+t'} \text{ } &\text{and } \frac{t^4}{1+t} = \frac{t'^4}{1+t'}\\
& \iff & t &= t',
\end{align*}
for $t,t'\neq 0$ and only $t=0$ maps to $0$. Thus, outside of the singular set of $C$, it follows that $\psi$ is an isomorphism and it follows that $\psi$ is a desingularization of $C$ around $0$.
\end{proof}

Notice that of course with this $\psi$ one can now always view the stalk $\coan{C,0}$ as the subalgebra $\bb{C}\left[\frac{t^5}{1+t}, \frac{t^4}{1+t}\right]$ of convergent power series in $x\circ \psi$ and $y\circ \psi$ in $\coan{\bb{C},0}=\bb{C}[t]$.

\begin{claim}
The element $h:=t^4= -(x+y)$ is such that $h^3$ is a universal denominator around $0$. That is, for every holomorphic function $g$ on $U$ there exists $g'$ on $C$ such that $h^3\cdot g=g'\circ \psi$.\\
In particular, it follows that any holomorphic function $H$ on $U$ around $0$ of the form
\[H= H(0)+ \sum_{i\geq 12} H_i\cdot t^i\]
is strongly holomorphic on $C$, where $H_i\in \bb{C}$.
\end{claim}

\begin{proof}
Let us first verify the equation for $h$, i.e.
\[-(x+y)= \frac{t^5}{1+t}+\frac{t^4}{1+t}=t^4\cdot \frac{t+1}{1+t}=t^4.\]
Notice that every holomorphic function $g$ on $U$ can be written as
\[g=g\left( 0 \right)+\sum_{i\geq 1} g_i\cdot \left( \frac{x}{y} \right)^i\]
around $0\in U$, since $\frac{x}{y}=t$. From the equation $-\left( x+y \right)=t^4=\left(\frac{x}{y}\right)^4$ one knows that $t^4$ is holomorphic on $C$ and one may rewrite $g$ such that
\[g= g\left( 0 \right) + \sum_{i=1}^3 g_i \left(\frac{x}{y} \right)^i - (x+y)\sum_{i=0}^3 g_{i+4} \left( \frac{x}{y} \right)^i + (x+y)^2\sum_{i=0}^3 g_{i+8} \left( \frac{x}{y} \right)^i \dots.\]
Therefore, one only has to verify that
\[h^3\cdot \left(\frac{x}{y}\right)^i\]
is holomorphic on $C$ for $i=1,2,3$. One can do this by verifying that
\[x^3\cdot \left(\frac{x}{y}\right)^i, x^2\cdot y \cdot \left(\frac{x}{y}\right)^i, x\cdot y^2 \cdot \left(\frac{x}{y}\right)^i, y^3\cdot \left(\frac{x}{y}\right)^i\]
are stronlgy holomorphic on $C$ around $0$ for $i=1,2,3$. The following calculations verify this:
\begin{align*}
\frac{x}{y}\cdot x^3 &= \frac{x^4}{y^4}\cdot y^3 = -(x+y)\cdot y^3\\
\frac{x}{y} \cdot x^2\cdot y &= x^3\\
\frac{x}{y} \cdot x \cdot y^2 &= x^2\cdot y\\
\frac{x}{y} \cdot y^3 &= x\cdot y^2\\
 & \\
\frac{x^2}{y^2} \cdot x^3 &= \frac{x^5}{y^2} = -(x+y)\cdot x \cdot y^2\\
\frac{x^2}{y^2} \cdot x^2\cdot y &= \frac{x^4}{y} = -(x+y)\cdot y^3\\
\frac{x^2}{y^2} \cdot x \cdot y^2 &=x^3\\
\frac{x^2}{y^2} \cdot y^3 &= x^2\cdot y\\
 & \\
\frac{x^3}{y^3} \cdot x^3 &= \frac{x^6}{y^3} = -(x+y)\cdot x^2\cdot y\\
\frac{x^3}{y^3} \cdot x^2\cdot y&= \frac{x^5}{y^2} = -(x+y)\cdot x \cdot y^2\\
\frac{x^3}{y^3} \cdot x\cdot y^2 &= \frac{x^4}{y} = -(x+y)\cdot y^3 \\
\frac{x^3}{y^3} \cdot y^3 &= x^3.
\end{align*}
The claim now follows from the fact that $h^3=-x^3-3x^2\cdot y - 3 x\cdot y^2 - y^3$.
\end{proof}

\begin{claim}
A holomorphic function $G$ on $U$ around $0$ such that
\[dG = \psi^*\alpha\]
is strongly holomorphic on $C$.
\end{claim}

\begin{proof}
It is clear that for a sufficiently small neighbourhood of $0\in U$ such a holomorphic function $G$ exists. Moreover, $G$ must be such that
\begin{align*}
\pard{G}{t} &= \left[ \left( \frac{t^5}{1+t} \right)^4\cdot \frac{t^4}{1+t}- \left( \frac{t^4}{1+t} \right)^5\cdot \frac{t^5}{1+t} \cdot\frac{1}{5} - \left( \frac{t^4}{1+t} \right)^6\cdot \frac{1}{6} \right]\pard{}{t}\left( \frac{t^5}{1+t} \right)\\
 &= \frac{ t^{24} + t^{25} -t^{25} \cdot\frac{1}{5} - t^{24} \cdot\frac{1}{6}}{(1+t)^6}\cdot \left( \frac{5t^4}{1+t} - \frac{t^5}{(1+t)^2} \right)\\
&= \frac{\tilde{G}}{(1+t)^8},
\end{align*}
for some appropriate $\tilde{G}$. Notice that $\tilde{G}$ is at least such that
\[\tilde{G}=\sum_{i\geq 24} \tilde{G}_i t^i,\]
for some $\tilde{G}_i\in \bb{C}$, and the series expansion of $\frac{1}{1+t}$ around $0$ reads:
\[\frac{1}{1+t}= \sum_{i=0}^{\infty} \left( -1 \right)^i t^i.\]
Hence, one has that at the very least $G$ is such that
\[\pard{G}{t}= \sum_{i\geq 25} i G_i t^{i-1}\]
and hence $G$ is of the form
\[G= G\left( 0 \right) + \sum_{i\geq 25} G_i t^i,\]
for some $G_i\in \bb{C}$. Therefore $G$ is strongly holomorphic.
\end{proof}

Now one can conclude Theorem~\ref{NonExact} namely:

\begin{claim}
The $1$-form $\alpha - dG$ on $C$ is non-zero closed and torsion around $0$.
\end{claim}

\begin{proof}
The form $\alpha - dG$ is non-zero because $\alpha$ is not exact around $0$. It is closed as it is the difference of two closed $1$-forms. Finally it is torsion, as $\psi$ is an isomorphism away from $0\in C$ and there one has $dG=\psi^*\alpha$.
\end{proof}

One is now able to give an example of a flat non-tame connection.

\begin{thm}
The connection $\nabla:= d + \alpha$ on $\coan{C}$ is flat but $\ker{\nabla}$ is not a local system around $0\in C$. Moreover, $\nabla$ is not tame.
\end{thm}

\begin{proof}
Flatness is clear, as $\alpha$ is closed. Suppose that around $0\in C$ there was a parallel section $s$, i.e.
\[\nabla s = ds + s\alpha =0.\]
Then $s$ is necessarily non-vanishing in $0\in C$ and therefore in a sufficiently small neighbourhood of $0\in C$ one has
\[\alpha = - d \log\left( s \right),\]
which contradicts $\alpha $ not being exact. Hence, $\ker{\nabla}$ is not a local system.\\
Moreover, note that
\[\nabla(\exp(-G))=-\exp\left( -G \right)\cdot dG +\exp\left( -G \right)\cdot \alpha= \exp\left( -G \right)\left( \alpha - dG \right)\neq 0\]
is a non-vanishing torsion $1$-form as established earlier -- note here that of course $\exp\left( -G \right)$ is a unit element.
\end{proof}

This example shows that on $C$ the flat connection $\nabla:=d+ \alpha$ on $\coan{C}$ does not have a local system of parallel sections, i.e. the sheaf $\ker{\nabla}$ is not locally constant. The kernel sheaf has rank $1$ away from $0\in C$ but rank $0$ at $0\in C$. A completely new phenomenon on singular complex analytic spaces, as such an example can never exists on a complex manifold. However, the difference between $\nabla$ and the connection $d+dG$ is only a closed torsion $1$-form.

\section{Riemann-Hilbert correspondence with singular fibers}
\label{SingFib}

The following section applies the collected material to the Riemann-Hilbert correspondence with singular fibers. As the upcoming discussion is going to be somewhat technical it is presented in multiple subsections. Subsection~\ref{WeakHoloMaxPropMod} introduces and discusses the concept of weak holomorphicity of a function defined on the regular part of a reduced complex analytic space. A firm grasp of this concept is necessary because one is able to prove that weakly holomorphic solutions for the Riemann-Hilbert correspondence always exist. In order to make the following discussion accessible, one finds an outline of the overall strategy in subsection~\ref{Strat}. Furthermore, one finds a discussion of the philosophy that led to the development of the current proofs in that section as well. Subsection~\ref{weakSol} then establishes the existence of weakly holomorphic local solutions in the general situation of torsion-free sheaves over the total space of a locally trivial morphism. Then in subsection~\ref{RedAbs} one can show that the obtained solutions are precisely holomorphic if and only if they are holomorphic when restricted to each fiber of the locally trivial morphism (Proposition~\ref{RelStrong} and Theorem~\ref{RelToAbs}). This procedure is called the \emph{reduction to the absolute case}. From this reduction it follows that the Riemann-Hilbert correspondence for torsion-free sheaves still holds on locally trivial morphisms with \emph{maximal} fibers (Definition~\ref{Normal} and Theorem~\ref{MaxRH}). The Riemann-Hilbert correspondence in the singular case can be reduced even further to the case of complex analytic curves. This reduction is discussed in subsection~\ref{RedToCurve} and the precise statement is proven in Proposition~\ref{RedToCurveGen} and Theorem~\ref{RedToCurveGenCor}. This further reduction to the curve case allows one to also prove the Riemann-Hilbert correspondence in the case where the fibers of the morphism are homogeneous complex analytic subspaces of $\bb{C}^n$. The argument relies on first proving that the absolute case holds on homogeneous curves (i.e. the union of finitely many straight lines) in Proposition~\ref{HomCurve}. One then obtains the general homogeneous case by Proposition~\ref{RedCurve} and Theorem~\ref{HomCor}.

\subsection{Weak holomorphicity, maximalization and proper modifications}
\label{WeakHoloMaxPropMod}

A brief overview of well-behaved meromorphic functions on reduced complex analytic spaces is given in this subsection. The main takeaway here is that singularities in the underlying space allow for the existence of locally bounded and even continuous meromorphic functions that are \emph{not} globally holomorphic. Such a phenomenon can of course not be observed on complex manifolds. In relation to this occurrence the notions of \emph{normal} and \emph{maximal} (Definition~\ref{Normal}) spaces are introduced. One can then observe that locally bounded meromorphic functions have a well-behaved graph that completely determines the function (Proposition~\ref{WeakGraph}). The section concludes by showing that holomorphic functions on the domain of definition of a proper modification can always been viewed as a locally bounded weakly holomorphic function on the image of the proper modification (Proposition~\ref{PropMod}).

\begin{defn}
Let $M$ be a reduced complex analytic space. A \emph{weakly holomorphic function (on $M$)} is a holomorphic function $f\colon M_{\mathrm{reg}}\to \bb{C}^r$ such that each component $f_i$ is locally bounded on $M$.\\
A \emph{continuous weakly holomorphic function (on $M$)} is a continuous function $g\colon M\to \bb{C}^r$ such that $\rst{g}{M_{\mathrm{reg}}}$ is holomorphic.\\
It is clear that a continuous weakly holomorphic function can always be viewed as a weakly holomorphic function, even though the domain of definition is technically different in the definition of these terms.
\end{defn}

\begin{defn}
Let $M$ be a reduced complex analytic space. Then $M$ is called \emph{normal} (resp. \emph{maximal}) if every weakly holomorphic function (resp. continuous weakly holomorphic function) $f$ on $U\subseteq M$ is holomorphic on $U$, where $U\subseteq M$ is open. One might say then that $f$ is even \emph{strongly holomorphic} to emphasis the holomorphicity.
\label{Normal}
\end{defn}

A very famous example of a complex analytic variety actually turns out the be a maximal complex analytic space.

\begin{exmp}
The \emph{Whitney Umbrella}
\[\left\{ \left( z_1,z_2,z_3 \right)\in \bb{C}^3 \mid z_1^2 -z_2^2\cdot z_3 = 0 \right\}\]
is a maximal complex analytic space. A proof can be found in e.g.~\cite[page 301]{Umbrealla} - it boils down to the Whitney Umbrella being a normal crossing everywhere except at the origin, but the origin has codimension $2$ and maximality follows in this case for local complete intersections.\\
In particular, normal crossings are maximal complex analytic spaces, see e.g.~\cite[page 297]{Umbrealla}.
\end{exmp}

The topology of a reduced complex analytic space actually supports a maximal complex analytic structure that almost agrees with the initial complex analytic structure. This notion is defined concretely in the following way:

\begin{defn}
Let $M$ be a reduced complex analytic space. Then a morphism $\mu\colon \hat{M}\to M$ is called a \emph{maximalization (of $M$)} if the following hold:
\begin{enumerate}[(i)]
\item $\hat{M}$ is maximal.
\item $\mu$ is a homeomorphism.
\item The restriction $\hat{M}\setminus \mu^{-1}\left( \mathrm{sing}\left( M \right) \right)\to M \setminus \mathrm{sing}\left( M \right)$ is biholomorphic.
\end{enumerate}
\end{defn}

Such a maximalization always exists in the reduced case.

\begin{thm}
Let $M$ be a reduced complex analytic space. Then there exists a maximalization $\mu\colon \hat{M}\to M$.
\label{MaxEx}
\end{thm}

\begin{proof}
See e.g.~\cite[Theorem 2.29, page 123]{fischer}
\end{proof}

Encoding weakly holomorphic functions in terms of their graph is a useful technique in the upcoming proof. The argument for this encoding is fairly straight forward, once one is familiar with the graph of a meromorphic function. An introduction can be found in e.g.~\cite[Chapter 4]{fischer}.

\begin{prop}
Let $M$ be a reduced complex analytic space. Let $f\colon M_{\mathrm{reg}}\to \bb{C}^r$ be a holomorphic function. Then the following statements are equivalent:
\begin{enumerate}[(i)]
\item $f$ is a weakly holomorphic function on $M$.
\item The closure of the graph $\Gamma_f\subseteq M_{\mathrm{reg}}\times \bb{C}^r$ of $f$ inside $M\times \bb{C}^r$ is analytic and the projection $\beta:=\rst{\pi}{\bar{\Gamma}_f}\colon \bar{\Gamma}_f \to M$ is a proper map.
\end{enumerate}
\label{WeakGraph}
\end{prop}

\begin{proof}
$(i) \implies (ii)$ follows by first recalling that the graph of a meromorphic function on an analytic space is well-defined and by noting that weakly holomorphic functions are meromorphic (see e.g.~\cite[Proposition 4.6, page 181]{fischer}). This settles that the closure of the graph is analytic in $M\times \bb{C}^r$. The question of the projection being proper can be answered locally. To this end, consider a compact subset $K\subseteq M$ and choose compact subsets $K_i\subseteq K$ such that
\[K=\bigcup_{i=1}^l K_i\]
and such that there exists an open neighbourhood $U_i\supseteq K_i$ around $K_i$ where $U_i$ is a local model in $V_i\subseteq \bb{C}^{n_i}$. Note that this is always possible as $K$ is compact. It is clear that
\[\beta^{-1}\left( K \right)= \bigcup_{i=1}^l \beta^{-1}\left( K_i \right)\]
holds, hence, it suffices to show that each $K_i$ is compact. However, as around $K_i$ one is in a local model it follows that $\beta^{-1}\left( K_i \right)$ can be viewed as a closed subset of $V_i\times \bb{C}^r\subseteq \bb{C}^{n_i+r}$. By the assumption of $f$ being locally bounded it follows that $\beta^{-1}\left( K_i \right)$ is a bounded subset. Hence, closed and bounded, hence, compact. The properness of $\beta$ follows.\\
$(ii) \implies (i)$ is clear because $\rst{\pi}{\bar{\Gamma}_f}$ is proper. To see this, let $p\in M$ be arbitrary and let $K\subseteq M$ be a compact neighbourhood around $p$ such that there exists an open neighbourhood $V\supseteq K$ around $K$ that is a local model in $\bb{C}^n$. Then $\beta^{-1}\left( K \right)$ can be viewed as a compact subset of $\bb{C}^{n+r}$. Therefore, also as a closed and bounded one and hence $f$ is bounded on $K$.
\end{proof}

Continuous weakly holomorphic functions are then of course even more well-behaved, as they will only have one limiting value towards the singular points.

\begin{prop}
Let $M$ be a reduced complex analytic space. Let $f\colon M_{\mathrm{reg}}\to \bb{C}^r$ be a holomorphic function. Then the following statements are equivalent:
\begin{enumerate}[(i)]
\item $f$ defines a continuous weakly holomorphic function on $M$.
\item The closure of the graph $\Gamma_f\subseteq M_{\mathrm{reg}}\times \bb{C}^r$ of $f$ inside $M\times \bb{C}^r$ is analytic and the projection $\beta:=\rst{\pi}{\bar{\Gamma}_f}\colon \bar{\Gamma}_f \to M$ is a proper bijection.
\end{enumerate}
\label{ContGraph}
\end{prop}

\begin{proof}
This is clear from applying Proposition~\ref{WeakGraph} and from a closed continuous bijection being a homeomorphism.
\end{proof}

\begin{defn}
Let $\psi\colon M'\to M$ be a morphism of complex analytic spaces. Then $\psi$ is called a \emph{proper modification} if $\psi$ is proper and there exists a nowhere dense analytic subset $A\subseteq M$ such that $\psi^{-1}\left( A \right)$ is nowhere dense and the restriction
\[X\setminus \psi^{-1}\left( A \right)\to M\setminus A\]
is biholomorphic.
\end{defn}

\begin{prop}
Let $\psi \colon M' \to M$ be a proper modification of reduced complex spaces and let $f\colon M'\to \bb{C}^{{r}}$ be a holomorphic function on $M'$. Then $f$ canonically defines a weakly holomorphic function $g$ on $M$ such that
\[\rst{f}{\psi^{-1}\left( M_{\mathrm{reg}} \right)} = g\circ \psi.\]
In particular, $f$ defines a continuous weakly holomorphic function if and only if $f$ is constant on $\psi^{-1}\left( \left\{ p \right\} \right)$ for every $p\in M$.
\label{PropMod}
\end{prop}

\begin{proof}
It is clear that the morphism $\phi:=\psi\times \id\colon M'\times \bb{C}^r\to M\times \bb{C}^r$ is also proper. This implies that the image $Y:=\phi\left( \Gamma_f \right)\subseteq M\times \bb{C}^r$ of the graph of $f$ on $M'$ via $\phi$ is a closed analytic subset of $M\times \bb{C}^r$. Moreover, it is clear that $Y$ equals the closure of the graph of the holomorphic function induced on $M\setminus A$, where $A$ is the nowhere dense analytic subset of $M$ away from which $\psi$ is biholomorphic. Thereby, it suffices to argue that the projection $\beta\colon Y \to M$ is proper. To conclude the properness, note that one is in the situation of the following commutative diagram
\[
\begin{tikzcd}
\Gamma_f \arrow[d,"\beta'"]  \arrow[r,"\phi'"] & Y \arrow[d,"\beta"]\\
M' \arrow[r,"\psi"] & M
\end{tikzcd}
\]
and each morphism except $\beta$ is proper by construction or assumption. Note that $\phi'$ is surjective. Let $K\subseteq M$ be a compact subset. Then
\[{\phi'}^{-1}\left( \beta^{-1}\left( K \right) \right)= {\beta'}^{-1}\left( {\psi}^{-1}\left( K \right) \right)\]
is compact. By surjectivity of $\phi'$, one has
\[\phi'\left( {\phi'}^{-1}\left( \beta^{-1}\left( K \right) \right) \right)= \beta^{-1}\left( K \right).\]
However, continuous images of compact sets are compact and hence $\beta^{-1}\left( K \right)$ is compact. Thus, $\beta$ is proper and thus by Proposition~\ref{WeakGraph} it follows that $g$ is weakly holomorphic on $M$.\\
The last statement about the continuity of $g$ is clear by Proposition~\ref{ContGraph}.
\end{proof} 

\begin{rem}
The graph of a holomorphic function over a reduced complex analytic space is of course such that the projection is biholomorphic.
\end{rem}

\subsection{Strategy outline}
\label{Strat}

The following proof for some cases of the relative Riemann-Hilbert correspondence on singular spaces is lengthy and technical. Therefore the strategy and the basic ideas behind the arguments are presented informally in this subsection. The ideal correspondence in the singular setting that one would like to obtain is something like the following:

\begin{openThm*}
Let $f\colon X\to N$ be a reduced locally trivial morphism of complex analytic spaces. Then there is a one-to-one correspondence between:
\begin{enumerate}[(i)]
\item pairs $\left( \mc{F},\nabla_f \right)$ consisting of a $\coan{X}$-coherent module and a flat tame $f$-relative connection, and,
\item $f$-relative local systems $V$.
\end{enumerate}
The correspondence sends pairs $\left( \mc{F},\nabla_f \right)$ to $\ker{\nabla_f}$ and $f$-relative local systems $V$ to $\left( V\otimes_{f^{-1}\coan{N}}\coan{X}, \id\otimes d_f \right)$.
\end{openThm*}

This type of theorem is considered the ``relative'' version of the Riemann-Hilbert correspondence. The ``absolute'' version of such a correspondence is the situation where the morphism $f$ is simply $X\to \left( \left\{ \mathrm{pt.} \right\}, \bb{C} \right) $. The general relative version is then, intuitively speaking, simply the parametrised version of the absolute correspondence.\\
The overall perspective on the relative Riemann-Hilbert correspondence taken in this work can be summarized by three main ideas:

\begin{enumerate}
\item If the absolute Riemann-Hilbert correspondence can be solved on all fibers of the locally trivial morphism $f$, then the relative Riemann-Hilbert correspondence can be solved for the morphism $f$.
\item If the absolute Riemann-Hilbert correspondence can be solved weakly, then it can be solved strongly, i.e. if there exist weakly holomorphic parallel frames for connections, then these frames are really strongly holomorphic.
\item If the absolute Riemann-Hilbert correspondence can be solved on complex analytic curves, then it can be solved in arbitrary dimension.
\end{enumerate}

The first two ideas are probably the most intuitive and reasonably sounding ideas, while the third idea might be a bit surprising at first. However, it will turn out that the major road block in executing this type of strategy is actually idea (2). Idea (1) and (3) can be accomplished in the case of torsion-free sheaves without further assumptions.\\
Before diving into the more concrete argument one might already guess some of the restrictions that have to be placed on the situation to make these ideas work. Specifically, it already seems prudent to work on reduced complex analytic spaces and torsion-free sheaves only, because in that situation the idea of weakly holomorphic solutions is actually sensibly defined. Without reducedness and torsion-freeness one already sees that the proof of Theorem~\ref{nonrelativetorelative} for a submersion requires new ideas. Therefore, this work considers the relative Riemann-Hilbert correspondence with \emph{reduced} locally trivial morphisms (Definition~\ref{LocTriv}) and \emph{torsion-free} coherent sheaves (Definition~\ref{TorFree}).\\
In ubsection~\ref{weakSol} it will be shown that for torsion-free sheaves on reduced locally trivial morphisms a flat relative connection always admits weakly holomorphic parallel generators. The argument is surprisingly simple. Consider for simplicity the absolute situation with a flat connection $\nabla$ on some reduced complex analytic space $M$. Let $\psi\colon \tilde{M}\to M$ be a desingularization. Then $\psi^*\nabla$ admits local parallel frames as $\tilde{M}$ is a complex manifold. Moreover, $\psi^*\nabla$ is trivial along the fibers of $\psi$. Hence, by Corollary~\ref{ExtPara} there exists a parallel frame in an open neighbourhood of any fiber of $\psi$. The desingularization $\psi$ is of course proper and therefore closed, hence, an open neighbourhood of a fiber contains a preimage of an open neighbourhood. Thus, by Proposition~\ref{PropMod} the parallel frame around a fiber of $\psi$ can be viewed as a weakly holomorphic section on $M$. These weakly holomorphic sections are then parallel away from $\mathrm{sing}\left( M \right)$ as $\psi$ is an isomorphism there.\\
In subsection~\ref{RedAbs} it is first proven (Proposition~\ref{AbsWeak}) that the absolute situation admits very nice \emph{continuous} weakly holomorphic frames $s_i$ that are parallel away from the singularities and furthermore, such that for every $j\in \bb{N}$ the restriction to the iterated singularities  $\rst{s_i}{\mathrm{sing}^j\left( M \right)}$ (Definition~\ref{IterSing}) are again continuous weakly holomorphic frames that are parallel away from the singularities of $\mathrm{sing}^j\left( M \right)$. This is proven rather elegantly by an induction on the dimension and by noticing that in the curve case one only has to make sure that the weakly holomorphic frames one constructs have the same values at a given singular point. The argument then carries this idea over into higher dimensions by considering a desingularization with an exceptional divisor that locally is a normal crossing and therefore maximal (see e.g.~\cite[page 297]{Umbrealla}). The maximality ensures that the continuous weakly holomorphic parallel frames on the singular set are holomorphic on the exceptional divisor and can then be extended to an entire open subset of the desingularization. This extension is the desired continuous weakly holomorphic frame.\\
Using these nice frames in the absolute situation one can show that the weakly holomorphic frames obtained in subsection~\ref{weakSol} are in fact continuous (Proposition~\ref{Czero}), essentially by simply noting that the relative frames restricted to the fibers are just the frames obtained in the absolute situation in Proposition~\ref{AbsWeak}.\\
Leveraging this idea even further yields that the obtained weakly holomorphic sections in the relative situation are precisely strongly holomorphic if and only if the continuous weakly holomorphic sections obtained in the absolute situation are stongly holomorphic (Proposition~\ref{RelStrong}). This argument heavily relies on the torsion-freeness of the sheaves that are considered, as this allows one to almost simply reduce the argument to an application of this Theorem by Grauert and Remmert~\cite[Satz 29]{SepHolo}:

\begin{thm*}
Let $N\times M$ be the product of two reduced complex analytic spaces and let $f\colon N\times M \to \bb{C}$ be a (set-theoretic) map such that for all $p\in N$ the functions $\rst{f}{\left\{ p\right\}\times M }$ and for all $q\in M$ the functions $\rst{f}{ N\times \left\{ q \right\}}$ are holomorphic on $\left\{ p \right\}\times M$ and $N\times \left\{ q \right\}$. Then the function $f$ is holomorphic on $N\times M$.
\end{thm*}

With these facts and results in hand one has finally obtained the following statement of Theorem~\ref{MaxRH} as the continuous weakly holomorphic frames are necessarily holomorphic on maximal spaces:

\begin{thm*}
Let $f\colon X \to N$ be a reduced locally trivial morphism of complex analytic spaces and assume that the fibers of $f$ are maximal complex analytic spaces. Then there is a one-to-one correspondence between
\begin{enumerate}[(i)]
\item pairs $\left( \mc{F},\nabla_f \right)$ of torsion-free coherent sheaves $\mc{F}$ with tame $f$-relative flat connections $\nabla_f$, and,
\item torsion-free $f$-relative local systems $V$.
\end{enumerate}
The correspondence sends pairs $\left( \mc{F},\nabla_f \right)$ to the sheaf $\ker{\nabla_f}$ and sends $f$-relative local systems $V$ to $\left( V\otimes_{f^{-1}\coan{N}} \coan{X}, \nabla^V \right)$.
\end{thm*}

Subsection~\ref{RedToCurve} then demonstrates that one can further reduce the general absolute case to the absolute case over complex analytic curves (Proposition~\ref{RedToCurveGen} and Theorem~\ref{RedToCurveGenCor}). The argument utilizes the description of the weakly holomorphic frames in terms of an analytic graph. One can then always find a complex analytic curve spanning the tangent space of the graph by a Theorem of J. Becker (Theorem~\ref{CurveSpan}) and it suffices to verify holomorphicity on such curves. This idea is then further applied to the case of homogeneous complex analytic subsets of $\bb{C}^n$, as it is possible to verify that the weakly holomorphic frames are always holomorphic on homogeneous complex analytic curves (i.e. on unions of finitely many straight complex lines) - this is the content of Proposition~\ref{HomCurve}. Together with Proposition~\ref{RedCurve} one then obtains the Riemann-Hilbert correspondence with homogeneous fibers and torsion-free sheaves in Theorem~\ref{HomCor}:

\begin{thm*}
Let $f\colon N\times M\to N$ be the projection of reduced complex spaces and assume that $M\subseteq \bb{C}^n$ is a homogeneous complex analytic space. Then there is a one-to-one correspondence between
\begin{enumerate}[(i)]
\item pairs $\left( \mc{F},\nabla_f \right)$ of torsion-free coherent sheaves $\mc{F}$ with tame $f$-relative flat connections $\nabla_f$, and,
\item torsion-free $f$-relative local systems $V$.
\end{enumerate}
The correspondence sends pairs $\left( \mc{F},\nabla_f \right)$ to the sheaf $\ker{\nabla_f}$ and sends $f$-relative local systems $V$ to $\left( V\otimes_{f^{-1}\coan{N}} \coan{X}, \nabla^V \right)$.
\end{thm*}

Now, the strategy laid out above is executed in the following sections beginning with proving the existence of weakly holomorphic parallel generators for flat connections.

\subsection{Existence of weak solutions}
\label{weakSol}

The following result establishes the existence of local weakly holomorphic parallel generators of flat connections on torsion-free sheaves over the total of a reduced locally trivial morphism. The main ingredient is Corollary~\ref{ExtPara} and realising that a connection pulled back to a fiber-wise desingularization is of course trivial along the fibers of such a proper modification.

\begin{prop}
Let $f\colon N\times M \to N$ be the projection of reduced complex analytic spaces. Denote by $\tilde{M}$ a desingularization of $M$ and by $\psi\colon N\times \tilde{M}\to N\times M$ the fiber-wise desingularization.\\
Moreover, let $\left( \mc{F},\nabla_f \right)$ be a torsion-free sheaf with flat $f$-relative connection on $N\times M$ and suppose that there exists an injective morphism $\theta\colon \mc{F}\to \mathcal{O}_{N\times M}^{\oplus r}$.\\
Then for every $\left( p,q \right)\in N\times M$ there exists an open neighbourhood $U\times W\subseteq N\times M$ of $\left( p,q \right)\in M$ such that $\rst{\psi^*\mc{F}}{\psi^{-1}\left( U\times W \right)}$ has global parallel generators $s_i\in \left(\psi^*\mc{F}\right)\left(\psi^{-1}\left( U\times W \right)\right)$ such that
\[\left(\psi^*\theta\right)\left( s_i \right)\in \left(\mathcal{O}_{N\times \tilde{M}}^{\oplus r}\right)\left(\psi^{-1}\left( U\times W \right)\right)\]
defines tuples of weakly holomorphic functions on $U\times W$, which are holomorphic on $U\times W_{\mathrm{reg}}$.\\
Moreover, the sections $s_i$ can be chosen such that
\[\rst{s_i}{\psi^{-1}\left(U\times\left\{ q \right\}\right)}={\psi'}^*t_i,\]
where $\psi'\colon \psi^{-1}\left( U\times \left\{ q \right\} \right)\to U\times \left\{ q \right\}$ is the restriction of $\psi$ and $t_i$ is a section of $\rst{\mc{F}}{U\times \left\{ q \right\}}$, i.e. the functions $\left( \psi^*\theta \right)\left( s_i \right)$ are constant on the fibers of $\psi$ over $\left( a,q \right)$ for every $a\in U$.
\label{RelWeak}
\end{prop}

\begin{proof}
Let $q\in M$ and denote $X_q:= \psi^{-1}\left( N\times \left\{ q \right\} \right)$ and because the question is local one may assume that $\mc{F}$ has global generators. Moreover, let $\iota\colon N\times \left\{ q \right\}\to N\times M$ and $\iota'\colon X_q\to N\times \tilde{M}$ be the inclusions of analytic subspaces. Then, one of course has
\[\left(\iota'^*\psi^* \mc{F}, \iota'^*\psi^*\nabla \right) = \left( \psi'^*\iota^*\mc{F},\psi'^*\iota^*\nabla \right),\]
where $\psi'\colon X_q \to N\times \left\{ q \right\}$ is the restriction of $\psi$ to $X_q$. However, $\iota$ is a section of $f$ and hence $\psi'^*\iota^*\nabla $ is the canonical connection on $\psi'^*\iota^*\mc{F}$ associated to $\psi'^{-1}\iota^*\mc{F}$. Because $\mc{F}$ is globally generated, it follows that $\psi'\iota^*\nabla$ has global parallel generators. Additionally, by Corollary~\ref{ExtPara}, one has
\[\ker{\psi'\iota^*\nabla}=\ker{\iota'^*\psi^*\nabla }=\iota'^{-1}\ker{\psi^*\nabla}\]
and that there exists an open neighbourhood $V$ of $X_q$ such that $\ker{\psi^*\nabla}$ is globally generated on $V$.\\
It still needs to be argued that $V$ contains an open subset of the form $\psi^{-1}\left( U\times W \right)$ with $p\in U$. For this, note that $\psi$ is a closed map and hence $C:=\psi\left( \left( N\times \tilde{M} \right)\setminus V \right)\subset M\times N$ is closed. As $N\times \left\{ p \right\}$ is also closed and
\[\left( N\times \left\{ p \right\} \right)\cap C=\emptyset,\]
it follows that there exists an open neighbourhood $V'\subseteq N\times M$ of $N\times \left\{ q \right\}$ such that
\[V'\cap C =\emptyset.\]
As $V'$ is open and contains $N\times \left\{ q \right\}$, it follows that for every $p\in N$ there exist open neighboourhoods $U\subseteq N$ around $p$ and $W\subseteq M$ around $q$ such that $U\times W\subseteq V'$. Subseqeuntly, $\psi^{-1}\left( U\times W \right)\subseteq V$ is an open neighbourhood of $\psi^{-1}\left( U\times \left\{ q \right\} \right)$.\\
Since $\psi\colon N\times \tilde{M}\to N\times M$ is a proper modification it follows that the functions $\left( \psi^*\theta \right)\left( s_i \right)$ are weakly holomorphic on $N\times M$ by Proposition~\ref{PropMod}.
\end{proof}

One would like to argue now that the thus constructed frame $s_i$ defines \textit{continuous} weakly holomorphic functions $\left( \psi^*\theta \right)\left( s_i \right)$ on $N\times M$, however for continuity one needs that these functions are constant on all fibers of $\psi$ (Proposition~\ref{PropMod}) and so far the previous result only guarantees local boundedness. The next section will establish that these functions are continuous weakly holomorphic and will show that the absolute case admits `nice' continuous weakly holomorphic solutions. Moreover, the case of maximal fibers will follow immediately from these results.

\subsection{Reduction to absolute case and maximal fibers}
\label{RedAbs}

In this subsection, it will be shown that it suffices to verify whether the sections obtained in Proposition~\ref{RelWeak} are holomorphic along each fiber. The main inspiration for this argument is the following theorem of Grauert and Remmert regarding separate holomorphicity on product complex analytic spaces:

\begin{thm}
Let $N\times M$ be the product of two reduced complex analytic spaces and let $f\colon N\times M \to \bb{C}$ be a (set-theoretic) map such that for all $p\in N$ the functions $\rst{f}{\left\{ p\right\}\times M }$ and for all $q\in M$ the functions $\rst{f}{ N\times \left\{ q \right\}}$ are holomorphic on $\left\{ p \right\}\times M$ and $N\times \left\{ q \right\}$. Then the function $f$ is holomorphic on $N\times M$.
\label{SepHolo}
\end{thm}

\begin{proof}
See~\cite[Satz 29]{SepHolo}.
\end{proof}

In order to apply this theorem one has to first show that the weakly holomorphic sections of Proposition~\ref{RelWeak} actually define set-theoretic maps, i.e. are actually continuous weakly holomorphic. One can do this by first establishing that in the absolute case there are continuous weakly holomorphic frames around the singular points - this is established in Proposition~\ref{AbsWeak}. One can then show that the sections constructed in Proposition~\ref{RelWeak} agree with the frames constructed in Proposition~\ref{AbsWeak} after restricting them to the fibers and continuity follows by Corollary~\ref{Czero}. One has thus obtained set-theoretic maps and can utilize Theorem~\ref{SepHolo} to reduced holomorphicity to the holomorphicity in the absolute case. Proposition~\ref{RelStrong} demonstrates precisely this application. All of these facts together yield two interesting results:
\begin{itemize}
\item Theorem~\ref{RelToAbs}: The Riemann-Hilbert correspondence holds for torsion-free sheaves and reduced locally trivial morphisms if and only if the sections in the absolute case (i.e. along the fibers) are holomorphic.
\item Theorem~\ref{MaxRH}: The Riemann-Hilbert correspondence holds for torsion-free sheaves and reduced locally trivial morphisms with maximal fibers.
\end{itemize}

Before diving into the arguments recall two important aspects of complex analytic spaces:

\begin{defn}
Let $M$ be a complex analytic space and let $j\in \bb{N}\setminus \left\{ 0 \right\}$. Then one recursively defines
\[\mathrm{sing}^j\left( M \right):=\mathrm{sing}\left( \mathrm{sing}^{j-1}\left( M \right) \right)\]
and sets $\mathrm{sing}^0\left( M \right)=M$. In particular, one has $\mathrm{sing}^1\left( M \right)=\mathrm{sing}\left( M \right)$.
\label{IterSing}
\end{defn}

\begin{thm}
Let $M$ be a reduced complex analytic space. Then there exists a proper modification $\psi\colon \tilde{M}\to M$ such that $\tilde{M}$ is a complex manifold and $\psi^{-1}\left( \mathrm{sing}\left( M \right) \right)$ is locally a normal crossing hypersurface. In particular, $\psi^{-1}\left( \mathrm{sing}\left( M \right) \right)$ is a maximal complex analytic space.
\label{Desing}
\end{thm}

\begin{proof}
See~\cite{desing} for the desingularization and for the maximality statement see~\cite[page 297]{Umbrealla}.
\end{proof}

Now one is in a position to adequately address the absolute situation and show that there exist well-behaved continuous weakly holomorphic frames.

\begin{prop}
Let $M$ be a reduced complex analytic space and $\left( \mc{O}_M^{\oplus r},\nabla \right)$ a free sheaf with flat connection. Then around every $p\in M$ there exists an open neighbourhood $U\subseteq M$ of $p$ and continuous weakly holomorphic sections $t_1,\dots,t_r$ on $U$ such that for every $i\in\left\{ 1,\dots,r \right\}$ the restriction $\rst{t_i}{U_{\mathrm{reg}}}$ is parallel with respect to $\nabla$ and that these $r$ holomorphic functions form a frame of $\coan{U_{\mathrm{reg}}}^{\oplus r}$.\\
Moreover, for every $j\in \bb{N}$ and $i\in \left\{ 1\dots,r \right\}$ the restrictions $\rst{t_i}{\mathrm{sing}^j\left( U \right)}$ are continuous weakly holomorphic on $\mathrm{sing}^j\left( U \right)$ and parallel on $\left( \mathrm{sing}^j\left( U \right) \right)_{\mathrm{reg}}$ with respect to $\rst{\nabla}{ \mathrm{sing}^j\left( U \right)}$.\\
Moreover, the value of the frames at $p$ can be freely chosen as long as they form a basis at that point.
\label{AbsWeak}
\end{prop}

\begin{proof}
This result can be obtained by an induction on the dimension of $M$. To this end, first suppose that $M$ has dimension $1$ and thus its singular set has dimension $0$. Let $\psi\colon \tilde{M}\to M$ be a desingularization and note that $\psi$ can be taken to be finite (e.g. simply take the normalization of $M$). Moreover, the Proposition is clear at the regular points of $M$ and only the singular points need to considered. Therefore let $p\in \mathrm{sing}\left( M \right)$ be a singular point. Around every $q\in \psi^{-1}\left( \left\{ p \right\} \right)=\left\{ q_1,\dots q_k \right\}$ there exists a parallel frame of $\psi^*\nabla$ as $\tilde{M}$ is a complex manifold and moreover, one may assume that these frames have the same values at all points of $\psi^{-1}\left( \left\{ p \right\} \right)$. Therefore these frames define continuous weakly holomorphic functions on $M$. And by construction they are parallel and a frame away from the singularities. The singular points of $M$ are just isolated points and hence, the restrictions to the singular set are also holomorphic and parallel as these are vacuous statements.\\
Suppose now that the Proposition holds for analytic spaces of dimension $\leq n-1$. Let $M$ be of dimension $n$ and note that $\mathrm{sing}\left( M \right)$ is of dimension $\leq n-1$. By assumption there exist continuous weakly holomorphic functions $u_1,\dots,u_r$ on $\mathrm{sing\left( M \right)}$ which satisfy the Proposition on $\mathrm{sing}\left( M \right)$. Let $\psi\colon \tilde{M} \to M$ be a desingularization of $M$ such that $\psi^{-1}\left( \mathrm{sing}\left( M \right) \right)$ locally is a normal crossing hypersurface. Moreover, let $\iota\colon \mathrm{sing}\left( M \right)\to M$, $\iota'\colon \psi^{-1}\left( \mathrm{sing}\left( M \right) \right)\to \tilde{M}$ be the inclusions and let $\psi'\colon \psi^{-1}\left( \mathrm{sing}\left( M \right) \right)\to \mathrm{sing}\left( M \right)$ be the restriction of $\psi$.\\
It is now going to be argued that the pull-back frames $\psi'^*u_i$ are actually holomorphic on $\psi^{-1}\left( \mathrm{sing}\left( M \right) \right)$. By definition they are already continuous functions and the question of holomorphicity is a local question. Therefore, let $q\in \psi^{-1}\left( \mathrm{sing}\left( M \right) \right)$ and because the inverse image is a local normal crossing it follows that there exists an open neighbourhood $V\subseteq \tilde{M}$ around $q$ and finitely many smooth, connected complex hypersurfaces $X_l$ such that
\[V\cap \psi^{-1}\left( \mathrm{sing}\left( M \right) \right)=\bigcup_{l=1}^m X_l.\]
For each $l\in \left\{ 1,\dots,m \right\}$ consider $\rst{\psi}{X_l}\colon X_l \to M$ and let $i_l$ be the largest integer such that $\psi\left( X_l \right)\subseteq \mathrm{sing}^{i_l}\left( M \right)$. Note that $i_l\geq 1$ as one is considering the exceptional hypersurface in $\tilde{M}$. In particular, one has that $\rst{\psi}{X_l}$ always factors through $\mathrm{sing}^{i_l}\left( M \right)$. Moreover, it follows that $\psi^{-1}\left( \left( \mathrm{sing}^{i_l}\left( M \right) \right)_{\mathrm{reg}} \right)$ is dense in $X_l$ since by definition $\psi^{-1}\left( \mathrm{sing}^{i_l+1}\left( M \right) \right)$ is a proper analytic subset and thereby nowhere dense, since $X_l$ is smooth and connected. Recall, however that $\rst{u_i}{\left( \mathrm{sing}^{i_l}\left( M \right) \right)_{\mathrm{reg}}}$ is holomorphic and hence $\left( \rst{\psi}{X_l} \right)^*u_i$ is continuous on $X_l$ and holomorphic on a dense subset. However, $X_l$ is a complex manifold and thus $\left( \rst{\psi}{X_l} \right)^*u_i$ is holomorphic everywhere. Thereby, one has obtained that ${\psi'}^*u_i$ is continuous and holomorphic on a dense subset of $V\cap \psi^{-1}\left( \mathrm{sing}\left( M \right) \right)$, but by Theorem~\ref{Desing} one has that $\psi^{-1}\left( \mathrm{sing}\left( M \right) \right)$ is a maximal complex analytic space and thus ${\psi'}^*u_i$ is locally holomorphic on the exceptional hypersurface. So one has obtained that the pull-back frame to the exceptional hypersurface is holomorphic.\\
It then also follows by the tameness of ${\iota'}^*\psi^*\nabla$ (Corollary~\ref{ExtPara}) that the pull-back frames are parallel as they are parallel on a dense subset by assumption. Then also by Corollary~\ref{ExtPara} one has 
\[\ker{ {\iota'}^*\psi^*\nabla}={\iota'}^{-1}\ker{\psi^*\nabla}\]
and it follows that there exists an open neighbourhood $W\subseteq \tilde{M}$ around $\psi^{-1}\left( \mathrm{sing}\left( M \right) \right)$ and a parallel frame $\tilde{t}_i$ of $\psi^*\nabla$ on $W$ that extends ${\psi'}^*u_i$. By construction the frame $\tilde{t}_i$ is constant along the fibers of $\psi$ and thus these frames are induced by continuous weakly holomorphic sections on the open neighbourhood $U:=\psi\left( W \right)\subseteq M$ around $\mathrm{sing}\left( M \right)$\footnotemark. Moreover, away from the singularities these sections are parallel and by construction, they restrict to $u_j$ on $\mathrm{sing}\left( M \right)$. This completes the induction argument.\\
That the value of the frames at $p$ can be freely chosen to be any basis at that point is clear from the construction.
\end{proof}

\footnotetext{Continuity follows by Proposition~\ref{PropMod} and the openness of $\psi\left( W \right)$ follows from $\psi$ being proper and surjective, hence a quotient map and $\psi^{-1}\left( \psi\left( W \right) \right)=W$ is open.}

Observe that on simply connected reduced complex analytic spaces one always obtains these continuous weakly holomorphic frames globally. 

\begin{lem}
Let $M$ be a simply connected reduced complex analytic space and let $\left( \coan{M}^{\oplus r},\nabla \right)$ be a free sheaf with flat connection. Then the continuous weakly holomorphic frames constructed in Proposition~\ref{AbsWeak} exist globally on $M$.
\label{SimpConnec}
\end{lem}

\begin{proof}
Note that the change-of-frame morphisms between the continuous weakly holomorphic frames obtained in Proposition~\ref{AbsWeak} are locally constant on a desingularization of $M$ and continuous weakly holomorphic on $M$, hence, they are locally constant. As such these frames together with their changes of frames naturally define a local system $V$ on $M$. Local systems are however in one-to-one correspondence with linear representations of the fundamental group (via parallel transport see e.g.~\cite[page 3]{deligne}) and hence the local system $V$ is trivial, i.e. isomorphic to the constant sheaf $\bb{C}^r$. As such, there exist $r$ globally defined sections that are parallel on $M_{\mathrm{reg}}$ and form a frame over the regular part of $M$.
\end{proof}

By restriction to the fibers it is now fairly straight forward to show that the weakly holomorphic sections obtained in Proposition~\ref{RelWeak} are actually continuous.

\begin{cor}
Continue the notation and setup from Proposition~\ref{RelWeak}. The weakly holomorphic functions $\left( \psi^*\theta \right)\left( s_i \right)$ are continuous weakly holomorphic on $U\times W$, after potentially shrinking $U\times W$ around $\left( p,q \right)$ so that $W$ is simply connected. In particular, they define (set-theoretic) maps $U\times W \to \bb{C}^r$.\\
Moreover, the weakly holomorphic sections defined by $\rst{s_i}{ \psi^{-1}\left(\left\{ n \right\}\times W\right)}$ on $\left\{ n \right\}\times W$ locally agree with the continuous weakly holomorphic sections obtained in Proposition~\ref{AbsWeak}.
\label{Czero}
\end{cor}

\begin{proof}
This statement is an immediate consequence of Proposition~\ref{AbsWeak}, but it is cumbersome to write out. The argument is as follows:\\
Note that for every $p  \in W$ the restriction $\rst{\mc{F}}{ \left\{ p \right\}\times W}  $ is a locally free sheaf with flat connection on $\left\{ p \right\}\times W$ and the sections $\rst{s_i}{ \psi^{-1}\left(\left\{ p \right\}\times W\right)}$ are parallel sections which are constant along the fiber of $\psi$ at $q\in W$. Let $t_i$ be the continuous weakly holomorphic sections around $q \in W$ constructed in Proposition~\ref{AbsWeak} (and guaranteed to exists globally on $W$ by Lemma~\ref{SimpConnec}) with the same value at $q$ as the weakly holomorphic sections $s_i$ (this makes sense because $s_i$ is constant along the fiber of $\psi$ at $q$). Now, denote by $\psi'\colon \psi^{-1}\left( \left\{ p \right\}\times W \right)\to \left\{ p \right\}\times W$ the restriction $\psi$. Then the sections ${\psi'}^*t_i$ are parallel holomorphic sections of ${\psi'}^*\rst{\nabla}{ \left\{ p \right\}\times W}$ just as the $s_i$ are. However, then ${\psi'}^*t_i -s_i$ is parallel and vanishes on $\psi'^{-1}\left( \left( p,q \right)\right)$. Parallel sections that vanish at a point vanish in entire open neighbourhoods and hence there exists an open neighbourhood $V\subseteq \psi'^{-1}\left( \left( p,q \right) \right)$ on which these frames are the same. However, $\psi'\left( V \right)$ contains an open neighbourhood of $\left( p,q \right)$\footnotemark and thus the weakly holomorphic sections $t_i$ and $\rst{s_i}{ \left\{ p \right\}\times W}$ agree in an open neighbourhood $\left( p,q \right)$ in $\left\{ p \right\}\times W$.\\
This shows that the sections $s_i$ are constant along the fibers of $\psi$ in an entire open neighbourhood of points $\left( a,q \right)$ for every $a\in U$. Hence, by shrinking $U\times W$ around $\left( p,q \right)$ one may assume that the sections $\left( \psi^*\theta \right)\left( s_i \right)$ are continuous weakly holomorphic on $U\times W$.
\end{proof}

\footnotetext{The morphism $\psi$ is closed and $V$ contains the entire fiber $\psi^{-1}\left( \left( p,q \right) \right)$. Therefore, by the same argument as in the proof of Proposition~\ref{RelWeak}, it follows that $V$ contains the preimage of an open neighbourhood of $\left( p,q \right)$.}

With this consideration in hand, one can now show that the relative Riemann-Hilbert Theorem for torsion-free sheaves follows from the absolute case. So far we have used an arbitrary embedding $\mc{F}\to \coan{N\times M}$ to obtain the weakly holomorphic frames, but in order to reduces the realtive case to the absolute case we will employ a specific embedding namely the canonical local embedding

\[\mc{F}\to {\mc{F}^{**}}\to \coan{N\times M}\]

which factors through the double dual. The reason for this is two-fold: Firstly, recall that the dual of a coherent sheaf naturally corresponds to the sheaf of sections of its associated linear fiber space (see e.g.~\cite[Section 1.6]{fischer}) and secondly the following Proposition shows that the dual sheaf of a coherent sheaf with flat connection is defined by a relative local system if the connections along the fibers have holomorphic parallel frames.

\begin{prop}
Let $f\colon N\times M \to N$ be the projection of reduced complex analytic spaces and let $\left( \mc{F},\nabla \right)$ be a torsion-free coherent $\coan{N\times M}$-sheaf with flat $f$-relative connection. Assume further that the connections $\rst{\nabla}{\left\{ p \right\}\times M}$ along the fibers of $f$ have local holomorphic frames for every $p\in N$ that are parallel on a dense subset of $\left\{ p\right\}\times M $.\\
Then for every point in $N\times M$ there exist an open neighbourhood $U\times W$ and a torsion-free $f$-relative local system $V$ such that
\[V\otimes_{f^{-1}\coan{N}} \coan{U\times W}\cong \mc{F}^*.\]
Moreover, the $f$-relative connection induced by the isomorphism above restricts to the dual connection\footnotemark of $\nabla$ over $U\times W_{\mathrm{reg}}$.
\label{DualConnec}
\end{prop}

\footnotetext{For a $f$-relative connection $\nabla$, with $f$ a submersion, the dual $\nabla^*$ on $\mc{F}^*$ is defined by the rule
$\left(\nabla^*_X \alpha\right)(s)=X.(\alpha(s))-\alpha\left( \nabla_X(s) \right)$, where $X$ is an arbitrary vector field along the fibers of $f$. Note that this definition relies on $f$ being a submersion.}

\begin{proof}
Let again $\psi \colon N\times \tilde{M}\to N\times M$ be a fiber-wise desingularization and denote by $g\colon N\times \tilde{M}\to N$ the projection. Let $\iota\colon \psi^{-1}\left( N\times \left\{ q \right\} \right):=N\times X \to N\times \tilde{M}$ be the inclusion of reduced spaces and denote by $\iota'\colon N\times \left\{ q \right\}\to N\times M$ the inlusion and by $\psi'\colon N\times X\to N\times \left\{ q \right\}$ the restriction of $\psi$. Here $q\in M$ is a fixed arbitrary point. Notice that we have the following natural isomorphisms
\[\iota^*\left( \psi^*\mc{F} \right)^*\cong \iota^*\left(g^*\mc{G}\right)^*\cong \left(\iota^*g^*\mc{G}\right)^*\cong \left( \iota^*\psi^*\mc{F} \right)^*\cong \left( {\psi'}^*{\iota'}^* \mc{F}\right)^*\cong {\psi'}^*\left( {\iota'}^*\mc{F} \right)^*.\]
This identification implies that there exist generating sections of $\iota^*\left( \psi^*\mc{F} \right)^*$ that are constant on the fibers of $\psi'$ (as sections of the underlying linear fiber space\footnotemark) -- namely those sections that are pulled back along $\psi'$. These sections are also parallel with respect to $\iota^*\left( \psi^*\nabla^* \right)$ as $\iota^*\psi^*\nabla$ is trivial (just like in the proof of Proposition~\ref{RelWeak}) and therefore so is its dual connection. \footnotetext{Recall the duality between coherent sheaves and linear fiber spaces. In terms of linear fiber spaces, $\mc{F}$ being a coherent sheaf means that locally $\mc{F}$ can be represented as the linear $1$-forms on an analytic subspaces $V\left( \mc{F} \right)\subseteq N\times M \times \bb{C}^r$, where $V\left( \mc{F} \right)$ has a vector space structure in the $\bb{C}^r$-component and this vector space structure varies holomorphically. The dual sheaf $\mc{F}^*$ can then be identified as the sheaf of sections $N\times M \to V\left( \mc{F} \right)$. These identifications are functorial and provide a contravariant equivalence between the category of coherent sheaves and the category of linear fiber spaces (see e.g.~\cite[Section 1.1 - 1.8]{fischer}).}\\
These parallel generators $s_i$ extend to an entire open neighbourhood of $N\times X$ by Corollary~\ref{ExtPara} (just like in the proof of Proposition~\ref{RelWeak}) as parallel sections of $\left(\psi^*\nabla\right)^*$ (after potentially shrinking $N$ and $M$).
By assumption $\rst{\mc{F}}{ \left\{ p \right\}\times M}$ has local holomorphic frames that are parallel on a dense set and as $\rst{\mc{F}}{ \left\{ p \right\}\times M}$ is locally free it follows that the local dual frames form a local system and therefore define a flat connection that away from the singularities agrees with dual connection of $\rst{\nabla}{ \left\{ p \right\}\times M}$ on $\left(\rst{\mc{F}}{ \left\{ p \right\}\times M}\right)^*$. Therefore the dual connection $\left( \rst{\psi^*\nabla}{ \left\{ p \right\}\times \tilde{M}} \right)^*$ has frames of parallel sections that are pull-backs from $\left\{ p \right\}\times M$. Now, just like in the proof of Corollary~\ref{Czero} one concludes that the sections $s_i$ are actually constant on all fibers of $\psi$.\\
However, here more is true, in Corollary~\ref{Czero} the sections that we pulled back from $N\times M$ were only weakly holomorphic, but here they are even strongly holomorphic. Moreover, the sections $s_i$ (as sections of a dual sheaf, therefore as sections of a linear fiber space) can be seen as maps $N\times M\to N\times M\times \bb{C}^r$ and one has just argued that the restriction $\rst{s_i}{  \left\{ p \right\}\times M }$ is strongly holomorphic. The holomorphicity of $\rst{s_i}{N\times \left\{ q \right\}}$ is clear, as the $s_i$ are holomorphic after pull-back along $\psi$ and $\psi$ is equal to the identity in the first component. One can therefore conclude via Theorem~\ref{SepHolo} that the sections $s_i$ are strongly holomorphic and are therefore elements of $\mc{F}^*$ on $N\times M$.\\
Naturally one obtains a surjective morphism $A\colon f^*\mc{G}^*\to \mc{F}^*$ by mapping the generators of $\mc{G}$ to $s_i$, which are generators of $\mc{F}$, as they are generators after pull-back via $\psi$. To see that this is well-defined, recall that this is well-defined on $N\times M_{\mathrm{reg}}$ and $\mc{F}^*$ is torsion-free, hence any relation between generators gets mapped to zero, as it gets mapped to zero on a dense subset.\\
As $f^*\mc{G}^*$ is torsion-free it follows that the morphism $A$ is also injective and therefore an isomorphism. The induced connection clearly restricts to the dual connection on $N\times M_{\mathrm{reg}}$.
\end{proof}

\begin{rem}
Let $f\colon N\times M\to N$ be the projection of reduced complex analytic spaces and let $\mc{F}$ be a $\coan{N\times M}$-coherent sheaf. Denote by $\Omega^1_{f,\mathrm{tf}}\left( \coan{N\times M} \right)$ the sheaf of $f$-relative $1$-forms modulo torsion. Notice that one can also define differential operators
\[\nabla_{\mathrm{tf}}\colon \mc{F}\to \mc{F}\otimes_{\coan{N\times M}}\Omega^1_{f,\mathrm{tf}}\left( \coan{N\times M} \right)\]
that satisfy a Leibniz-rule just like a connection and if $\psi\colon N\times \tilde{M}\to N\times M$ is a fiber-wise desingularization then $\nabla_{\mathrm{tf}}$ can be pulled-back to a regular connection over a submersion, as torsion $1$-forms are annihalted by the pull-back of $\psi$. Call such an operator a \emph{$f$-relative tf-connection}.\\
Morpshisms preserving tf-connections are defined in the obvious way and cokernels of morphisms and kernels of surjetive morphisms of tf-connection carry tf-connections by simple adaptations of the proofs of Lemma~\ref{connecpreserve} and Proposition~\ref{MorphOfConnec}.\\
In particular, it follows that if $\mc{F}$ is a coherent sheaf with $f$-relative tf-connection such that $\rst{\mc{F}}{M\times N_{\mathrm{reg}}}$ is the zero sheaf then $\mc{F}=0$. To see this, pull $\mc{F}$ back via $\psi$ to a sheaf with connection and then pull-back to each fiber, where the pull-back will be locally free by Lemma~\ref{ConnecLocFree}. But the pull-back is zero on a dense subset and therefore it is zero everywhere. Fiber rank is invariant under pull-back and hence $\mc{F}$ is the zero sheaf.\\
\label{tfConnec}
\end{rem}

\begin{prop}
Let $f\colon N\times M \to N$ be the projection of reduced complex analytic spaces and let $\left( \mc{F},\nabla \right)$ be a torsion-free coherent $\coan{N\times M}$-sheaf with flat $f$-relative connection. Assume further that the connections $\rst{\nabla}{\left\{ p \right\}\times M}$ along the fibers of $f$ have local holomorphic frames for every $p\in N$ that are parallel on a dense subset of $\left\{ p\right\}\times M $.\\
Then for every point in $M\times N$ there exist an open neighbourhood $U\times W$ and a torsion-free $f$-relative local system $V$ such that
\[V\otimes_{f^{-1}\coan{U}} \coan{U\times W}\cong \rst{\mc{F}}{U\times W}.\]
Moreover, the connection induced by this $f$-relative local system agrees with the connection $\nabla$ on $N\times W_{\mathrm{reg}}$.
\label{RelStrong}
\end{prop}

\begin{proof}
Let $\psi\colon N\times \tilde{M}\to N\times M$ be a fiber-wise desingularization. The question is local and one can therefore assume that one has embeddings of the following form
\[\mc{F}\hookrightarrow \mc{F}^{**}\cong f^*\mc{H}\hookrightarrow f^*\coan{N}^{\oplus r}\]
where the right embedding is a morphism of the induced connections and the left embedding is a morphism of connections on $N\times M_{\mathrm{reg}}$ as $\mc{F}^{**}$ is endowed with the double dual connection there. Moreover, assume that $N$ and $M$ are Stein spaces. In particular, one obtains an embedding
\[\theta\colon \mc{F}\hookrightarrow \left(\coan{N\times M}^{\oplus r},d_f\right)\]
that is a morphism of connections on a dense subset. By Corollary~\ref{Czero} and Proposition~\ref{DualConnec} we know that there are holomorphic sections $s_i\in \coan{N\times M}^{\oplus r}$, the pull-back of which along $\psi$ are elements of $\psi^*\mc{F}$ which are parallel with respect to $\psi^*\nabla$ and generators of the relative local system $\ker{\psi^*\nabla}$. Therefore in this case these sections are in the kernel of $d_f$ as the embedding is a morphism of connections on $N\times M_{\mathrm{reg}}$. As in the proof of Proposition~\ref{DualConnec} one obtains an injective morphism $f^*\mc{G}\to \coan{N\times M}^{\oplus r}$ which preserves the canonical connections defined on these sheaves.\\
In the following we project all connections involved so far to tf-connections and our morphisms will then all be morphisms of tf-connections -- even $\theta$ as
\[\coan{N\times M}^{\oplus r}\otimes_{\coan{N\times M}} \Omega^1_{f,\mathrm{tf}}\left( \coan{N\times M} \right)\]
is torsion-free and $\theta$ is a morphism of connections on $N\times M_{\mathrm{reg}}$.\\
In particular, the quotient sheaf $\quo{\coan{N\times M}^{\oplus r}}{f^*\mc{G}}$ comes with a tf-connection via the cokernel construction of Lemma~\ref{connecpreserve} and the morphism $\theta'\colon \mc{F}\to \quo{\coan{N\times M}^{\oplus r}}{f^*\mc{G}}$ becomes a morphism of tf-connections. The cokernel of $\theta'$ then also carries a compatible tf-connection. Due to the exact sequence
\[
\begin{tikzcd}
0 \arrow[r] & \im{ \theta'}=\ker{B} \arrow[r] & \quo{\coan{N\times M}^{\oplus r}}{f^*\mc{G}} \arrow[r,"B"] & \coker{\theta'} \arrow[r] &0
\end{tikzcd},
\]
$B$ being a morphism of tf-connections and Proposition~\ref{MorphOfConnec}, it follows that $\im{\theta'}$ carries a tf-connection. However, $\im{\theta'}=0$ on $N\times M_{\mathrm{reg}}$ as there $\mc{F}=f^*\mc{G}$ as subsheaves of $\coan{N\times M_{\mathrm{reg}}}$. Hence by Remark~\ref{tfConnec} one has $\im{\theta'}=0$ and hence $\mc{F}\subseteq f^*\mc{G}$. By reversing roles of $f^*\mc{G}$ and $\mc{F}$ in the previous argument one obtains the inclusion the other way around and thus $\mc{F}=f^*\mc{G}$.
\end{proof}

The reduction of the relative case to the absolute case is thereby completed. Recall that maximal complex analytic spaces are precisely such that continuous weakly holomorphic functions are holomorphic (Definition~\ref{Normal}) and thus one now easily obtains the following case of the Riemann-Hilbert correspondence.

\begin{thm}
Let $f\colon X \to N$ be a reduced locally trivial morphism of complex analytic spaces and assume that the fibers of $f$ are maximal complex analytic spaces. Then there is a one-to-one correspondence between
\begin{enumerate}[(i)]
\item pairs $\left( \mc{F},\nabla_f \right)$ of torsion-free coherent sheaves $\mc{F}$ with tame $f$-relative flat connections $\nabla_f$, and,
\item torsion-free $f$-relative local systems $V$.
\end{enumerate}
The correspondence sends pairs $\left( \mc{F},\nabla_f \right)$ to the sheaf $\ker{\nabla_f}$ and sends $f$-relative local systems $V$ to $\left( V\otimes_{f^{-1}\coan{N}} \coan{X}, \nabla^V \right)$.
\label{MaxRH}
\end{thm}

\begin{proof}
The proof is now a simple gathering of the obtained results. Start with a pair $\left( \mc{F},\nabla_f \right)$ of a torsion-free coherent sheaf with tame flat $f$-relative connection. By the maximality of the fibers it follows that the frames of connections on the fibers obtained in Proposition~\ref{AbsWeak} are strongly holomorphic. By Proposition~\ref{RelStrong}, it follows that $\left( \mc{F},\nabla_f \right)$ has local generators that are parallel away form the singularities of the fibers of $f$ and these generators locally form a relative local sytem $f^{-1}\mc{G}$ such that locally $\mc{F}= f^*\mc{G}$. The tameness of $\nabla_f$ then implies that these sections are parallel everywhere. Hence locally it holds that $\ker{\nabla_f}=f^{-1}\mc{G}$. One therefore has that $\ker{\nabla_f}$ is a torsion-free $f$-relative local system such that $\left( \mc{F},\nabla_f \right)\cong \left( \ker{\nabla_f}\otimes_{f^{-1}\coan{N}}\coan{X},\id\otimes d_f \right)$.\\
Conversely, torsion-free $f$-relative local system lead to tame flat relative connections by Theorem~\ref{kernelofcanonical}.
\end{proof}

More generally, the preceding argument shows the following reduction of the relative case to the absolute case.

\begin{thm}
Let $f\colon X\to N$ be a reduced locally trivial morphism of complex spaces and assume the continuous weakly holomorphic parallel frames obtained in Proposition~\ref{AbsWeak} are strongly holomorphic for the fibers of $f$. Then there is a one-to-one correspondence between
\begin{enumerate}[(i)]
\item pairs $\left( \mc{F},\nabla_f \right)$ of torsion-free coherent sheaves $\mc{F}$ with tame $f$-relative flat connections $\nabla_f$, and,
\item torsion-free $f$-relative local systems $V$.
\end{enumerate}
The correspondence sends pairs $\left( \mc{F},\nabla_f \right)$ to the sheaf $\ker{\nabla_f}$ and sends $f$-relative local systems $V$ to $\left( V\otimes_{f^{-1}\coan{N}} \coan{X}, \nabla^V \right)$.
\label{RelToAbs}
\end{thm}

With this reduction to the absolute case in hand, one can push further and reduce the absolute case in arbitrary dimension to the one dimensional case.

\subsection{Reduction to curve case and homogeneous fibers}
\label{RedToCurve}

On a manifold it always suffices to test differentiability or holomorphicity after pull-back to any smooth curve. On a singular space the situation is more difficult. However, a somewhat similar phenomenon is still true, but one has to allow the curves to have singularities. This observation is based on the following theorem by J. Becker.

\begin{thm}
Let $M$ be a reduced complex analytic space. Then for every $p\in M$ there exists a reduced complex analytic curve $C\subseteq U\subseteq M$ such that $T_pC=T_pM$. Here $U\subseteq M$ is n open neighbourhood of $p\in M$.
\label{CurveSpan}
\end{thm}

\begin{proof}
See~\cite[page 394, Theorem 1]{Becker}
\end{proof}

One can utilize the preceding statement to take a curve spanning the tangent space of the graph of a continuous weakly holomorphic function and test holomorphicity simply on the projection of this curve. As the tangent space of a singular space still contains sufficient information to test for a morphism to be an immersion, however in general it is not enough to test for a local isomorphism. But a closed, bijective immersion is necessarily biholomorphic. The precise argument is presented in the proof of the following Proposition.

\begin{prop}
The continuous weakly holomorphic frames locally obtained in Proposition~\ref{AbsWeak} are stongly holomorphic on any reduced complex analytic space if and only if they are strongly holomorphic on any reduced complex analytic curve.
\label{RedToCurveGen}
\end{prop}

\begin{proof}
The direction $\implies$ is a tautology.\\
Consider now the direction $\impliedby$: Continue the notation from Proposition~\ref{AbsWeak} and assume that $M$ has been shrunk such that the frames in question exist (this is valid as the question is local). Let $\pi_i\colon M_i\to M$ be the analytic projection from the graph of $t_i$. The morphism $\pi_i$ is biholomorphic away from $\mathrm{sing}\left( M \right)$, however, it is a homeomorphism everywhere. Let $p\in \pi_i^{-1}\left( \mathrm{sing}\left( M \right) \right)$ and assume there is a curve $C_i\subseteq U_i \subseteq M_i$ such that $T_pC_i=T_pM_i$ (this is possible by Theorem~\ref{CurveSpan}). One can moreover assume that $C_i=\bigcup_{j=1}^nC^i_j$ such that this union is the decomposition of $C_i$ into locally irreducible and connected components such that each $C^i_j$ contains $p$.\\
As the morphism $\pi_i$ is proper and $1$-to-$1$, it follows that $C_i':=\pi_i\left( C_i \right)\subseteq \pi_i\left( U_i \right)\subseteq M$ is a complex analytic curve in $M$. Moreover, $C_i$ is the graph of $t_i$ restricted to $C'_i$. The claim now is that this restriction is strongly holomorphic by assumption. To verify this claim, one only has to observe that the restriction is once again continuous weakly holomorphic on $C'_i$ and parallel away from $\mathrm{sing}\left( C'_i \right)$. Let $k_{ij}\in \bb{N}$ be the integers such that the image of
\[\rst{\pi_i}{C^i_j}\colon C^i_j \to \pi_i\left( U_i \right)\]
is contained in $\mathrm{sing}^{k_{ij}}\left( U_i \right)$ but not in $\mathrm{sing}^{k_{ij}+1}\left( U_i \right)$. This implies that
\[\pi_i\left( C_j^i \right)\cap \mathrm{sing}^{k_{ij}+1}\left( U_i \right)\in \left\{ \emptyset,\left\{ \mathrm{pt.} \right\} \right\},\]
as the $C_j^i$ are $1$-dimensional. However, recall that the frames $t_i$ were such that they are also parallel on $\mathrm{sing}^{k_{ij}}\left( U_i\right)_{\mathrm{reg}}$ for any $k_{ij}$. Hence, the restriction $\rst{t_i}{\pi_i\left( C_j^i \right)}$ is parallel on a dense subset and hence parallel away from the singularities. Therefore also $\rst{t_i}{C'_i}$ is parallel on a dense subset and hence parallel away from the singularities, as $C'_i = \bigcup_{j=1}^n \pi_i\left( C^i_j \right)$. However, by assumption, such frames on curves are strongly holomorphic and thus the restriction of $\pi_i$
\[\pi'_i\colon C_i \to C_i'\]
is biholomorphic. This implies the Jacobian
\[T_p\pi'_i\colon T_pC_i =T_pM_i\to T_pC'_i\subseteq T_pM_i\]
of $\pi'_i$ is bijective and therefore the Jacobian $T_p\pi_i$ is definitely injective. By~\cite[page 79, Proposition 2.4]{fischer} one has that $\pi_i$ is an immersion, however, $\pi_i$ is also a homeomorphism and thus $\pi_i$ is biholomorphic (\cite[page 3, Lemma 0.23]{fischer}). Therefore, the frames $t_i$ are strongly holomorphic.
\end{proof}

Proposition~\ref{RedToCurveGen} and Theorem~\ref{RelToAbs} now immediately imply the following reduction to the case of complex analytic curves.

\begin{thm}
Let $f\colon X\to N$ be a reduced locally trivial morphism of complex spaces and assume the continuous weakly holomorphic parallel frames obtained in Proposition~\ref{AbsWeak} are strongly holomorphic for all reduced complex analytic curves. Then there is a one-to-one correspondence between
\begin{enumerate}[(i)]
\item pairs $\left( \mc{F},\nabla_f \right)$ of torsion-free coherent sheaves $\mc{F}$ with tame $f$-relative flat connections $\nabla_f$, and,
\item torsion-free $f$-relative local systems $V$.
\end{enumerate}
The correspondence sends pairs $\left( \mc{F},\nabla_f \right)$ to the sheaf $\ker{\nabla_f}$ and sends $f$-relative local systems $V$ to $\left( V\otimes_{f^{-1}\coan{N}} \coan{X}, \nabla^V \right)$.
\label{RedToCurveGenCor}
\end{thm}

Now, one can fine-tune the preceding procedure to a special case, namely that of homogeneous complex analytic spaces.

\begin{defn}
Let $M\subseteq \bb{C}^n$ be a reduced complex analytic subspace. Then $M$ is called a \emph{homogeneous complex analytic space} if $p\in M$ implies that $\lambda\cdot p \in M$ for every $\lambda \in \bb{C}$.
\end{defn}

\begin{rem}
Note that a homogeneous complex analytic space of dimension $1$ is just the union of finitely many straight complex lines.
\end{rem}

In order to specialize this particular strategy to the homogeneous case, one first needs to recall how Theorem~\ref{CurveSpan} is proven. Namely, one obtains it from induction by this following statement about hypersurfaces spanning the tangent space.

\begin{lem}
Let $M$ be a reduced complex analytic space of dimension $r\geq 2$. Write $M=M'\cup M''$ where $\mathrm{dim}\left( M' \right)\leq r-1$ and $M''$ is pure $r$-dimensional and let $p\in M''$. Suppose that $X_i$ are countably many codimension $1$ reduced complex analytic subspaces of $M''$ such that for all $i\in \bb{N}$ one has $p\in X_i$ and
\[\bigcup_{i\in \bb{N}}X_i\subseteq M''\]
is dense in $M''$. Then there exists $k\in \bb{N}$ such that
\[T_p\left( \bigcup_{i=1}^kX_i\cup M' \right)=T_pM.\]
\label{HyperSpan}
\end{lem}

\begin{proof}
See e.g.~\cite[page 396f]{Becker}.
\end{proof}

With this Lemma one can now specialize Proposition~\ref{RedToCurveGen} to homogeneous spaces and homogeneous curvees.

\begin{prop}
The continuous weakly holomorphic frames locally obtained in Proposition~\ref{AbsWeak} are strongly holomorphic on any homogeneous complex analytic space if and only if they are strongly holomorphic on any homogeneous complex analytic curve.
\label{RedCurve}
\end{prop}

\begin{proof}
The statement follows just like Proposition~\ref{RedToCurveGen} but one has to argue why the curves can be assumed to be homogeneous. The argument for this is presented by induction on dimension and by first arguing that the curve case yields holomorphicity around the origin.\\
Observe that if $M\subseteq \bb{C}^n$ is a homogeneous complex analytic space and one writes $M=M'\bigcup M''$ where $\mathrm{dim}\left( M' \right)\leq r-1$ and $M''$ is pure $r$-dimensional, then $M'$ and $M''$ are both homogeneous in $\bb{C}^n$.\\
Note that if $M\subseteq \bb{C}^n$ is a homogeneous complex analytic space, then one may assume that there does not exist a proper linear subspace $V\subset \bb{C}^n$ such that $M\subseteq V$, because if there is such a subspace then one may consider $M\subseteq V$ as the local homogeneous model. By iterating this one may assume that there is no such subspace, since any such subspace necessarily needs to have dimension greater equal to the embedding dimension of $M$. In such a local model one observes that this implies that for any $n-1$ dimensional linear subspace $W\subseteq \bb{C}^n$ the intersection
\[W\cap M''\subseteq W \cong \bb{C}^{n-1} \subseteq \bb{C}^n\]
is a homogeneous complex analytic space of dimension $r-1$ by~\cite[p. 102]{grauert}, as $M''$ and $W$ have non-empty intersection at $0\in \bb{C}^n$.\\
In particular, one may take $\left\{ W_k \right\}_{k\in \bb{N}}$ a countable family of linear $n-1$-dimensional subspaces of $\bb{C}^n$ such that $\bigcup_{k\in \bb{N}}W_k$ is dense in $\bb{C}^n$. Then the intersection $X_k:= W_k\cap M''$ is a countable family of $r-1$-dimensional homogeneous subspaces in $M''$ with dense union.\\
Recall the definitions in the proof of Proposition~\ref{RedToCurveGen} and let $\pi_i\colon M_i\to M$ be the analytic projection from the graph of the continuous weakly holomorphic section $t_i$ defined around $0\in M\subseteq \bb{C}^n$. Note that $M_i=\pi_i^{-1}\left( M' \right)\cup \pi_i^{-1}\left( M'' \right)$ with $\pi_i^{-1}\left( M' \right)$ of dimension smaller than $r$ and $\pi_i^{-1}\left( M'' \right)$ of pure dimension $r$ as $\pi_i$ is bijective. Moreover, the family of codimension $1$ subsets $\pi_i^{-1}\left( X_k \right)$ is dense in $\pi_i^{-1}\left( M'' \right)$ and hence by Lemma~\ref{HyperSpan} it follows that
\[T_{\pi_i^{-1}\left( 0 \right)}\left( \pi_i^{-1}\left( M' \right)\cup \bigcup_{k=1}^l \pi_i^{-1}\left( X_k \right) \right) = T_{\pi_i^{-1}\left( 0 \right)}M_i.\]
By the same argument as in the proof of Proposition~\ref{RedToCurveGen} it therefore suffices to verify the holomorphicity of $t_i$ on $M'\cup \bigcup_{k=1}^l X_k$, which is homogeneous of dimension $\leq r-1$, to obtain the holomorphicity of $t_i$ around $0\in M\subseteq \bb{C}^n$.\\
By iteration it therefore suffices to verify the holomorphicity of $t_i$ on homogeneous curves to obtain the holomorphicity of $t_i$ around $0\in M\subseteq \bb{C}^n$.\\
Now, it remains to argue that the frames $t'_i$ around other points $p\in \mathrm{sing}\left( M \right)$ are holomorphic if the frames $t_i$ are holomorphic around $0$. To this end, consider the blow-up $\psi'\colon\mathrm{Bl}_0\left( \bb{C}^n \right) \to \bb{C}^n$ of $\bb{C}^n$ at the origin and recall that the projection $f'\colon \mathrm{Bl}_0\left( \bb{C}^n \right)\to \bb{CP}^{n-1}$ is the tautological line bundle of $\bb{CP}^{n-1}$. In particular, it is a locally trivial bundle. Moreover, denote by $\psi\colon \tilde{M}\to M$ the restriction of $\psi'$ to $M$ and
\[\tilde{M}:=\mathrm{Closure}\left( \psi'^{-1}\left( M \right)\setminus \psi'^{-1}\left( 0 \right) \right)\subseteq \mathrm{Bl}_0\left( \bb{C}^n \right),\]
i.e. $\tilde{M}$ is the blow-up of $M$ at $0$. The image of $\tilde{M}$ under $f'$ is the projectivization $\bb{P}\left( M \right)$ of $M$ and is in particular an analytic subset. Denote the restriction of $f'$ by $f\colon \tilde{M}\to \bb{P}\left( M \right)$ and notice that $f$ remains a locally trivial line bundle.\\
Let $U\subseteq \bb{P}\left( M \right)$ be an open neighbourhood around $f\left( p \right)$ such that $f$ is equivalent to $U\times \bb{C}\to U$ over $U$. Notice that the connection $\psi^*\nabla$ on $\tilde{M}$ naturally defines a $f$-relative connection $\nabla_f$ simply by applying the quotient map. However, by Theorem~\ref{nonrelativetorelative} and the simply connectedness of $\bb{C}$, it follows that $\nabla_f$ is trivial on $U\times \bb{C}$ and that there exists a global $f$-relative parallel frame $s_i$ for $\nabla_f$ over $U\times \bb{C}$. But $\psi^*t_i$ is also parallel with respect to $\nabla_f$, as it is parallel with respect to $\nabla$. As such, there exist holomorphic functions $a_{ij}\in \coan{U}$ such that
\[\psi^*t_i = \sum_{j=1}^k \left(a_{ij}\circ f\right)\cdot s_j.\]
The right hand-side extends to $f^{-1}\left( U \right)$ and, therefore, so does the left hand-side. This implies that the frames $t_i$ around $0$ naturally extend to an open neighbourhood $W'$ of $0$ also containing $p$. This neighbourhood can be assumed to be connected and because the $t_i$ are parallel on $W'_{\mathrm{reg}}$ near $0$ they are parallel on $W'_{\mathrm{reg}}$. The extension of the frame $t_i$ constructed in this way is holomorphic on $W'$ if it is holomorphic around $0$ and as both frames $t_i$ and $t'_i$ are parallel on the regular part of an open neighbourhood of $p$, it follows that also $t'_i$ is holomorphic since these parallel frames only differ by a constant matrix.
\end{proof}

By once again employing the blow-up in a simple way, one obtains that the absolute case on homogeneous curves holds. The idea is that homogeneous curves are simply unions of finitely many straight lines through the origin. The blow-up at the origin then neatly arranges such lines next to each other as a fiber bundle over $\bb{CP}^{n-1}$. The connection on the homogeneous curve that one started with can now be viewed as connections along some fibers of this fiber bundle and one constructs a flat relative connection on this fiber bundle that restricts to these connections on the relevant fibers. Flat relative connections on fiber bundles with smooth fibers are however where the Riemann-Hilbert correspondence is already known to hold and one can leverage that fact to obtain the case of homogeneous curves.

\begin{prop}
Let $M\subseteq \bb{C}^n$ be a homogeneous complex analytic subspace of dimension $1$ with $n\geq 2$. Then the continuous weakly holomorphic frames obtained in Proposition~\ref{AbsWeak} are strongly holomorphic.
\label{HomCurve}
\end{prop}

\begin{proof}
Note that $M$ is simply a union of linear complex lines in $\bb{C}^n$ meeting in the origin. The continuous weakly holomorphic frame from Proposition~\ref{AbsWeak} therefore exists globally by construction and the only singular point of $M$ is the origin $0\in \bb{C}^n$. Recall that any locally free sheaf on $M$ is free as $M$ is Stein and contractible, because topological classification of fiber bundles is the same as the holomorphic classification on Stein spaces (see~\cite[Satz 6]{SteinClass}). Observe that the connection $\nabla$ can therefore be extended to
\[\nabla'\colon \coan{\bb{C}^n}^{\oplus r}\to \coan{\bb{C}^n}^{\oplus r}\otimes \Omega^1\left( \coan{\bb{C}^n} \right),\]
as $\bb{C}^n$ is also Stein. Denote by $\psi\colon \mathrm{Bl}_0\left( \bb{C}^n \right)\to \bb{C}^n$ the blow-up of $\bb{C}^n$ at the origin $0$ and by $f\colon \mathrm{Bl}_0\left( \bb{C}^n \right)\to \bb{CP}^{n-1}$ the locally trivial tautological line bundle. Then $\psi^*\nabla'$ naturally defines a $f$-relative connection $\nabla_f$ by simply applying the quotient map. Note that $\nabla_f$, restricted to the fibers over $\bb{P}\left( M \right)$ without the exceptional set, is simply $\nabla$ restricted to $M_{\mathrm{reg}}$ (as $M$ is just a union of straight complex lines and the fibers of $f$ are straight complex lines). By Theorem~\ref{nonrelativetorelative} and $\bb{C}$ being simply connected one obtains that the evaluation morphism
\[e\colon f_*\ker{\nabla_f} \to s^{-1}\ker{\nabla_f}\cong s^* \coan{\mathrm{Bl}_0\left( \bb{C}^n \right)}, t\mapsto s^{-1}t\]
is locally on $\bb{CP}^{n-1}$ an isomorphism and therefore globally. Here, $s\colon \bb{CP}^{n-1}\to \mathrm{Bl}_0\left( \bb{C}^n \right)$ is the zero section.\\
There exists a global $f$-relative parallel frame $s'_i$ on $\mathrm{Bl}_0\left( \bb{C}^n \right)$ with respect to $\nabla_f$. The morphism $\psi$ is an isomorphism away from $0\in \bb{C}^n$ and as $\left\{ 0 \right\}\subseteq \bb{C}^n$ has codimension greater equal $2$ it follows that the sections $s'_i$ are pull-backs of holomorphic $s_i$ on $\bb{C}^n$.\\
Moreover, restricted to $M$ the sections $\rst{s_i}{M}$ are parallel with respect to $\nabla$ on $M_{\mathrm{reg}}$. As these parallel frames only differ by a constant matrix it follows that the frames $t_i$ constructed in Proposition~\ref{AbsWeak} are also strongly holomorphic on $M$.
\end{proof}

Combining Proposition~\ref{HomCurve} with Proposition~\ref{RedCurve} and Theorem~\ref{RelToAbs} yields the relative Riemann-Hilbert correspondence for homogeneous fibers and torsion-free sheaves.

\begin{thm}
Let $f\colon N\times M\to N$ be the projection of reduced complex spaces and assume that $M\subseteq \bb{C}^n$ is a homogeneous complex analytic space. Then there is a one-to-one correspondence between
\begin{enumerate}[(i)]
\item pairs $\left( \mc{F},\nabla_f \right)$ of torsion-free coherent sheaves $\mc{F}$ with tame $f$-relative flat connections $\nabla_f$, and,
\item torsion-free $f$-relative local systems $V$.
\end{enumerate}
The correspondence sends pairs $\left( \mc{F},\nabla_f \right)$ to the sheaf $\ker{\nabla_f}$ and sends $f$-relative local systems $V$ to $\left( V\otimes_{f^{-1}\coan{N}} \coan{X}, \nabla^V \right)$.
\label{HomCor}
\end{thm}

This concludes the discussion on relative Riemann-Hilbert theorems for morphisms with singular fibers. The next section establishes how one can connect the previous discussion to the study of real analytic pseudo-holomorphic structures and obtain Newlander-Nirenberg type Theorems in this way.

\section{Newlander-Nirenberg Theorem and $\bar{\partial}$-operators}
\label{trans}

On vector bundles over complex manifolds it is a standard fact, that a holomorphic structure is equivalent to an integrable generalised $\bar{\partial}$-operator (see~\cite{malgrange}).\footnotemark{}\footnotetext{Similar theorems also hold for more general sheaves of modules over smooth functions over complex manifolds, see~\cite{pali} and~\cite{weinkove}.}
 One already knows, by Section~\ref{pq}, that $\left( 0,1 \right)$-forms can be thought of as relative differential forms on the complexification. One can use this to show that real analytic pseudo-holomorphic structures (generalised $\bar{\partial}$-operators) can be seen as relative connections on the complexification with respect to the canonical fibration.\\
This allows one to solve real analytic Newlander-Nirenberg Theorems on sheaves over complex analytic spaces by solving the corresponding Relative Riemann-Hilbert Theorems on the canonical fibration of the complexification.

\begin{defn}
Let $M$ be a complex analytic space and $\mc{F}$ a $\corean{M}$-module. Then an $\coan{M}$-linear morphism
\[\bar{\partial}_{\mc{F}}\colon \mc{F}\to \mc{F}\otimes_{\corean{M}} \Omega^{0,1}M\]
such that
\[\bar{\partial}_{\mc{F}}\left( f\cdot s \right) = \bar{\partial}f\otimes s + f \cdot \bar{\partial}s,\]
for all $f\in \corean{M}$ and $s\in \mc{F}$, is called \emph{pseudo-holomorphic structure}.\\
In other words, $\bar{\partial}_{\mc{F}}$ is a relative connection with respect to the canonical morphism $\left( M,\corean{M} \right)\to \left( M,\coan{M} \right)$.\\
The pseudo-holomorphic structure is called \emph{flat} or \emph{integrable} if it is flat as a relative connection.
\end{defn}

For the upcoming proof it is necessary to extend the definiton of derivations to incorporate more general sheaves as their domain.

\begin{defn}
Let $M$ be a $\bb{K}$-ringed space and let $\mc{F}$ and $\mc{G}$ be two $\mc{A}_M$-modules. Then a \emph{first-order differential operator from $\mc{G}$ to $\mc{F}$} is a $\bb{K}$-linear morphism $\delta \colon \mc{G}\to \mc{F}$ such that the morphism
\[ \left[ \delta,f \right]\colon\rst{\mc{G}}{U} \to \rst{\mc{F}}{U},\; s \mapsto \delta\left( f\cdot s \right) - f \delta\left( s \right)\]
is $\mc{A}_M$-linear for every $f\in \mc{A}_M\left( U \right)$ and $U\subseteq M$ open. Denote the sheaf of such operators by $\mathrm{Diff}^{\left( 1 \right)}\left( \mc{G},\mc{F} \right)$.
\end{defn}

\begin{rem}
Note that this clearly extends the definition of a derivation and is a sensible notion of first-order differential operators between general sheaves. Also, connections over analytic spaces are evidently first-order differential operators.
\end{rem}

To begin with, it is investigated how first-order differential operators and derivations are related.

\begin{lem}
Let $M$ be a $\bb{K}$-analytic space and $\mc{F}$ be an $\mc{A}_M$-module. Then the following statements hold:
\begin{enumerate}[(i)]
\item $\mathrm{Diff}^{\left( 1 \right)}\left( \mc{A}_M,\mc{F} \right)\cong \homo{\mc{A}_M,\mc{F}}\oplus \mathrm{Der}\left( \mc{A}_M,\mc{F} \right)$.
\item $\mathrm{Diff}^{\left( 1 \right)}\left( \mc{A}_M^{\oplus n}, \mc{F} \right) \cong \mathrm{Diff}^{\left( 1 \right)}\left({\mc{A}_M,\mc{F}}\right)^{\oplus n}$.
\end{enumerate}
\label{DiffSplit}
\end{lem}

\begin{proof}
\begin{enumerate}[(i)]
\item First of all, note that both, $\homo{\mc{A}_M,\mc{F}}$ and $\mathrm{Der}\left( \mc{A}_M,\mc{F} \right)$, are natural subsheaves of the first-order differential operators. Let $\delta \in \mathrm{Diff}^{\left( 1 \right)}\left( \mc{A}_M,\mc{F} \right)$. Consider the following morphism
\[\delta'\colon \mc{A}_M \to \mc{F},\; f\mapsto \delta\left( f \right)-f\cdot \delta\left( 1 \right)\]
and note that it is clearly a first-order differential operator. Hence,
\[ \left[ \delta',g \right]\left( f \right) = f\cdot\left[ \delta',g \right]\left( 1 \right),\]
but expanding this out yields:
\begin{align*}
\left[ \delta',g \right]\left( f \right) &= \delta'\left( gf \right) - g\delta'\left( f \right)\\
f\left[ \delta',g \right]\left( 1 \right) &= f\delta'\left( g \right)-fg\delta'\left( 1 \right) =f\delta'\left( g \right) -fg \delta\left( 1 \right) + fg\delta\left( 1 \right)\\
&= f\delta'\left( g \right).
\end{align*}
Showing that $\delta'$ is a derivation. Moreover, the morphism
\[A\colon \mathrm{Diff}^{\left( 1 \right)}\left( \mc{A}_M,\mc{F} \right) \to \mathrm{Der}\left( \mc{A}_M,\mc{F} \right),\; \delta \mapsto \delta'\]
is such that $A\left( \delta \right)=\delta$ for $\delta\in \mathrm{Der}\left( \mc{A}_M,\mc{F} \right)$ and $A\left( \delta \right)=0$ if and only if $\delta \in \homo{\mc{A}_M,\mc{F}}$. Hence, $A$ induces the desired splitting.
\item Let $\delta \in \mathrm{Diff}^{\left( 1 \right)}\left( \mc{A}_{M}^{\oplus n}, \mc{F} \right)$ and then consider the morphism
\[B\colon  \mathrm{Diff}^{\left( 1 \right)}\left( \mc{A}_{M}^{\oplus n}, \mc{F} \right) \to  \mathrm{Diff}^{\left( 1 \right)}\left( \mc{A}_{M}, \mc{F} \right)^{\oplus n},\; \delta \mapsto \left( \delta\circ \iota_1,\dots,\delta\circ\iota_n \right),\]
where $\iota_i\colon \mc{A}_M\to \mc{A}_M^{\oplus n}$ denotes the embedding of the $i$-th component. Conversely, consider the map
\[C\colon  \mathrm{Diff}^{\left( 1 \right)}\left( \mc{A}_{M}, \mc{F} \right)^{\oplus n} \to  \mathrm{Diff}^{\left( 1 \right)}\left( \mc{A}_{M}^{\oplus n}, \mc{F} \right),\; \left( \delta_1,\dots,\delta_n \right)\mapsto \sum_{i=1}^n \delta_i.\]
As first-order differential operators are in particular $\bb{K}$-linear maps, it follows that $B\circ C = \id$ and $C\circ B =\id$. Hence, $B$ and $C$ yield the desired isomorphism.
\end{enumerate}
\end{proof}

The preceding lemma allows one to realise that on analytic spaces first-order differential operators are determined in an open neighbourhood by their action on the stalks. This mirrors similar results for homomorphisms and derivations. 

\begin{lem}
Let $A\colon \mc{A}_M^{\oplus n}\to \mc{G}$ be an epimorphism of coherent $\mc{A}_M$-modules where $M$ is an analytic space and let $\mc{H}$ be another finitely generated $\mc{A}_M$-module. Then the induced morphism
\[\mathrm{Diff}^{\left( 1 \right)}\left({\mc{G},\mc{H}}\right)\to \mathrm{Diff}^{\left( 1 \right)}\left({\mc{A}_M^{\oplus n},\mc{H}}\right),\; \delta \mapsto \delta \circ A\]
is injective.\\
Also, the evaluation morphism to stalks is an isomorphism in this case, i.e. the canonical morphism
\[\mathrm{Diff}^{\left( 1 \right)}\left( \mc{G},\mc{H} \right)_p \to \mathrm{Diff}^{\left( 1 \right)}\left( \mc{G}_p, \mc{H}_p \right)\]
is an isomorphism for every $p\in M$.
\label{InjDerv}
\end{lem}

\begin{proof}
Consider $\delta \in \mathrm{Diff}^{\left( 1 \right)}\left({\mc{G},\mc{H}}\right)$ and suppose that $\delta \circ A = 0$. This means that if $U\subseteq M$ is an open Stein set in $M$, then $A\left( U \right)$ is surjective and hence $\delta\left( U \right)=0$. As Stein sets form a basis of the topology of $M$ it follows that $\delta =0$ and hence the induced morphism is injective.\\
For the second claim, note that one has the following commutative diagram for every $p\in M$:
\[
\begin{tikzcd}
\mathrm{Diff}^{\left( 1 \right)}\left( \mc{G},\mc{H} \right)_p \arrow[d,"a"] \arrow[r,hook,"b"] & \mathrm{Diff}^{\left( 1 \right)}\left( \mc{A}_M^{\oplus n}, \mc{H} \right)_p \arrow[d,hook,two heads,"c"] \\
\mathrm{Diff}^{\left( 1 \right)}\left( \mc{G}_p,\mc{H}_p \right) \arrow[r,hook,"d"] & \mathrm{Diff}^{\left( 1 \right)}\left( \mc{A}_{M,p}^{\oplus n}, \mc{H}_p \right)
\end{tikzcd}.
\]
The vertical arrow on the right is an isomorphism by Lemma~\ref{DiffSplit} and Proposition~\ref{dervgermsgermsderv}. As the diagram is commutative it follows that the vertical arrow on the left is a monomorphism.\\
Now, let $\delta^p\in \mathrm{Diff}^{\left( 1 \right)}\left( \mc{G}_p,\mc{H}_p \right)$ and consider $\tilde{\delta}^p:= c^{-1}\left( d\left( \delta^p \right) \right)\in\mathrm{Diff}^{\left( 1 \right)}\left( \mc{A}_M^{\oplus n}, \mc{H} \right)_p$ and denote by $\tilde{\delta}\in \mathrm{Diff}^{\left( 1 \right)}\left( \mc{A}_M^{\oplus n}, \mc{H} \right)$ a local representative. One may of course restrict $\tilde{\delta}$ to $\ker{A}$ and with that one obtains a first-order differential operator
\[\delta':= \rst{\tilde{\delta}}{\ker{A}}\colon \ker{A}\to \mc{H}.\]
Hence, for every $f\in \mc{A}_M$ one has that
\[ \left[ \delta',f \right]\colon \ker{A} \to \mc{H}\]
is a linear homomorphism. However, for all $s\in \ker{A}(U)$, where $p\in U$, one has
\begin{align*}
\left[ \delta',f \right]\left( s \right)_p&=\delta'\left( f\cdot s \right)_p - f\cdot \delta'\left( s \right)_p\\
&= \tilde{\delta}^p\left( f\cdot s \right) - f\cdot\tilde{\delta}^p\left( s \right)\\
&= \delta^p\left( A\left( f\cdot s \right)_p \right) - f\cdot\delta^p\left( A\left( s \right)_p \right)=0.
\end{align*}
This shows that $\left[ \delta',f \right]$ is zero in an open neighbourhood $U_f$ of $p$ for every $f\in \mc{A}_M$. In particular, one may choose a local model of $\iota\colon M \to W$ around $p$ and denote by $x_1,\dots,x_k$ the standard coordinates on $W$. Note that the equivalence classes $g_1^q:=\left[ x_1-x_1\left( q \right) \right],\dots,g_k^q:=\left[x_k-x_k\left( q \right)\right]$, for $q\in M$, generate the maximal ideal $m_{M,q}$ of $M$ at $q$. Consider then the open subset $V:=\bigcap_{i=1}^k U_{g_i^p}$ and note that for every $s\in \ker{A}(U)$, where $ U\subseteq V$ is open, one has
\[\delta'\left( g^p_i\cdot s \right)=g^p_i\delta'\left( s \right)\implies \delta'\left( g_i^q\cdot s \right)=g_i^q\delta'\left( s \right),\]
since $g_i^q=g_i^p + x_i\left( p \right)-x_i^q\left( q \right)$. Let $s_i$ be generators of $\ker{A}$ on $V$ by shrinking the open neighbourhood of $p$, if necessary, and note that
\[\delta'\left( s_i \right)_p=0\]
as before. However, then one may assume that
\[\delta'\left( s_i \right)=0\]
holds in the entire open neighbourhood $V$ by once again shrinking $V$, if necessary. Now, for every $q\in V$ consider 
\begin{align*}
\delta'\left( \sum_i f_i \cdot s_i \right)&=\sum_i \delta'\left( \left(f_i-f_i\left( q \right) \right)\cdot s_i\right) + \sum_i f_i\left( q \right)\delta'\left( s_i \right)\\
&=\sum_i\delta'\left( \sum_j g^q_jf^{1,j}_i \cdot s_i \right)\\
&= \sum_{i,j} g^q_j\delta'\left( f_i^{1,j}\cdot s_i \right),
\end{align*}
where $\sum_j g^q_jf^{1,j}_i=f_i-f_i\left( q \right)$. By iterating this argument one sees that sections of the form
\[\delta'\left( \sum_i f_i \cdot s_i \right)\]
have a germ at $q$ that is contained in $\bigcap_i m_{M,q}^i\cdot \mc{H}$. By Krull intersection this implies that such sections are zero. As the sections $s_i$ are generators of $\ker{A}$ on $V$ it follows that $\delta'$ is just zero on $V$.\\
Recall, that $\delta'$ was just $\tilde{\delta}$ restricted to $\ker{A}$ and therefore $\tilde{\delta}$ vanishes on $\ker{A}$ in an open neighbourhood $p$. Thus, one can pass $\tilde{\delta}$ to the quotient which is $\mc{G}$. Hence, one obtains a differential operator
\[\delta\colon \mc{G}\to \mc{H}\]
such that $b\left( \delta_p \right)= \tilde{\delta}= c^{-1}\left( d\left( \delta^p \right) \right)$. Hence,
\[d\left( a\left( \delta_p \right) \right)=c\left( b\left( \delta_p \right) \right)= d\left( \delta^p \right)\]
and since $d$ is a monomorphism, it follows that $a\left( \delta_p \right)=\delta^p$.
\end{proof}

\begin{rem}
It is clear that if $\delta\colon \mc{F} \to \mc{G}$ is a $\bb{C}$-linear first-order differential operator between coherent $\corean{M}$-sheaves on a real analytic space $M$, then $\delta$ is the same as viewing it as a first-order differential operator between coherent $\rean{M}$-sheaves that is also $\bb{C}$-linear.
\end{rem}

With Lemma~\ref{InjDerv}, one can now show that pseudo-holomorphic structures can be realised as topological restrictions of relative connections on the complexification.

\begin{thm}
Let $\left( M,\coan{M} \right)$ be a $\bb{C}$-analytic space and let $\Phi\colon M^{\bb{C}}\to M$ be the canonical fibration. Suppose that $\mc{F}$ is a coherent $\corean{M}$-module and let $\nabla\colon \mc{F}\to \mc{F}\otimes_{\corean{M}} \Omega^{0,1}M$ be a pseudo-holomorphic structure. Then there exists a coherent $\coan{M^{\bb{C}}}$-module $\tilde{\mc{F}}$ and a relative connection 
\[\nabla_{\Phi}\colon \tilde{\mc{F}}\to \tilde{\mc{F}}\otimes_{\coan{M^{\bb{C}}}} \Omega^1_{\Phi}\left( \coan{M^{\bb{C}}} \right)\]
such that $\iota^{-1}\nabla_{\Phi}= \nabla$, where $\iota\colon M \to M^{\bb{C}}$ denotes the inclusion, after potentially shrinking $M^{\bb{C}}$.\\
If $\nabla$ is integrable, then, after potentially shrinking $M^{\bb{C}}$, the relative connection $\nabla_{\Phi}$ may be assumed to be integrable as well.
\label{pseudocomplex}
\end{thm}

\begin{proof}
The existence of $\tilde{\mc{F}}$ is given by e.g.~\cite[I.2.8]{realan} and by Proposition~\ref{Rel10Dif} it follows that $\nabla$ is equivalent to a first-order differential operator
\[\nabla\colon\iota^{-1} \tilde{F} \to \iota^{-1}\left( \tilde{\mc{F}}\otimes_{\coan{M^{\bb{C}}}} \Omega^1_{\phi}\left( \coan{M^{\bb{C}}} \right) \right).\]
By Lemma~\ref{InjDerv} one obtains for every $p\in M$ an open neighbourhood $U^p\subseteq M^{\bb{C}}$ of $p$ and a first-order differential operator
\[\nabla^p\colon \rst{\tilde{\mc{F}}}{U^p}\to \rst{ \tilde{\mc{F}} }{U^p}\otimes_{\coan{U^p}} \Omega^1_{\phi}\left( \coan{U^p} \right) \]
such that $\nabla^p_p=\nabla_p$. Again by Lemma~\ref{InjDerv} one has $\iota^{-1}\nabla^p=\rst{\nabla}{U_p\cap M}$, after potentially shrinking $U^p$ around $p$. By paracompactness of $M^{\bb{C}}$ one may assume that $\left\{ U^{p_i} \right\}_{i\in I}$ is a locally finite subcover of the $U^p$'s and one can consider the set
\[W:= \left\{ q\in M^{\bb{C}}\mid q\in U^{p_i}\cap U^{p_j} \implies \nabla^{p_i}_q=\nabla^{p_j}_q \right\}.\]
This set is non-empty as $M\subseteq W$ and it is open because the covering was assumed to be locally finite and one can apply Lemma~\ref{InjDerv}. Over this open subset $W$ the first-order differential operators all germ-wise agree and thus glue to a first-order differential operator $\nabla_{\Phi}$ such that $\iota^{-1}\nabla_{\Phi}=\nabla $.\\
It remains to be shown that this differential operator is a $\Phi$-relative connection. To this end, let $q\in M$ and let $A\colon \coan{V}^{\oplus k} \to \rst{\tilde{\mc{F}}}{V}$ be an epimorphism, where $V\subseteq M^{\bb{C}}$ is a Stein open neighbourhood of $q$. Define a $\Phi$-relative connection $\nabla'$ on $\coan{V}^{\oplus k}$ such that
\[A\otimes \id \left( \nabla'\left(e_i  \right) \right) = \nabla_{\Phi}\left( A\left( e_i \right) \right),\]
here $e_i$ denotes the $i$-th basis element, and then extending by imposing the $d_{\Phi}$-Leibniz rule for connections. Note that the previous construction is of course possible because the domain of definition of $A$ is a free module and $V$ is a Stein open subset and therefore $A\otimes\id\left( V \right)$ is surjective. Because $\nabla$ is a relative connection (or: pseudo-holomorphic structure) it follows that
\[\iota^{-1}\left( A\otimes \id \circ \nabla' \right)=\nabla \circ \iota^{-1}A.\]
Then, by Lemma~\ref{InjDerv}, one has $A\otimes \id \circ \nabla' =\nabla_{\Phi} \circ A$, after shrinking $V$ if necessary. Now, it is clear that $\nabla_{\Phi}$ satisfies the $d_{\Phi}$-Leibniz rule, because locally for $f\cdot s = f\cdot A\left( t \right)$ one has
\begin{align*}
\nabla_{\Phi}\left( f\cdot s \right) &= \nabla_{\Phi}\left( A\left( f\cdot t \right) \right)\\
&=A\otimes \id\left( \nabla'\left( f\cdot t \right) \right)\\
&= A\otimes\id\left( t\otimes d_{\Phi}f + f\nabla'\left( t \right) \right)= s\otimes d_{\Phi}f + f\nabla_{\Phi}\left( s \right).
\end{align*}
This equality holds locally and thus also globally. Hence, the first part of the statement is shown.\\
Note that since $\iota^{-1}\nabla_{\Phi}=\nabla$ it is clear that $\nabla_{\Phi}$ is flat/integrable around $M\subseteq M^{\bb{C}}$ if and only if $\nabla$ is integrable.
\end{proof}

One has already observed that some \emph{tameness} condition has to be imposed in order to answer the relevant questions in this work. The definition below explains what tame should mean in the context of pseudo-holomorphic structures.

\begin{defn}
Let $M$ be a complex analytic space and $\mc{F}$ a coherent $\corean{M}$-module with a pseudo-holomorphic structure $\bar{\partial}_{\mc{F}}$. Then the operator $\bar{\partial}_{\mc{F}}$ is called \emph{tame} if the extension to a relative connection on the complexification given by Thereom~\ref{pseudocomplex} may be assumed to be tame.
\end{defn}

\begin{rem}
Note that this definition of tameness is not really a satisfactory one, as it seems somewhat artificial to have to complexify a given pseudo-holomorphic structure first and then test for tameness. However, the author is not aware of a satisfactory way around this, in this context. It should be the subject of future research how and when a flat pseudo-holomorphic structure is tame without employing the complexification.
\label{remtame}
\end{rem}

One can then immediately utilize the complexification and the relative connection associated to a pseudo-holomorphic structure to obtain a Newlander-Nirenberg-type Theorem for real analytic tame integrable pseudo-holomorphic structures by means of utilizing Theorem~\ref{MaxRH}.

\begin{thm}
Let $M$ be a maximal complex analytic space and $\left( \mc{F},\bar{\partial}_{\mc{F}} \right)$ a torsion-free coherent $\corean{M}$-module with tame integrable pseudo-holomorphic structure. Then $\mc{G}:=\ker{\bar{\partial}_{\mc{F}}}$ is a torsion-free coherent $\coan{M}$-module and $\mc{F}=\mc{G}\otimes_{\coan{M}}\corean{M}$.\\
In other words, there is a 1-to-1 correspondence between torsion-free coherent $\coan{M}$-modules and torsion-free coherent $\corean{M}$-modules equipped with a tame integrable pseudo-holomorphic structure.
\label{holomorpseudo}
\end{thm}

\begin{proof}
Applying Theorem~\ref{pseudocomplex} puts the question into the situation of Theorem~\ref{MaxRH}. Denote by $\nabla$ the extension of $\bar{\partial}_{\mc{F}}$ to $\mc{H}$ on $M^{\bb{C}}$ where $\mc{H}$ denotes a torsion-free coherent $\coan{M^{\bb{C}}}$-module such that $\iota^{-1}\mc{H}\cong \mc{F}$. Therefore,
\[\mc{F}\cong \iota^{-1}\mc{H} \cong \iota^{-1}\ker{\nabla}\otimes_{\coan{M}} \corean{M}.\]
Now, $\iota^{-1}\ker{\nabla}\subseteq \ker{\bar{\partial}_{\mc{F}}}$. Let $s\in \ker{\bar{\partial}_{\mc{F}}}$. Then there exists $t \in \mc{H}$ such that $\iota^{-1}t=s$. Hence,
\[0=\bar{\partial}_{\mc{F}}\left( s \right) = \iota^{-1}\nabla\left( \iota^{-1}t \right) = \iota^{-1}\left( \nabla t \right).\]
This implies that $\nabla \left( t \right)$ is zero around $M$ in $M^{\bb{C}}$ and hence $\iota^{-1}\ker{\nabla}= \ker{\bar{\partial}_{\mc{F}}}$. Notice here that $\ker{\nabla}$ is a torsion-free coherent $\Phi^{-1}\coan{M}$-sheaf and hence $\ker{\bar{\partial}_{\mc{F}}}$ is a torsion-free coherent $\coan{M}$-sheaf, since $\Phi\circ \iota=\id_M$ as set-theoretic maps.\\
For the statement about the 1-to-1 correspondence, note that $\corean{M}$ is faithfully flat over $\coan{M}$ (as it is a flat ring extension of local rings). This implies that if $\mc{H}$ is a torsion-free coherent $\coan{M}$-module, then $\nabla':=\id\otimes \bar{\partial}$ is an integrable pseudo-holomorphic structure. It is tame as one can view it as $\iota^{-1}\Phi^*\mc{H}$ with relative connection $\iota^{-1}\left( \id\otimes d_{\Phi} \right)$, which is tame. Moreover,
\[\mc{H}\hookrightarrow \ker{\nabla'},\]
but when one tensors this with $\corean{M}$ over $\coan{M}$ the morphism above is an isomorphism. As $\corean{M}$ is faithfully flat over $\coan{M}$, it follows that $\mc{H}=\ker{\nabla'}$.
\end{proof}

\begin{rem}
Notice that the reasoning in the proof of Theorem~\ref{holomorpseudo} does not a priori work for a homogeneous complex analytic space, as it is unclear and in general false that the extended relative connection obtained from the pseudo-holomorphic structure is defined on the entirety of $M\times \bar{M}$. As such one can not easily apply Theorem~\ref{HomCor}.
\end{rem}

\section{Conclusion and further questions}
\label{Conc}

In this paper the classical Relative Riemann-Hilbert correspondence for relative complex analytic flat connections on submersions was extended to the case of torsion-free coherent sheaves over locally trivial morphisms between reduced spaces with maximal and homogeneous fibers (Theorem~\ref{MaxRH} and Theorem~\ref{HomCor}). It was also shown that real analytic $\left( 0,1 \right)$-forms and real analytic pseudo-holomorphic structures on complex analytic spaces have well-behaved interpretations on the complexification (Theorem~\ref{pseudocomplex}). This interpretation of such structures can then be leveraged to prove a Newlander-Nirenberg-type theorem for integrable generalised $\bar{\partial}$-operators on torsion-free sheaves over maximal complex analytic spaces (Theorem~\ref{holomorpseudo}). A couple of open questions remain:
\begin{enumerate}[(1)]
\item How can the methods developed here be adapted to non-reduced image spaces and coherent sheaves with torsion?
\item How can the methods developed here be adapted to deal with more general fibers?
\item What is a more immediate and intuitive characterisation of tame pseudo-holomorphic structures (see Remark~\ref{remtame})?
\item How does this work and the developed methods fit into the more general theory of $\mc{D}$-modules?
\end{enumerate}

\bibliographystyle{unsrt}
\bibliography{bib}
\end{document}

%% file: sh.tex
\DeclareMathOperator{\preimage}{preim}

\DeclareMathOperator{\image}{im}

\DeclareMathOperator{\cokernel}{coker}

\DeclareMathOperator{\kernel}{ker}

\DeclareMathOperator{\homom}{Hom}

\DeclareMathOperator{\id}{id}

\DeclareMathOperator{\derivation}{Der}

\newcommand{\rean}[1]{ {C^{\omega}_{#1}} }

\newcommand{\coan}[1]{ {\mc{O}_{#1}} }

\newcommand{\corean}[1]{ {\bb{C}^{\omega}_{#1}} }

\newcommand{\pard}[2]{{\frac{\partial #1}{\partial #2}}}

\newcommand{\rst}[2]{{% we make the whole thing an ordinary symbol
  \left.\kern-\nulldelimiterspace % automatically resize the bar with \right
  #1 % the function
  \vphantom{\big|} % pretend it's a little taller at normal size
  \right|_{#2} % this is the delimiter
  }}

\newcommand{\quo}[2]{{\left.\raisebox{.2em}{$#1$}\middle/\raisebox{-.2em}{$#2$}\right.}}

\newcommand{\preim}[2]{\preimage_{#1}\left( #2 \right)}

\newcommand{\bb}[1]{\mathbb{#1}}

\newcommand{\mc}[1]{\mathcal{#1}}

\newcommand{\im}[1]{\image\left( #1 \right)}

\renewcommand{\ker}[1]{\kernel\left( #1 \right)}

\newcommand{\coker}[1]{\cokernel\left( #1 \right)}

\newcommand{\homo}[1]{\homom\left( #1 \right)}

\newcommand{\der}[1]{\derivation \left( #1 \right)}